\documentclass{article}
\usepackage[utf8]{inputenc}
\usepackage{amsmath}
\usepackage{amssymb}
\usepackage{fullpage}
\usepackage{enumitem}
\usepackage{tikz}
\usepackage{kantlipsum}
\usepackage{amsthm}
\usepackage{hyperref}
\usepackage{subcaption}

\makeatletter
\renewcommand{\slimits@}{\limits}
\renewcommand{\nmlimits@}{\limits}
\makeatother

\DeclareMathOperator{\dom}{dom}

\DeclareMathOperator{\simm}{\sim}
\newcommand{\simmm}{\simm_m \limits }
\newcommand{\R}{\mathbb{R}}

\newcommand{\lsym}{\lambda_{\text{sym}}}
\newcommand{\lskew}{\lambda_{\text{skew}}}
\renewenvironment{proof}{{\bfseries Proof:}}

\newtheorem{theorem}{Theorem}
\newtheorem{lemma}[theorem]{Lemma}

\newtheorem{corollary}[theorem]{Corollary}
\newtheorem{conjecture}[theorem]{Conjecture}

\numberwithin{equation}{section}
\numberwithin{figure}{section}
\numberwithin{theorem}{section}

\title{Polynomials on the Sierpinski Gasket with Respect to Different Laplacians which are Symmetric and Self-Similar}
\author{Christian Loring, W. Jacob Ogden, Ely Sandine\thanks{Supported by the Summer Program for Undergraduate Research at Cornell University.}, Robert S. Strichartz} 

\date{September 21, 2019}

\begin{document}

\maketitle
\begin{center}
\textbf{Abstract}
\end{center}
We study the analogue of polynomials (solutions to $\Delta^{n+1} u =0$ for some $n$) on the Sierpinski gasket ($SG$) with respect to a family of symmetric, self-similar Laplacians constructed by Fang, King, Lee, and Strichartz, extending the work of Needleman, Strichartz, Teplyaev, and Yung on the polynomials with respect to the standard Kigami Laplacian. We define a basis for the space of polynomials, the monomials, characterized by the property that a certain ``derivative" is 1 at one of the
boundary points, while all other ``derivatives" vanish, and we compute the values of the monomials at the boundary points of $SG$. We then present some data which suggest surprising relationships between the values of the monomials at the boundary and certain Neumann eigenvalues of the family of symmetric self-similar Laplacians. Surprisingly, the results for the general case are quite different from the results for the Kigami Laplacian. \\ \par

\par 

\noindent First published in: Loring Christian, Ogden W., Sandine Ely, Strichartz Robert, \textit{Polynomials on the Sierpinski Gasket with Respect to Different Laplacians which are Symmetric and Self-Similar}. J. Fractal Geom. [forthcoming]. \copyright \ European Mathematical Society.

\section{Introduction} 
\indent The Sierpinski gasket ($SG$) can be thought of as the simplest nontrivial example of a fractal which supports a theory of differential calculus. The study of this theory began with Kigami's analytic construction of a Laplacian on $SG$ (\cite{k1},\cite{k2}) and has expanded to a wide scope encompassing the fractal analogues of many results in standard analysis (\cite{k3}, \cite{str}). With this theory of differential calculus comes a theory of polynomials, functions which are annihilated by taking sufficiently many Laplacians. The polynomials with respect to the standard Kigami Laplacian have been studied in detail in \cite{nsty}. This paper will generalize these results by studying polynomials with respect to a larger class of Laplacians on $SG$ constructed in \cite{fkls}.

\par 
A webpage accompanying this paper, which contains numerical data, graphs, and the programs used to generate them, can be found at \url{http://pi.math.cornell.edu/~reuspurweb/} (\cite{loss}).
\par

This paper begins with a short overview of the construction of the standard Laplacian on $SG$, as presented in \cite{str}, followed by a brief discussion of the polynomials with respect to the standard Laplacian. Here, the coefficients $\alpha_j, \beta_j, \gamma_j$ are defined, and the relationship between $\beta_j$ and the Neumann eigenvalues of the standard Laplacian which was discovered in \cite{nsty} is mentioned. The construction of the standard Laplacian motivates the construction of the one-parameter family of symmetric self-similar Laplacians which is the focus of \cite{fkls}; their construction is summarized here and accompanied by the appropriate definitions of the normal and tangential derivatives which make sense for functions in the domain of a Laplacian in this family. This section concludes with the definition of a polynomial with respect to a Laplacian in the family of Laplacians from \cite{fkls} and the definition of the corresponding monomials, which form a basis for the space of polynomials. \par

The focus of section 2 is the computation of the values of the monomials with respect to a Laplacian in the family of Laplacians at the boundary of $SG$. These values are functions of the parameter $r$ which parametrizes the family of Laplacians and the generalizations of the coefficients $\alpha_j, \beta_j, \gamma_j$. Two slightly different methods for computing these coefficients are presented; both are based on the same idea, which is that composing a polynomial with one of the contraction mappings used in the construction of $SG$ yields another polynomial which can be expressed uniquely as a linear combination of monomials. One of these methods results in a system of linear equations which can be solved recursively to obtain the values of $\alpha_j(r), \beta_j(r) ,$ and $\gamma_j(r)$ for $j$ as large as desired. The other method yields explicit recurrence relations which require only knowledge of the boundary values of all lower degree monomials to find $\alpha_j(r), \beta_j(r),$ or $\gamma_j(r)$. These recurrence relations are quite complicated and seem to not be very computationally useful. \par

Section 3 generalizes several results of \cite{nsty} about polynomials, monomials, and their derivatives to the family of symmetric self-similar Laplacians. Several properties of the standard polynomials extend to this context easily, but there are notable differences in their behavior.  \par 

The coefficients $\alpha_j(r), \beta_j(r)$, and $\gamma_j(r)$ are rational functions of the parameter $r$ whose degree grows rapidly with $j$, so numerical investigation seems to be the only tractable method to learn anything about the long term behavior of these sequences. Section 4 presents a variety of numerical data such as graphs of $\alpha_j(r), \beta_j(r) , \gamma_j(r)$ as well as ratios between successive terms of these sequences and graphs of the monomials themselves. Interesting patterns emerge in these graphs. These data demonstrate that among this family of Laplacians, the standard Laplacian is exceptional, and they suggest relationships between the Neumann eigenvalues of the Laplacians and the ratios of successive terms of the sequences $\{\alpha_j(r)\}$ and $\{\beta_j(r)\}$. 
\par
Section 5 discusses questions that remain and states conjectures based on the numerical data presented in section 4.
 \par 
For completeness, and since it does not appear elsewhere, the calculation of the Neumann eigenvalues of the family of Laplacians is described in an appendix. This computation uses the spectral decimation established in \cite{fkls} to obtain the eigenvalues of the Laplacians as a limit of the eigenvalues of discrete approximations to the Laplacians.

\par

The Sierpinski gasket is usually defined as the unique nonempty compact subset of the plane satisfying the self-similar identity
\begin{equation}
SG = \bigcup_{i=0}^2 F_i SG
\end{equation}
where $F_i$ is the contraction of the plane 
\begin{equation} F_ix = \frac{1}{2}(x-q_i)+q_i
\end{equation}
which fixes $q_i$, and $q_0, q_1, q_2$ are the vertices of a triangle $T$. By iterating this decomposition of $SG$, a more general self-similar identity 
\begin{equation}
SG = \bigcup_{|w|=m} F_wSG ,
\label{selfsim}
\end{equation}
is obtained,
where $w = (w_1, w_2, \dots, w_m)$ is a \textit{word of length $|w|=m$} with $w_i\in \{0,1,2\}$ and $F_w = F_{w_1}  \circ\dots \circ F_{w_m}$. The image $F_w SG $ where $|w|=m$ is a \textit{cell of level $m$}.
It is useful to consider $SG$ as being approximated by a sequence of graphs $\{\Gamma_m\}$ where $\Gamma_0$ is the complete graph on $V_0 = \{q_0, q_1, q_2\}$ (the vertices of $T$), and $\Gamma_m$ has the vertices $V_m = \bigcup_{i=0}^2 F_i V_{m-1} $ and edges $x \simmm y $ if $x = F_wq_j$ and $y = F_wq_k$ for some $j \neq k$. The vertices $q_0, q_1, q_2$ form the \textit{boundary} of $SG$. The set $V_*=\bigcup_{m=0}^\infty V_m$ is the set of \textit{vertices}, and the $V_*\setminus V_0$ is the set of \textit{junction points}.\par 

Given these definitions, it is natural to define a self-similar probability measure on $SG$ which satisfies 
\begin{equation}
\mu(F_iA) = \frac13 \mu( A ).
 \end{equation}
A more general self-similar probability measure which is compatible with the IFS consisting of three mappings is defined by requiring that
 \begin{equation}
 \mu(F_iA)=  \mu_i \mu( A ) 
 \end{equation}
 for weights $\mu_i$ satisfying $\sum_{i=0}^2 \mu_i =1$, but it is easy to see that the only such measure which is symmetric in the sense that measures of sets are preserved under the action of the dihedral symmetry group $D_3$ on $SG$ is the standard measure in which $\mu_i = \frac13$ for all $i$.  \par

 The other preliminary construction necessary to construct a Laplacian is a self-similar energy, an energy form satisfying 
 \begin{equation}
 \mathcal{E} (u) = \sum_{ i=0}^2 r_i^{-1} \mathcal{E}(u \circ F_i).
 \end{equation}
 Utilizing the graph approximations of $SG$, a sequence of graph energies $\mathcal{E}_m(u)$ can be defined for functions $u$ defined on $V_m$ by setting
 \begin{equation}
 \mathcal{E}_m (u) = \sum_{x \simmm y}  
 {r_m(x,y)}^{-1} (u(x)-u(y))^2 
 \end{equation}
 where $r_m(x,y)$ is the \textit{resistance} assigned to the edge $\{x,y\}$ in $\Gamma_m$ and $r_{m+1} ( F_i x , F_i y ) = r_i r_m ( x, y ), r_i >0. $ This energy $\mathcal{E}$ for functions defined on all of $SG$ arises as 
 \begin{equation}
 \mathcal{E}(u) = \lim_{m \to \infty } \mathcal{E}_m (u). \end{equation}
 The energy is a quadratic form, and there is an associated bilinear form 
 \begin{equation} \mathcal{E}(u,v) = \lim_{m \to \infty} \sum_{x \simmm y}  
 {r_m(x,y)}^{-1} (u(x)-u(y))(v(x) - v(y)).
 \end{equation} 
  The standard energy on $SG$ is obtained by setting $r_i=\frac{3}{5}$ and $r_0(x,y) =1$ for all edges in $V_0$, and up to a constant multiple it is clear that this standard energy is the only symmetric ($\mathcal{E} (u) = \mathcal{E}(u \circ \sigma)$ for $\sigma \in D_3$) self-similar energy form which is compatible with the 3-element IFS. Given a function $u$ defined on $V_m$, the \textit{harmonic extension} of $u$ to $V_{m+1}$ is the extension of $u$ which minimizes $\mathcal{E}_{m+1}(u)$. A \textit{harmonic function} is the harmonic extension of a function defined on $V_0$ to $V_*$ and extended to all of $SG$ by continuity. Requiring that $\mathcal{E}(u) = \mathcal{E}_0(u) $ for all harmonic functions $u$ restricts the possible values of $r_i$ to a surface in $\R^3$, for more on this see \cite{cs}. This requirement is called the renormalization condition and will be enforced here. \par

 Kigami's construction of a Laplacian relates a self-similar measure and a self-similar energy via the weak formulation of the Laplacian: $\Delta u$ is a function such that
 \begin{equation}
 - \mathcal{E} (u,v) = \int_{SG} v \Delta u \: d \mu
 \end{equation} for all $v$ with finite energy vanishing at the boundary,
and this Laplacian is self-similar, meaning that there are constants $L_i= r_i \mu_i$ such that 
\begin{equation}
\Delta(u \circ F_i) = L_i (\Delta u ) \circ F_i
\end{equation} 
(in this case $L_i = \frac{1}5$), and symmetric, meaning that the Laplacian commutes with the $D_3$ action on $SG$.  Up to a constant, the standard Laplacian is the only self-similar symmetric Laplacian on $SG$ which is compatible with the 3-element IFS. The Laplacian can be computed at junction points of $SG$ by a pointwise formula reminiscent of the second difference quotient from analysis on the line and extended by continuity to all of $SG$. If $u$ is a function defined on $SG$ and $x$ is a junction point of two cells of level $m$, then the \textit{standard level $m$ graph Laplacian} of $u$ at $x$, $\Delta^{(m)} u (x) $ is defined as 
\begin{equation}
\Delta^{(m)} u (x) = \sum_{x \simmm y } u(y) - u(x) 
\end{equation}
and the standard Laplacian is a renormalized limit of these graph Laplacians:
\begin{equation}\label{pointwise}
\Delta u (x)  = \frac{3}{2} \lim_{m \to \infty } 5^m \Delta^{(m)} u(x) .
\end{equation}
 \par 

In addition to the Laplacian, there are two other ``differential operators" for functions on $SG$, the \textit{normal derivative}, $\partial_n$, and \textit{tangential derivative}, $\partial_T$. However, these operators are defined only at the boundary points of $SG$--they cannot be extended to operators on all of $SG$ (local versions defined at the junction points exist and are discussed later). Their definitions depend on the energy; in the standard case they are defined as 
\begin{equation}
\partial_n u ( q_i) = \lim_{m \to \infty } \left ( \frac{ 3}{5} \right )^{-m} (2 u (q_i) - u(F_i^m q_{i+1}) - u (F_i^m q_{i-1})) 
\end{equation}
\begin{equation}
\partial_T u (q_i) = \lim_{m \to \infty} \left ( \frac{1}{5}\right )^{-m} (u (F_i^m q_{i+1} ) - u ( F_i^m q_{i-1}) )
\end{equation}
where cyclic notation is used for the indices. The normal derivative interacts with the Laplacian through the Gauss-Green formula:
\begin{equation} 
\int_{SG} u \Delta v - v \Delta u \: d \mu = \sum_{ V_0} u \partial _n v - v \partial_n u ,
\end{equation}
while the significance of the tangential derivative is less clear. \par 

Given a Laplacian $\Delta$, it makes sense to define a multiharmonic function, or \textit{polynomial} on $SG$ as any function $P$ satisfying 
\begin{equation}
\Delta^{n+1} P =0
\end{equation}
for some $n$, in analogy with the fact that polynomials on the line vanish if a sufficient number of derivatives, or equivalently a sufficiently high power of the Laplacian, is taken. The \textit{degree} of a polynomial $P$ on $SG$ is the minimal $n$ such that $\Delta^{n+1} P =0$. The space of polynomials of degree $n$ has dimension $3n+3$, hence a polynomial of degree $n$ is uniquely determined by its \textit{$n$-jet} at $q_0$, the list of values
$$ (P(q_0), \partial_nP(q_0), \partial_TP(q_0), \Delta P(q_0) , \partial_n \Delta P(q_0) , \partial_T \Delta P(q_0), \dots ,  \Delta^n P(q_0) , \partial_n \Delta^n P(q_0) , \partial_T \Delta^n P(q_0)).$$
Further extending the analogy with polynomials on the line, it is natural to define a basis for the space of polynomials characterized by the property that the jet of each basis element has only a single nonzero entry which is a 1. That is, the \textit{monomials} on $SG$, $P_{j,1}, P_{j,2}, P_{j,3}$, are defined by 
\begin{equation} 
\Delta^nP_{j,k} (q_0) = \delta_{jn} \delta_{k1} \end{equation}
\begin{equation}
\partial_n\Delta^nP_{j,k}(q_0) = \delta_{jn} \delta_{k2} \end{equation}
\begin{equation}
\partial_T \Delta^n P_{j,k} (q_0) = \delta_{jn} \delta_{k3}.
\end{equation}
These monomials are the analogue of the functions $\frac{x^n}{n!}$ on the interval. \par 

The monomials with respect to the standard Laplacian are the focus of \cite{nsty}. Specifically, emphasis is placed on three sequences of values $\{\alpha_j\}, \{\beta_j\}, \{\gamma_j\}$ defined as 
\begin{equation}
\alpha_j = P_{j,1}(q_1), \quad \beta_j=P_{j,2}(q_1), \quad \gamma_j = P_{j,3}(q_1).
\end{equation}
These values are the analogues of the coefficients 
$\frac{1}{n!}$ on the interval. The fast decay of $\frac{1}{n!}$ as $n$ increases (faster than any exponential) is important as it means that there is a wide class of permissible sequences of coefficients for Taylor series which define functions which are entire analytic. With this idea in mind, it is reasonable to consider the decay rates of the sequences $\{\alpha_j\}, \{\beta_j\},$ and $ \{\gamma_j\}$ as $j$ increases with an aim towards a reasonable definition of entire analytic functions on $SG$. In \cite{nsty}, it was shown that $\alpha_j$ and $\gamma_j$ approach 0 faster than any exponential as $j$ increases, while the values of $\beta_j$ decay as $(- \lambda_2)^{-j}$ where $\lambda_2$ is the second nonzero Neumann eigenvalue of the standard Laplacian. \par 

Since the standard Laplacian is the only Laplacian which is self-similar and symmetric and compatible with the 3-element IFS, it is necessary to alter the iterated function system which defines $SG$ in order to obtain more Laplacians which are self-similar and symmetric. In \cite{fkls}, the standard IFS is replaced with 
\begin{equation} \{ F_{ij}=F_i \circ F_j \: | \: F_i = \frac{ 1}{2} (x - q_i ) + q_i\}. \end{equation}
This IFS consists of 9 contractions, and similar to before the Sierpinski gasket is defined as 
\begin{equation} 
SG= \bigcup_{0 \leq i,j \leq 2 } F_{ij} SG.
\end{equation}
This definition of $SG$ is a special case of \eqref{selfsim} when $m=2$.
However, this iterated function system has the property that its images form 2 distinct equivalence classes under the $D_3$ action on $SG$; this allows self-similar symmetric measures and energy forms to be defined that do not assign the same measure and resistance to every cell. The images $F_{ii}SG $ contain the boundary points of $SG$ and will be called \textit{outer cells} and images $F_{ij}SG$, $i \neq j$ will be called \textit{inner cells}. With this IFS, \textit{cells of level $m$} will be defined as the images $F_wSG$ where $w=(w_1, w_2, \dots , w_m ), \ w_i \in \{ij \: | \: i ,j \in \{ 0,1,2\}\}$ and $F_w = F_{w_1} \circ \dots \circ F_{w_m}$. The inner cells and the outer cells are the two equivalence classes under the action of $D_3$. From this IFS, a 1-parameter family of symmetric self-similar Laplacians was constructed in \cite{fkls}; their construction is summarized here.\par 
A self-similar symmetric probability measure on $SG$ with this 9-element IFS must assign the same measure to each of the outer cells and the same measure to each of the inner cells, so 
\begin{equation}
 \mu(F_{ij} A) = \begin{cases}
 \mu_0 \mu(A) & i = j, \\
 \mu_1 \mu(A) & i \neq j
 \end{cases}
 \end{equation}
 where $ 3 \mu_0 + 6 \mu_1 =1.$ This leaves one parameter free to vary. These measures will be denoted $\mu_s$ where $s = \frac{\mu_1}{\mu_0}$.\par 
 
 Similarly, different resistance can be assigned to the outer and inner cells to obtain a set of self-similar symmetric energy forms. This is achieved by setting
 \begin{equation}
 r_{m+1}(F_{ij} x, F_{ij}y ) = \begin{cases} 
 r_0 r_m (x,y) & i = j, \\
 r_1 r_m (x,y) & i \neq j. \end{cases}
 \end{equation}
 
 Again this leaves one parameter free to vary: $r = \frac{r_0}{r_1}$. Imposing the renormalization condition forces 
 \begin{equation} r_0(r) = \frac{6r(r+2) }{9r^2 + 26r + 15 }, \quad r_1(r) = \frac{ 6(r+2) }{9r^2 + 26r + 15 },
 \end{equation}
 see \cite{fkls} for the computation of these values using the $\Delta-Y$ transform. The graph energies 
 \begin{equation} 
 \mathcal{E}_{r,m}(u)= \sum_{x \simmm y}  
 {r_m(x,y)}^{-1} (u(x)-u(y))^2
 \end{equation}
 are defined as before.
 These assignments of resistance give rise to the energy form 
 \begin{equation} \mathcal{E}_r (u) = \lim_{m \to \infty} \mathcal{E}_{r,m}(u) \end{equation}
  and the associated bilinear energy. The set of functions $u$ for which $\mathcal{E}_r(u)$ is finite is $\dom \mathcal{E}_r$. 
 
 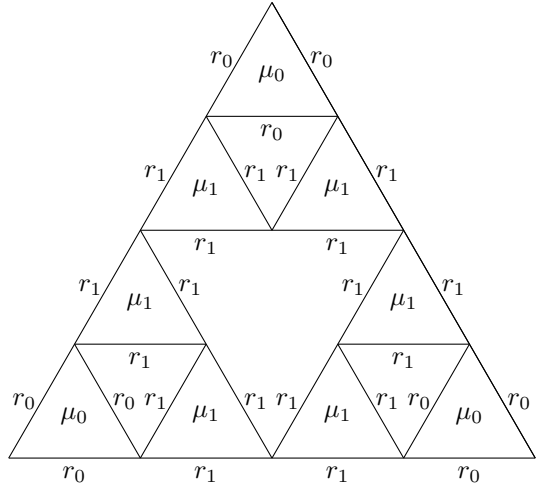
\begin{figure}
 \begin{center}
 \begin{tikzpicture}[scale=7]
  
\draw[ -] (-.5,0) -- (.5,0);
\draw[ -](-.5,0) --(0,.86603);
\draw[ -](.5,0) --(0,.86603);
\draw[ -](.5,0) --(0,.86603);
\draw[ -](0,0) --(.25, .43302);
\draw[ -](0,0) --(-.25, .43302);
\draw[ -](.25, .43302) --(-.25, .43302);
\draw[ -](0, .43302) --(-.125, .64952);
\draw[ -](0, .43302) --(.125, .64952);
\draw[ -](.125, .64952) --(-.125, .64952);
\draw[ -](-.25, 0)-- (-.125,.21651);
\draw[ -](-.25, 0)-- (-.375,.21651);
\draw[ -](-.375,.21651)-- (-.125,.21651);
\draw[ -](.25, 0)-- (.125,.21651);
\draw[ -](.25, 0)-- (.375,.21651);
\draw[ -](.375,.21651)-- (.125,.21651);
\filldraw (.125,-.03) circle (0pt) node {$r_1$};
\filldraw (-.125,-.03) circle (0pt) node {$r_1$};
\filldraw (.375,-.03) circle (0pt) node {$r_0$};
\filldraw (-.375,-.03) circle (0pt) node {$r_0$};
\filldraw ( .47, .10825) circle (0pt) node {$r_0$};
\filldraw ( .345, .32476) circle (0pt) node {$r_1$};
\filldraw ( .22, .54127) circle (0pt) node {$r_1$};
\filldraw ( .095, .75778) circle (0pt) node {$r_0$};
\filldraw ( -.47, .10825) circle (0pt) node {$r_0$};
\filldraw ( -.345, .32476) circle (0pt) node { $r_1$};
\filldraw ( -.22, .54127) circle (0pt) node { $r_1$};
\filldraw ( -.095, .75778) circle (0pt) node { $r_0$};

\filldraw (.28, .10825) circle (0pt) node { $r_0$};
\filldraw (.22, .10825) circle (0pt) node { $r_1$};
\filldraw(.25, .18651) circle (0pt) node { $r_1$};
\filldraw(.03, .10825) circle (0pt) node { $r_1$};
\filldraw(.155, .32476) circle (0pt) node { $r_1$};
\filldraw(.125, .40302) circle (0pt) node { $r_1$};
\filldraw(.03, .54127) circle (0pt) node { $r_1$};

\filldraw(0,.61952) circle (0pt) node { $r_0$};

\filldraw (-.28, .10825) circle (0pt) node { $r_0$};
\filldraw (-.22, .10825) circle (0pt) node { $r_1$};
\filldraw(-.25, .18651) circle (0pt) node { $r_1$};
\filldraw(-.03, .10825) circle (0pt) node { $r_1$};
\filldraw(-.155, .32476) circle (0pt) node { $r_1$};
\filldraw(-.125, .40302) circle (0pt) node { $r_1$};
\filldraw(-.03, .54127) circle (0pt) node { $r_1$};

\filldraw(0, .73) circle (0pt) node { $\mu_0$};

\filldraw(.25, .29699) circle (0pt) node {$\mu_1$};
\filldraw(.125, .51350) circle (0pt) node {$ \mu_1$};
\filldraw(.375, .08048) circle (0pt) node {$\mu_0$};
\filldraw(.125, .0848) circle (0pt) node {$\mu_1$};

\filldraw(-.25, .29699) circle (0pt) node {$\mu_1$};
\filldraw(-.125, .51350) circle (0pt) node {$ \mu_1$};
\filldraw(-.375, .08048) circle (0pt) node {$\mu_0$};
\filldraw(-.125, .0848) circle (0pt) node {$\mu_1$};

\end{tikzpicture}
\caption{The assignment of measure and resistance to the level 1 cells of $SG$.}
\end{center}
\end{figure}
\par 
Again, the weak formulation gives rise to a Laplacian given a measure and energy of this form: $\Delta_{r,s} u$ is a function such that
\begin{equation}
 - \mathcal{E}_r (u,v) = \int_{SG} v \Delta_{r, s} u \: d \mu_s
\end{equation}
for all $v\in \dom \mathcal{E}_r$ which vanish at the boundary. The subscripts $r, s$ denote the dependence on $r$ and on the measure $\mu_s$. 
 \par 
 Finally, in \cite{fkls} the condition that the Laplacian renormalization constants $L_{ij}$ are uniform over all cells is imposed, that is $  r_0 \mu_0= r_1 \mu_1.$ Thus $r = s$, so $r$ is the only free parameter and
 \begin{equation}
 \mu_0(r) = \frac{ 1}{3(2r+1)}, \quad \mu_1(r) = \frac{ r}{3(2r+1)} ,
 \end{equation}
 so 
 \begin{equation} L_{ij}= L(r) = \frac{ 2r(r+2) }{ (2r+1) (9r^2 +26r +15)}.
 \end{equation}
 Laplacians in this 1-parameter family will be denoted $\Delta_r$ to emphasize the dependence on $r$ and the set of functions $u$ for which $\Delta_r u$ exists and is continuous will be denoted $\dom \Delta_r$. These Laplacians satisfy 
 \begin{equation} \Delta_r ( u \circ F_{ij} ) = L(r) ( \Delta_ r u ) \circ F_{ij}. \label{Laplacian scaling}
 \end{equation} 
 \par

 For a given value of $r$, $\Delta_r u(x)$ can be calculated at a junction point $x$ by a pointwise formula in the spirit of \eqref{pointwise}. Let $x\in V_m\setminus V_0$. Then the \textit{level $m$ graph Laplacian} $\Delta_{r}^{(m)}$ is defined by 
 \begin{equation}
 \Delta_{r}^{(m)} u ( x) =  \frac{\sum_{ y \simmm x } r_m(x,y)^{-1} (u(y)-u(x))}{ \int_{SG} \psi_{r,x} ^{(m)} \: d \mu_r }\end{equation}
 where $\psi_{r,x} ^{(m)} $ is defined to be piecewise harmonic with respect to $\mathcal{E}_r$ in $SG\setminus V_m$ and satisfies $\psi_{r,x} ^{(m)}(y) = \delta_{xy}$ for $y \in V_m$ (\cite{su}). 
 The Laplacian can be computed as a limit of graph Laplacians by 
 \begin{equation} 
 \Delta_r u (x) = \lim_{m \to \infty} \Delta_r^{(m)} u (x) \end{equation}
 and it can be shown that the convergence is uniform for $u \in \dom \Delta_r$.
  Another graph Laplacian, $\tilde \Delta_r^{(m)} $, will also be useful. 
Let $x \in V_m \setminus V_0,$ $x= F_w q_i = F_{w^\prime } q_j $, $j \neq i$, be the junction point between the two cells of level $m$ $F_wSG$ and $F_{w^\prime } SG$. Then $x$ has neighbors in the level $m$ graph approximation of $SG$ $y_1 = F_{w}q_{i-1},\  y_2 = F_w q _{i+1},\  y_3 = F_{w^\prime } q _{j-1} ,\  y_4 = F_{w^\prime } q_{j+1}.$ Define
\begin{equation} 
\tilde \Delta _r^{(m)} u (x) = \begin{cases} 
u(y_1) + u(y_2) + u(y_3) + u(y_4) - 4u(x) & r_m(x,y_1) = r_m (x,y_3), \\
ru(y_1) + r u(y_2 ) + u(y_3) + u(y_4) - (2+2r) u(x) &r r_m (x,y_1) = r_m (x,y_3). \end{cases} \label{graphlaplacian} \end{equation}
From equations (7.1), (7.2), (7.3) of \cite{fkls}, it can be seen that 
\begin{equation} 
\Delta_r^{(m)} u(x) = \begin{cases}\frac{3}{2} L(r) ^{-m} \tilde \Delta_r^{(m)} u(x) &  r_m(x,y_1) = r_m (x,y_3) \\
\frac{3}{r+1} L(r)^{-m} \tilde \Delta_r^{(m)} u(x) & r r_m (x,y_1) = r_m (x,y_3) .\end{cases} \label{graph laplacian} \end{equation}

 \par 
 In the case $r=1$, $\Delta_1$ is the standard Laplacian, and when $r=1$, $\mu_0 = \mu_1 = \frac{1}{9} = \left (\frac{1}{3} \right )^2,$ and $r_0 = r_1 = \frac{9}{25} = \left ( \frac{3}{5}\right ) ^2. $
 \par 
 As mentioned previously, the definitions of the normal and tangential derivatives depend on the energy, so for each value of $r$ there are associated normal and tangential derivatives $\partial_{n,r}, \partial_{T,r}.$ To define these derivatives, it is first necessary to consider the harmonic functions with respect to this family of energy forms. Let $u(q_0) =1, u ( q_1) = u ( q_2 ) =0$. Consider an extension of $u$ to $V_1$, with values on $V_1 \setminus V_0$ labelled $u_1,u_2, u_3, u_4, u_5, u_6, u_7$ as shown in figure 1.2. 

 \begin{figure} 
 \begin{center}
 \begin{tikzpicture}[scale=7]
\draw[ -] (-.5,0) -- (.5,0);
\draw[ -](-.5,0) --(0,.86603);
\draw[ -](.5,0) --(0,.86603);
\draw[ -](.5,0) --(0,.86603);
\draw[ -](0,0) --(.25, .43302);
\draw[ -](0,0) --(-.25, .43302);
\draw[ -](.25, .43302) --(-.25, .43302);
\draw[ -](0, .43302) --(-.125, .64952);
\draw[ -](0, .43302) --(.125, .64952);
\draw[ -](.125, .64952) --(-.125, .64952);
\draw[ -](-.25, 0)-- (-.125,.21651);
\draw[ -](-.25, 0)-- (-.375,.21651);
\draw[ -](-.375,.21651)-- (-.125,.21651);
\draw[ -](.25, 0)-- (.125,.21651);
\draw[ -](.25, 0)-- (.375,.21651);
\draw[ -](.375,.21651)-- (.125,.21651);
\filldraw (.55,.-.05) circle (0pt) node {$0$};
\filldraw (-.55,.-.05) circle (0pt) node { $0$};
\filldraw (0,.91603) circle (0pt) node { $1$};
\filldraw (.175,.69952) circle (0pt) node {$u_1$};
\filldraw (-.175,.69952) circle (0pt) node {$u_1$};
\filldraw (.3,.45) circle (0pt) node {$u_2$};
\filldraw (-.3,.45) circle (0pt) node {$u_2$};
\filldraw (0,.375) circle (0pt) node {$u_3$};
\filldraw (.425,.25) circle (0pt) node {$u_4$};
\filldraw (-.425,.25) circle (0pt) node {$u_4$};
\filldraw (.075,.25) circle (0pt) node {$u_5$};
\filldraw (-.075,.25) circle (0pt) node {$u_5$};
\filldraw (.25,.-.05) circle (0pt) node {$u_6$};
\filldraw (-.25,.-.05) circle (0pt) node { $u_6$};
\filldraw (0,.-.05) circle (0pt) node {$u_7$};
\end{tikzpicture}
\caption{The labelling of the values of the extension of $u$ to $V_1$.}
\end{center}
\end{figure}
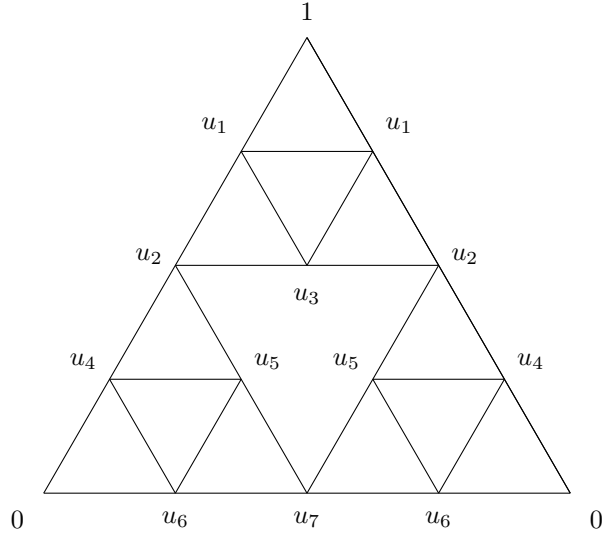
The values of the harmonic extension $\tilde u$ of $u$ to $V_1$ are the values which minimize 
\begin{center} \begin{gather*} 
\mathcal{E}_{r,1}(\tilde u )  =2 \bigg (  \frac{(1-u_1)^2 }{r_0} +  \frac{(u_1-u_2)^2 }{r_1}+ \frac{(u_1-u_3)^2}{r_1}+ \frac{(u_2-u_3)^2 }{r_1}+ \frac{(u_2-u_4)^2 }{r_1}+ \frac{(u_2-u_5)^2 }{r_1}+ \frac{u_4^2 }{r_0}+  \frac{(u_4-u_5)^2 }{r_1}\\+
 \frac{(u_4-u_6)^2 }{r_0} + \frac{(u_5-u_6)^2 }{r_1}+
 \frac{(u_5-u_7)^2 }{r_1}+  \frac{u_6^2 }{r_0}+  \frac{(u_6-u_7)^2 }{r_1}\bigg )   \tag{\stepcounter{equation}\theequation}
 \end{gather*}
 \end{center}
 and are obtained by solving 
 \begin{equation} \nabla \mathcal{E}_{r,1} ( \tilde u )  = 0.\label{harmext} \end{equation}
 The solution to \eqref{harmext} is 
 \begin{equation}
 u_1= \frac{3r^2 + 14r + 15 }{9r^2 + 26r+15} ,\ 
  u_2= \frac{3r^2 + 10r + 7 }{9r^2 + 26r+15} ,\ 
   u_3= \frac{3r^2 + 12r + 11 }{9r^2 + 26r+15} ,\ 
    u_4= \frac{3r^2 + 7r }{9r^2 + 26r+15}  , \ \end{equation}
$$     u_5= \frac{3r^2 + 7r + 2 }{9r^2 + 26r+15}, \ 
      u_6= \frac{3r^2 + 5r }{9r^2 + 26r+15},\ 
       u_7= \frac{3r^2 + 6r + 1 }{9r^2 + 26r+15}. 
$$
       \par     
 The space $\mathcal{H}_r$ of harmonic functions with respect to the energy $\mathcal{E}_r $ is a 3-dimensional vector space with a simple basis given by $\{u, u \circ \rho, u \circ \rho^{-1}\} $ where $\rho \in D_3$ is the rotation which sends $q_i$ to $q_{i+1}$ (cyclic notation). The factors of $\frac{ 3}{5} $ and $\frac{1}{5} $ in the definitions of the standard normal and tangential derivatives, respectively, arise as the eigenvalues of the mapping from $\mathcal{H}_1 \to \mathcal{H}_1$ given by 
 $u \mapsto u \circ F_0$. It is easy to see that the functions $h_1 = u \circ \rho + u \circ \rho^{-1} $ and $h_2= u \circ \rho - u \circ \rho^{-1} $, which are symmetric and skew-symmetric, respectively, with respect to the reflection $R_0 \in D_3$ which fixes $q_0$, are eigenfunctions of the mapping 
$ u \mapsto u \circ F_{00}.$ The associated eigenvalues are 
\begin{equation} 
\lsym(r) = u_4 + u_6 = \frac{6r (r + 2 )}{9r^2 + 26 r + 15 }=r_0, \quad \lskew(r)= u_4 - u_6 = \frac{2r}{9r^2 + 26 r + 15}.
\end{equation}
Then the normal derivative and tangential derivative with respect to $\mathcal{E}_r$ are defined as 
\begin{equation}
\partial_{n,r} u ( q_i ) = \lim_{m \to \infty } \lsym(r)^{-m} ( 2 u (q_i ) - u ( F_{ii}^m q_{i+1} ) - u ( F_{ii}^m q _ { i-1} ) ), \end{equation} \begin{equation}
\partial_{T, r} u ( q_i)  = \lim_{m \to \infty} \lskew(r)^{-m} ( u ( F_{ii}^m q_{i+1} ) - u ( F_{ii}^m q_{i-1} ) ) ,
\end{equation}
respectively. \par 

\begin{theorem} Suppose $u \in \dom \Delta_r$ for some $r$. Then $\partial_{n,r} u (x) $ exists for all $x \in V_0$, and the Gauss-Green formula 
\begin{equation}
\mathcal{E}_r(u,v) = - \int_{SG} ( \Delta_r u ) v \: d \mu_r+ \sum_{ V_0 } v (x) \partial_{n,r} u (x)\label{Gauss-Green}
\end{equation}
 holds for all $v \in \dom \mathcal{E}_r$. 
\end{theorem}
\begin{proof} Begin with the definition of the bilinear graph energy
$$
\mathcal{E}_{r,m}(u,v) = \sum_{ x \simmm y  }r_m (x,y ) ^{-1} (u (y) - u(x)) (v(y) - v(x))  $$ and collect terms that appear with a factor of $v(x) $ to obtain 
\begingroup
\allowdisplaybreaks
\begin{center}
\begin{gather*} 
\mathcal{E}_{r,m} (u,v) =- \sum_{V_m } v(x) \sum_{ y \simmm x } r_m ( x, y )^{-1} (u(y) - u(x) ) \\ = - \sum_{V_m \setminus V_0} v(x) \Delta_r^{(m)} u(x) \int_{SG} \psi_{r,x}^{(m)} \: d \mu_r + \sum_{V_0} v(x) \sum_{ y \simmm x }r_0^{-m} (u(x) - u (y) )  .    \tag{\stepcounter{equation}\theequation} 
 \end{gather*}
 \end{center}
 \endgroup
 The first term can be rewritten as 
 \begin{equation} 
 - \sum_{V_m \setminus V_0} v(x) \Delta_r^{(m)} u(x) \int_{SG} \psi_{r,x}^{(m)} \: d \mu_r=-\int_{SG} \left (\sum_{V_m \setminus V_0} v(x) \Delta_r^{(m)} u(x) \psi_{r,x}^{(m)} \right )  \: d\mu_r . \end{equation}
 It is not difficult to see that $\sum_{V_m \setminus V_0} v(x) \Delta_r^{(m)} u(x) \psi_{r,x}^{(m)}$ converges to $( \Delta_r u ) v $ except on $V_0$ which is a set of measure zero, so, taking the limit as $m \to \infty$, (1.44) becomes 
 \begin{equation} 
 \mathcal{E}_r (u,v) = - \int_{SG} ( \Delta_r u ) v \: d \mu_r + \lim_{m \to \infty}  \sum_{V_0} v(x) \sum_{ y \simmm x }r_0^{-m} (u(x) - u (y) ) .\label{1.45} \end{equation} 
 This shows that the limit  $ \lim_{m \to \infty}  \sum_{V_0} v(x) \sum_{ y \simmm x }r_0^{-m} (u(x) - u (y) )$ exists regardless of $v$, so choosing $v \in \dom \mathcal{E}_r$ vanishing at all boundary points except for $q_i$ shows that $\partial_{n,r} u(q_i)$ exists and that \eqref{Gauss-Green} holds. \qed 
  \end{proof} \par \vspace{4mm}

The normal and tangential derivatives can also be defined at points in $V_m$ by viewing them as the boundary points of a cell of level $m$. The normal derivative of $u$ at $F_w q_i$ with respect to the cell $F_w SG$ is 
\begin{equation} 
\partial_{n,r}^w u (F_w q_i) = \lsym(r) ^{-|w|} \partial_{n,r} (u \circ F_w ) (q_i),
\label{local normal}
\end{equation}
and, similarly, the tangential derivative with respect to $F_wSG$ is 
\begin{equation} 
\partial_{T,r}^w u (F_w q_i) = \lskew(r) ^{-|w|} \partial_{T,r} (u \circ F_w) (q_i) .
\label{local tangential}
\end{equation}

\begin{theorem}
Suppose $u \in \dom \Delta_r $. Then at each junction point $x = F_wq_i = F_{w^\prime} q_j $, $|w|= |w^\prime| $, the local normal derivatives exist and 
\begin{equation}
\partial_{n,r}^w u (F_w q_i ) + \partial_{n,r}^{w^\prime} u (F_{w^\prime} q_j) =0 
\end{equation} 
if $r_{|w|}(x,y_1)=r_{|w|}(x,y_3)$, and 
\begin{equation} 
r\partial_{n,r}^w u (F_w q_i) + \partial_{n,r}^{w^\prime} u (F_{w^\prime} q_j) =0 
\end{equation}
if $rr_{|w|}(x,y_1)=r_{|w|}(x,y_3)$, where $y_1, y_2, y_3, y_4 $ are as in \eqref{graphlaplacian}. This is called the matching condition for normal derivatives. \end{theorem}

\begin{proof}
The existence follows from theorem 1.1. Since $u \in \dom \Delta_r$, 
$$ \lim_{m \to \infty } \Delta_r^{(m)} u(x) $$ exists. Since $$\lim_{m \to \infty } \int_{SG} \psi_{r,x}^{(m)} \: d \mu _r =0,$$ it follows that $$\lim_{m \to \infty }  \sum_{y \simmm x } r_m(x,y)^{-1} ( u(y) - u(x) ) =0.$$
If all edges containing $x$ have equal resistance, then 
$$ \lim_{m \to \infty } \lsym^{-|w| } r_{|w|} (x,y_1) \sum_{y \simmm x } r_m (x,y) ^{-1} (u(x) - u(y) ) = 0 = \partial_{n,r}^w u (F_w q_i ) + \partial_{n,r}^{w^\prime} u (F_{w^\prime} q_j).$$
If $rr_{|w|}(x,y_1)=r_{|w|}(x,y_3)$, then 
$$ \lim_{m \to \infty } \lsym^{-|w|} \left (rr_{|w|} ( x,y_1 ) \sum_{\substack{ y \simmm x  \\ y \in F_w SG }} r_m ( x,y) ^{-1} (u(x) - u(y) ) + r_{|w|} (x,y_3) \sum_{ \substack{ y \simmm x \\ y \in F_{w^\prime }SG}} r_m (x,y) ^{-1} (u(x) - u(y) ) \right)$$$$ = 0 = r\partial_{n,r} ^wu (F_w q_i) + \partial_{n,r} ^{w^\prime}u (F_{w^\prime} q_j). \qed$$
 \end{proof}
\par 
\vspace{4mm}

The functions $P$ which satisfy 
\begin{equation} \Delta_r ^{n+1} P =0 \end{equation}
for some $n$ are the polynomials with respect to $\Delta_r$. The monomials $P_{j,1}^{(r)}, P_{j,2}^{(r)}, P_{j,3}^{(r)}$ are defined by 

\begin{equation} 
\Delta_r^nP_{j,k}^{(r)} (q_0) = \delta_{jn} \delta_{k1} \label{pj1 definition} \end{equation}
\begin{equation}
\partial_{n,r}\Delta^nP_{j,k}^{(r)}(q_0) = \delta_{jn} \delta_{k2} \label{pj2 definition} \end{equation}
\begin{equation}
\partial_{T,r} \Delta^n P_{j,k}^{(r)} (q_0) = \delta_{jn}  \label{pj3 definition} \delta_{k3}.
\end{equation}
The values of the monomials at the boundary of $SG$ now define three sequences of functions $(0, \infty) \to \R$: $\{ \alpha_j(r)\}, \{\beta_j(r)\}, \{\gamma_j(r) \} $ with 
\begin{equation}
\alpha_j (r)= P_{j,1}^{(r)} (q_1), \quad \beta_j(r)=P_{j,2}^{(r)}(q_1), \quad \gamma_j (r)= P_{j,3}^{(r)}(q_1). 
\end{equation}
The next section discusses how the boundary values of the monomials are calculated. \par

\section{Boundary Values of Monomials}
On the interval, computing the values of polynomials at any given point is trivial because of the algebraic structure of polynomials. However, polynomials on the Sierpinski Gasket lack an algebraic structure, in fact, for most values of $r$, the product of two nonconstant polynomials is not even in the domain of the Laplacian (\cite{bst}). However, the values of polynomials on $SG$ can be computed to any desired accuracy by taking advantage of the fact that if $P$ is a polynomial, then $P \circ F_{ij}$ is also a polynomial and hence has a unique representation as a linear combination of monomials. The values of the polynomials with respect to the standard Laplacian are computed in \cite{nsty} and \cite{cq}. The method here is based on the method in \cite{cq}. \par

Polynomials of degree 0 are harmonic functions, so (1.39) gives an algorithm for computing the degree 0 polynomials.
\begin{theorem} Let $h$ be a harmonic function with respect to $\mathcal{E}_r$. Then \begin{equation}\Delta_r h =0. \end{equation} \end{theorem}

\begin{proof} Suppose $h(q_0)=1, h(q_1) = h(q_2) = 0$. Then it is not hard to check using (1.39) that for all $x \in V_1\setminus V_0$, $\tilde \Delta_{r}^{(1)} h(x) =0.$ By symmetry and linearity, it follows that $\tilde \Delta_{r}^{(m)} h(x) =0 $ for all $x \in V_*$, so $\Delta_r h=0.$ \qed
\end{proof}
\par 
\vspace{4mm}
The harmonic functions which have both normal derivatives and tangential derivatives vanishing at $q_0$ are constant. Therefore 
\begin{equation}
P_{0,1}^{(r)} (x) =1.
\end{equation}
It is easy to see that the harmonic function $h$  with $h(q_0) =0 , h(q_1)= h(q_2) = -\frac{1}{2}$ has $\partial_{n,r} h(q_0) =1$ and $\partial_{T,r} h(q_0)=0$, hence \begin{equation}
h = P_{0,2}^{(r)}. 
\end{equation}
Similarly, the harmonic function $h^\prime$ with $h^\prime(q_0) = 0, h^\prime(q_1)=-h^\prime(q_2) =\frac{1}{2} $ has 
$\partial_{n,r} h(q_0) =0$ and $\partial_{T,r} h(q_0)=1$, so 
\begin{equation}
h^\prime= P_{0,3}^{(r)}.
\end{equation} It is then clear that 
\begin{equation} 
\alpha_0(r) =1, \quad \beta_0(r) = - \frac{1}{2}, \quad \gamma_0(r) = \frac{1}{2}.
\end{equation}
\par

The monomials on the line are all even or odd functions. Similarly, the monomials on $SG$ have symmetries with respect to the reflection $R_0 \in D_3$ which fixes $q_0$. 
\begin{lemma} The monomials satisfy 
\begin{equation}
P_{j,k}^{(r)}\circ R_0 = \begin{cases} P_{j,k}^{(r)} & k =1,2, \\
-P_{j,k}^{(r)} & k=3.
\end{cases}
\end{equation}
\end{lemma}

\begin{proof}
By the symmetry of $\Delta_r$, 
$$ \Delta_r u(q_0) = \Delta_r (u \circ R_0)(q_0).$$ 
It is also easy to see from the definition of the normal derivative that 
$$ \partial_{n,r} u(q_0) = \partial_{n,r} (u\circ R_0) (q_0), $$
and similarly the definition of the tangential derivative implies 
$$ \partial_{T,r} u(q_0) = - \partial_{T,r} ( u \circ R_0) (q_0),$$ 
and the desired conclusion follows. 
\qed
\end{proof}
\par
\vspace{4mm}

The degree 0 monomials are all eigenfunctions of the mapping $u \mapsto u \circ F_{00}$, and from the definition of monomials in terms of the jet at $q_0$ it is reasonable to suspect that all the monomials are eigenfunctions of this mapping. Indeed, by \eqref{Laplacian scaling} and the definitions of the normal and tangential derivatives, all the terms of the jet of a function are scaled by constants under composition with $F_{00}$, so, since the jet of a monomial has only a single nonzero entry, all the monomials are eigenfunctions of $u \mapsto u \circ F_{00}$. 
\begin{lemma} The monomials satisfy
\begin{equation}
P_{j,1}^{(r)} \circ F_{00} = L(r)^j P_{j,1}^{(r)},
\end{equation}
\begin{equation}
P_{j,2}^{(r)} \circ F_{00} = L(r)^j \lsym(r) P_{j,2}^{(r)},
\end{equation}
\begin{equation}
P_{j,3}^{(r)} \circ F_{00} = L(r)^j \lskew(r) P_{j,3}^{(r)}.
\end{equation}
\end{lemma}
\begin{proof} By definition, $$\Delta_r^j P_{j,1}^{(r)} =P_{0,1}^{(r)}=1,$$ so \eqref{Laplacian scaling} implies that $$\Delta_r^j( P_{j,1}^{(r)} \circ F_{00} ) = L(r)^j.$$ Because the Laplacian is a linear operator, it follows that $$P_{j,1}^{(r)} \circ F_{00} = L(r)^j P_{j,1}^{(r)}.$$ Again using the definitions of the monomials, 
$$\Delta_r^j P_{j,2}^{(r)} = P_{0,2}^{(r)},$$
so by \eqref{Laplacian scaling}
$$\Delta_r^j (P_{j,2}^{(r)} \circ F_{00}) = L(r)^j (P_{0,2}^{(r)} \circ F_{00}) = L(r)^j \lsym P_{0,2}^{(r)}, $$
so 
$$P_{j,2}^{(r)} \circ F_{00} =  L(r)^j \lsym(r) P_{j,2}^{(r)}.$$
Finally, 
$$\Delta_r^j P_{j,3}^{(r)} = P_{0,3}^{(r)}$$ 
so 
$$\Delta_r^j (P_{j,3}^{(r)} \circ F_{00}) = L(r)^j (P_{0,3}^{(r)} \circ F_{00}) = L(r)^j \lskew P_{0,3}^{(r)} $$
which implies 
$$P_{j,3}^{(r)} \circ F_{00} = L(r)^j \lskew(r) P_{j,3}^{(r)}. \qed$$

\end{proof} \par 
\vspace{4mm}

This lemma expresses the values of the monomials at the points $F_{00} q_i$ in terms of their boundary values. The next lemma expresses the values of the monomials at all points in $V_1$ in terms of the boundary values of $P_{j,1}^{(r)}$. 

\begin{lemma}
Let $x \in V_1 \setminus V_0$ and let $P$ satisfy $\Delta_r^{k+1} P=0$. Then 
\begin{equation}
\tilde \Delta_r^{(1)} P(x)= \begin{cases}
4 \sum_{i=1}^k L(r)^i \Delta_r^i P(x) \alpha_i(r)  & x \neq F_{jj} q_\ell\ \text{for all }j, \ell,\\
(2+2r) \sum_{i=1}^k L(r)^i \Delta_r^i P(x) \alpha_i(r) & x = F_{jj} q_\ell\ \text{for some }j, \ell. \end{cases} \end{equation}
\end{lemma}

\begin{proof} Without loss of generality, suppose $x = F_{j_1 j_2} q_0$, this can always be accomplished by composing $P$ with a rotation. Because the monomials form a basis for the polynomials, 
\begin{align*}
P \circ F_{j_1j_2 } =& ( P \circ F_{j_1j_2})(q_0) P_{0,1}^{(r)} + \partial_{n,r} ( P \circ F_{j_1j_2})(q_0) P_{0,2}^{(r)} + \partial_{T,r} ( P \circ F_{j_1j_2})(q_0)P_{0,3}^{(r)} \\
& + \Delta_r ( P \circ F_{j_1j_2})(q_0) P_{1,1}^{(r)} + \partial_{n,r} \Delta_r ( P \circ F_{j_1j_2})(q_0)  P_{1,2}^{(r)} + \partial_{T,r} \Delta_r ( P \circ F_{j_1j_2})(q_0) P_{1,3}^{(r)} \\
& + \dots   \tag{\stepcounter{equation}\theequation}\\
& + \Delta_r ^k( P \circ F_{j_1j_2})(q_0) P_{k,1}^{(r)} + \partial_{n,r} \Delta_r ^k( P \circ F_{j_1j_2})(q_0)  P_{k,2}^{(r)} + \partial_{T,r} \Delta_r ^k( P \circ F_{j_1j_2})(q_0) P_{k,3}^{(r)}.
\end{align*}
By \eqref{Laplacian scaling}, \eqref{local normal}, and \eqref{local tangential}, (2.11) becomes 
\begin{align*}
P \circ F_{j_1j_2 } =&P(x) P_{0,1}^{(r)} + \lsym(r) \partial_{n,r}^{j_1j_2} P(x) P_{0,2}^{(r)} + \lskew(r) \partial_{T,r} ^{j_1j_2} P(x) P_{0,3}^{(r)} \\
& + L(r) \Delta_r P(x) P_{1,1}^{(r)} + L(r) \lsym(r) \partial_{n,r}^{j_1j_2}  \Delta_r P(x) P_{1,2}^{(r)} + L(r) \lskew (r) \partial_{T,r}^{j_1j_2}  \Delta_r P(x) P_{1,3}^{(r)} \\
& + \dots   \tag{\stepcounter{equation}\theequation}\\
&+L(r)^k \Delta_r ^kP(x) P_{k,1}^{(r)} + L(r)^k \lsym(r) \partial_{n,r}^{j_1j_2}  \Delta_r^k P(x) P_{k,2}^{(r)} + L(r)^k \lskew (r) \partial_{T,r}^{j_1j_2}  \Delta_r^k P(x) P_{k,3}^{(r)}.
\end{align*}
The cell $F_{j_1j_2 } SG$ intersects another cell, $F_{j_3 j_4} SG$ at $x$, so $x = F_{j_3 j_4} q_\ell$. Let $\sigma \in D_3$ be a rotation sending $q_0 \mapsto q_\ell$. Then 
\begin{align*}
P \circ F_{j_3 j_4} \circ \sigma  = & (P \circ F_{j_3 j_4} \circ \sigma )(q_0) P_{0,1}^{(r)} + \partial_{n,r} (P \circ F_{j_3 j_4} \circ \sigma )(q_0)  P_{0,2}^{(r)} + \partial_{T,r} (P \circ F_{j_3 j_4} \circ \sigma )(q_0) P_{0,3}^{(r)} \\
&+ \Delta_r(P \circ F_{j_3 j_4} \circ \sigma )(q_0) P_{1,1}^{(r)} + \partial_{n,r} \Delta_r(P \circ F_{j_3 j_4} \circ \sigma )(q_0)  P_{1,2}^{(r)} + \partial_{T,r}\Delta_r (P \circ F_{j_3 j_4} \circ \sigma )(q_0) P_{1,3}^{(r)} \\
&+ \dots \tag{\stepcounter{equation}\theequation}\\
&\Delta_r^k(P \circ F_{j_3 j_4} \circ \sigma )(q_0) P_{k,1}^{(r)} + \partial_{n,r} \Delta_r^k(P \circ F_{j_3 j_4} \circ \sigma )(q_0)  P_{k,2}^{(r)} + \partial_{T,r}\Delta_r^k (P \circ F_{j_3 j_4} \circ \sigma )(q_0) P_{k,3}^{(r)} \\
=&P(x) P_{0,1}^{(r)} + \lsym(r) \partial_{n,r}^{j_3j_4} P(x) P_{0,2}^{(r)} + \lskew(r) \partial_{T,r} ^{j_3j_4} P(x) P_{0,3}^{(r)} \\
& + L(r) \Delta_r P(x) P_{1,1}^{(r)} + L(r) \lsym(r) \partial_{n,r}^{j_3j_4}  \Delta_r P(x) P_{1,2}^{(r)} + L(r) \lskew (r) \partial_{T,r}^{j_3j_4}  \Delta_r P(x) P_{1,3}^{(r)} \\
& +\dots \\
&+L(r)^k \Delta_r ^kP(x) P_{k,1}^{(r)} + L(r)^k \lsym(r) \partial_{n,r}^{j_3j_4}  \Delta_r^k P(x) P_{k,2}^{(r)} + L(r)^k \lskew (r) \partial_{T,r}^{j_3j_4}  \Delta_r^k P(x) P_{k,3}^{(r)}.
\end{align*}
The values of $P$ at the neighbors of $x$ are obtained by substituting $q_1$ and $q_2$ for $x$ in (2.12) and (2.13). Therefore if all edges which meet at $x$ have equal resistance, that is, $x \neq F_{jj}q_\ell $ for all $j,\ell$, then 
\begin{equation} 
\tilde \Delta _r^{(1)} P(x) = (P \circ F_{j_1j_2} )(q_1) + (P\circ F_{j_1 j_2} ) (q_2 ) + (P \circ F_{j_3 j_4} \circ \sigma )(q_1)+ (P \circ F_{j_3 j_4} \circ \sigma )(q_2)- 4 P(x) \end{equation}
By lemma 2.2, $P_{i,1}^{(r)} (q_1) = P_{i,1}^{(r)} (q_2)= \alpha_i(r)$, $P_{i,2} ^{(r)} (q_1) = P_{i,2}^{(r)} (q_2) = \beta_{i}(r) $ and $P_{i,3}^{(r)} (q_1) = - P_{i,3}^{(r)} (q_2) = \gamma_i(r) $, so (2.14) becomes 
\begin{equation} 
\tilde \Delta_r^{(1)} P(x) = \sum_{i=0}^k L(r)^i\left ( 4 \Delta_r^i P(x) \alpha_i(r) + 2  \lsym \partial_{n,r}^{j_1j_2} \Delta^i P(x) \beta_i(r) + 2  \lsym \partial_{n,r}^{j_3j_4} \Delta^iP(x) \beta_i(r) \right )- 4 P(x) 
\end{equation}
since all the terms involving tangential derivatives cancel. Theorem 1.2 implies $\partial_{n,r} ^{j_1j_2}\Delta^i P(x) = - \partial_{n,r} ^ {j_3j_4} \Delta^i P(x) $ so, since $\alpha_0(r) =1$,
$$
\tilde \Delta_r^{(1)} P(x) = 4 \sum_{i=0}^k L(r)^i \Delta_r^i P(x) \alpha_i(r) - 4 P(x) =4 \sum_{i=1}^k L(r)^i \Delta_r ^i P(x) \alpha_i(r) .
$$
Now consider the case where not all resistances of edges meeting at $x$ are equal.  Then by similar reasoning as before invoking theorem 1.2 and lemma 2.2
$$
\tilde \Delta_r ^{(1)} P(x) = (2+2r) \sum_{i=1}^k L(r)^i \Delta_r^i P(x) \alpha_i(r) . \qed
$$
 \end{proof}
 \par 
\vspace{4mm}
Lemma 2.4 is the main workhorse in the computation of the values of monomials. It gives a method that allows the values of $P_{n,k}^{(r)}$ to be computed if the values of the $P_{j,k}^{(r)} $ for $0\leq j<n$ and $\alpha_j(r)$ for $0\leq j \leq n$ are known. \par

\begin{lemma} \begin{equation} \alpha_1(r) = \frac{1}{6}. \end{equation} \end{lemma} 

\begin{proof} Consider the space $SG \vee SG$ where the basepoints are the point $q_0$ in each copy of $SG$. By extending the definition of the energy $ \mathcal{E}_r$ to this space in the natural way and extending the measure $\mu_r$, the weak formulation of the Laplacian can be used to naturally extend the definition of the Laplacian. Define a function $u$ on $SG \vee SG$ which is $P_{1,1}^{(r)}$ on each copy of $SG$. Then $\Delta_r u =1$ identically on $SG \vee SG$. Therefore 
$$ \lim_{m \to \infty} \Delta_r^{(m)} u (q_0) =1.$$
By \eqref{graph laplacian} and lemma 2.3, this says that 
$$ 6 \alpha_1(r) =1,$$ so $\alpha_1(r) =\frac{1}{6}.$ \qed
\end{proof} \par 
\vspace{4mm}

Lemma 2.5 combined with the other initial data $\alpha_0(r) =1$, $\beta_0(r) = -\frac{1}{2}$ and $\gamma_0(r) = \frac{1}{2}$ is sufficient to compute the boundary values of all the monomials using lemma 2.4.  \par 

For simplicity, label the values of $P_{j,k}^{(r)}$ on $V_1 \setminus V_0$, taking advantage of the symmetry and self-similarity of the monomials established in lemmas 2.2 and 2.3, as shown in figure 2.1. \par 

 \begin{figure} 
 \begin{center}
 \begin{tikzpicture}[scale=7]
\draw[ -] (-.5,0) -- (.5,0);
\draw[ -](-.5,0) --(0,.86603);
\draw[ -](.5,0) --(0,.86603);
\draw[ -](.5,0) --(0,.86603);
\draw[ -](0,0) --(.25, .43302);
\draw[ -](0,0) --(-.25, .43302);
\draw[ -](.25, .43302) --(-.25, .43302);
\draw[ -](0, .43302) --(-.125, .64952);
\draw[ -](0, .43302) --(.125, .64952);
\draw[ -](.125, .64952) --(-.125, .64952);
\draw[ -](-.25, 0)-- (-.125,.21651);
\draw[ -](-.25, 0)-- (-.375,.21651);
\draw[ -](-.375,.21651)-- (-.125,.21651);
\draw[ -](.25, 0)-- (.125,.21651);
\draw[ -](.25, 0)-- (.375,.21651);
\draw[ -](.375,.21651)-- (.125,.21651);
\filldraw (.51,.-.05) circle (0pt) node {$\alpha_j(r)$};
\filldraw (-.51,.-.05) circle (0pt) node { $\alpha_j(r)$};
\filldraw (0,.91603) circle (0pt) node { $\delta_{j0}$};
\filldraw (-.52,.91603) circle (0pt) node { $P_{j,1}^{(r)}$};
\filldraw (.23,.69952) circle (0pt) node {$L(r)^j\alpha_{j}(r)$};
\filldraw (-.23,.69952) circle (0pt) node {$L(r)^j\alpha_{j}(r)$};
\filldraw (.3,.45) circle (0pt) node {$a_{j,1}$};
\filldraw (-.3,.45) circle (0pt) node {$a_{j,1}$};
\filldraw (0,.375) circle (0pt) node {$b_{j,1}$};
\filldraw (.425,.25) circle (0pt) node {$c_{j,1}$};
\filldraw (-.425,.25) circle (0pt) node {$c_{j,1}$};
\filldraw (.075,.25) circle (0pt) node {$d_{j,1}$};
\filldraw (-.075,.25) circle (0pt) node {$d_{j,1}$};
\filldraw (.25,.-.05) circle (0pt) node {$e_{j,1}$};
\filldraw (-.25,.-.05) circle (0pt) node { $e_{j,1}$};
\filldraw (0,.-.05) circle (0pt) node {$f_{j,1}$};
\end{tikzpicture}
 \begin{tikzpicture}[scale=7]
\draw[ -] (-.5,0) -- (.5,0);
\draw[ -](-.5,0) --(0,.86603);
\draw[ -](.5,0) --(0,.86603);
\draw[ -](.5,0) --(0,.86603);
\draw[ -](0,0) --(.25, .43302);
\draw[ -](0,0) --(-.25, .43302);
\draw[ -](.25, .43302) --(-.25, .43302);
\draw[ -](0, .43302) --(-.125, .64952);
\draw[ -](0, .43302) --(.125, .64952);
\draw[ -](.125, .64952) --(-.125, .64952);
\draw[ -](-.25, 0)-- (-.125,.21651);
\draw[ -](-.25, 0)-- (-.375,.21651);
\draw[ -](-.375,.21651)-- (-.125,.21651);
\draw[ -](.25, 0)-- (.125,.21651);
\draw[ -](.25, 0)-- (.375,.21651);
\draw[ -](.375,.21651)-- (.125,.21651);
\filldraw (.51,.-.05) circle (0pt) node {$\beta_j(r)$};
\filldraw (-.51,.-.05) circle (0pt) node { $\beta_j(r)$};
\filldraw (0,.91603) circle (0pt) node { $0$};
\filldraw (-.52,.91603) circle (0pt) node { $P_{j,2}^{(r)}$};
\filldraw (.28,.69952) circle (0pt) node {$\lsym L(r)^j\beta_{j}(r)$};
\filldraw (-.28,.69952) circle (0pt) node {$\lsym L(r)^j\beta_{j}(r)$};
\filldraw (.3,.45) circle (0pt) node {$a_{j,2}$};
\filldraw (-.3,.45) circle (0pt) node {$a_{j,2}$};
\filldraw (0,.375) circle (0pt) node {$b_{j,2}$};
\filldraw (.425,.25) circle (0pt) node {$c_{j,2}$};
\filldraw (-.425,.25) circle (0pt) node {$c_{j,2}$};
\filldraw (.075,.25) circle (0pt) node {$d_{j,2}$};
\filldraw (-.075,.25) circle (0pt) node {$d_{j,2}$};
\filldraw (.25,.-.05) circle (0pt) node {$e_{j,2}$};
\filldraw (-.25,.-.05) circle (0pt) node { $e_{j,2}$};
\filldraw (0,.-.05) circle (0pt) node {$f_{j,2}$};

\end{tikzpicture}\\ \vspace{4mm}
 \begin{tikzpicture}[scale=7]
\draw[ -] (-.5,0) -- (.5,0);
\draw[ -](-.5,0) --(0,.86603);
\draw[ -](.5,0) --(0,.86603);
\draw[ -](.5,0) --(0,.86603);
\draw[ -](0,0) --(.25, .43302);
\draw[ -](0,0) --(-.25, .43302);
\draw[ -](.25, .43302) --(-.25, .43302);
\draw[ -](0, .43302) --(-.125, .64952);
\draw[ -](0, .43302) --(.125, .64952);
\draw[ -](.125, .64952) --(-.125, .64952);
\draw[ -](-.25, 0)-- (-.125,.21651);
\draw[ -](-.25, 0)-- (-.375,.21651);
\draw[ -](-.375,.21651)-- (-.125,.21651);
\draw[ -](.25, 0)-- (.125,.21651);
\draw[ -](.25, 0)-- (.375,.21651);
\draw[ -](.375,.21651)-- (.125,.21651);
\filldraw (.51,.-.05) circle (0pt) node {$-\gamma_j(r)$};
\filldraw (-.51,.-.05) circle (0pt) node { $\gamma_j(r)$};
\filldraw (0,.91603) circle (0pt) node { $0$};
\filldraw (-.52,.91603) circle (0pt) node { $P_{j,3}^{(r)}$};
\filldraw (.295,.69952) circle (0pt) node {$-\lskew L(r)^j\gamma_{j}(r)$};
\filldraw (-.285,.69952) circle (0pt) node {$\lskew L(r)^j\gamma_{j}(r)$};
\filldraw (.3,.45) circle (0pt) node {$-a_{j,3}$};
\filldraw (-.3,.45) circle (0pt) node {$a_{j,3}$};
\filldraw (0,.375) circle (0pt) node {$0$};
\filldraw (.425,.25) circle (0pt) node {$-c_{j,3}$};
\filldraw (-.425,.25) circle (0pt) node {$c_{j,3}$};
\filldraw (.075,.25) circle (0pt) node {$-d_{j,3}$};
\filldraw (-.075,.25) circle (0pt) node {$d_{j,3}$};
\filldraw (.25,.-.05) circle (0pt) node {$-e_{j,3}$};
\filldraw (-.25,.-.05) circle (0pt) node { $e_{j,3}$};
\filldraw (0,.-.05) circle (0pt) node {$0$};
\end{tikzpicture}
\caption{The labelling of the values of $P_{j,k}^{(r)}$ on $V_1 $.}
\end{center}
\end{figure}
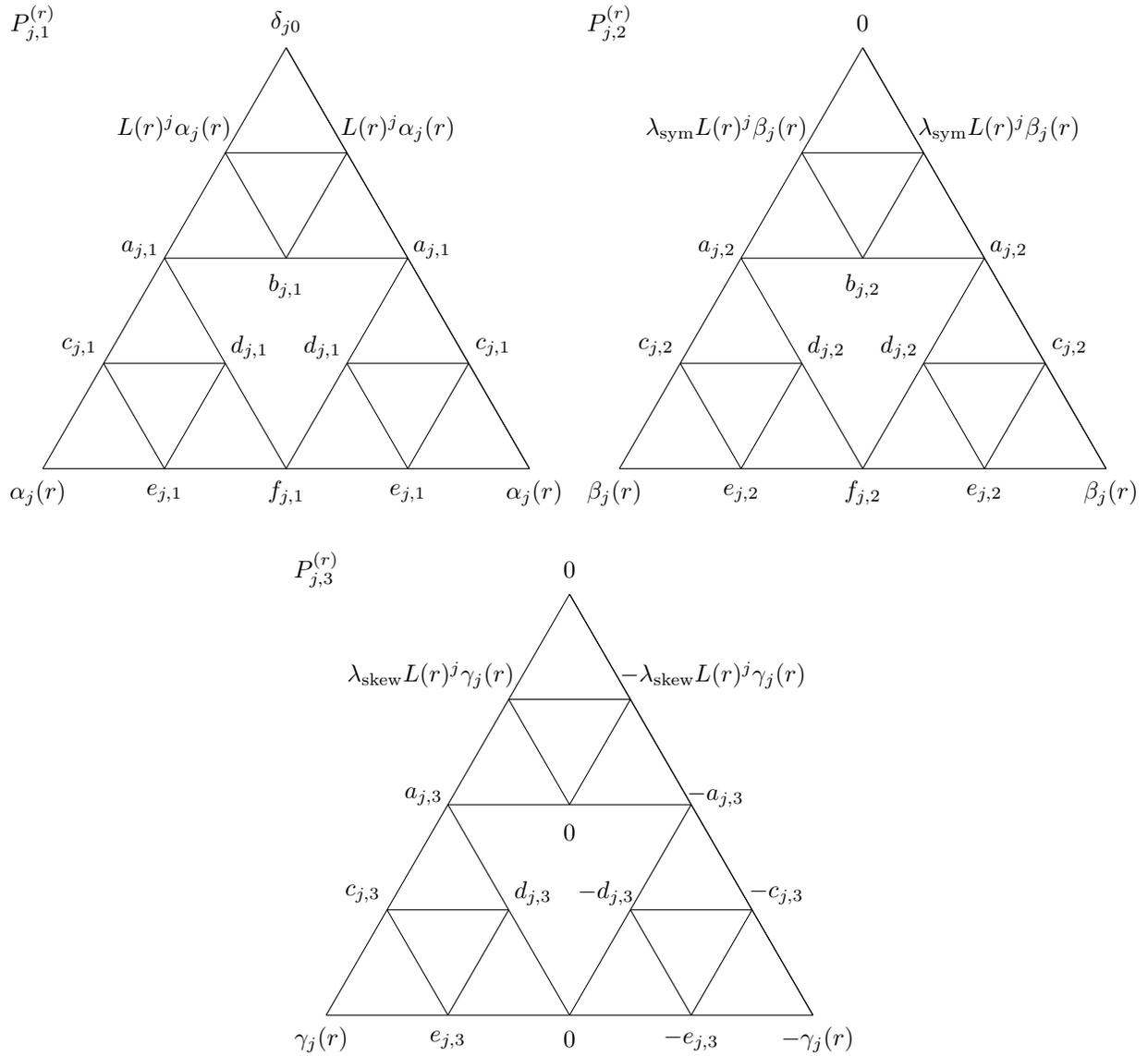

\begin{theorem} The terms of the sequence $\{ \alpha_j(r) \} _{j=0}^\infty$ satisfy the system 
 \begin{align*}
\delta_{j0} + L(r)^j \alpha_j(r)+ra_{j,1} + rb_{j,1} =& (2+2r) \sum_{i=0}^jL(r)^j \alpha_{i}(r) \alpha_{j-i}(r),\\
  L(r)^j \alpha_j(r) + b_{j,1}+c_{j,1}+d_{j,1} =& 4 \sum_{i=0}^j L(r)^i \alpha_{i}(r) a_{j-i,1}, \\
   2L(r)^j \alpha_j(r) + 2 a_{j,1}=&4 \sum_{i=0}^j L(r)^i \alpha_i(r) b_{j-i,1} ,\\
    \alpha_j(r) + r a_{j,1} + r d_{j,1} + e_{j,1}=& (2+2r) \sum_{i=0}^j L(r)^i \alpha_i(r) c_{j-i,1}\tag{\stepcounter{equation}\theequation},\\
     a_{j,1}+c_{j,1}+e_{j,1}+f_{j,1}=&4 \sum_{i=0}^j L(r)^i \alpha_i(r) d_{j-i,1}, \\
      \alpha_{j}(r)+c_{j,1}+rd_{j,1}+r f_{j,1}=& (2+2r)\sum_{i=0}^j L(r)^i \alpha_i(r) e_{j-i,1} ,\\
       2d_{j,1}+2e_{j,1}=& 4 \sum_{i=0}^j L(r)^i \alpha_i(r) f_{j-i,1} ,\\
       \end{align*}
with initial data $$\alpha_0(r) = a_{0,1} = b_{0,1} = c_{0,1} = d_{0,1} = e_{0,1}= f_{0,1}= 1, \quad \alpha_{1}(r) = \frac{1}{6}.$$
\label{alphasystemtheorem}
\end{theorem}
\begin{proof} These equations are obtained by applying lemma 2.4 to evaluate the graph Laplacian $\tilde \Delta_r^{(1)} P_{j,1}^{(r)}$ at each point in $V_1 \setminus V_0$. The initial data have already been established. \qed \end{proof} \par 
\vspace{4mm}
For fixed $j$ and fixed $r$, the system in theorem 2.6 is an affine system of 7 equations in 7 unknowns ($\alpha_j(r), a_{j,1}, b_{j,1}, c_{j,1},d_{j,1}, e_{j,1}, f_{j,1})$, so this system can be solved to recursively compute $\alpha_j(r)$ for $j$ as large as desired. The values $\beta_j(r)$ and $\gamma_j(r)$ can be computed in the same way, but their computation requires the knowledge of the values $\alpha_j(r)$. \par 

\begin{theorem}  The terms of the sequence $\{ \beta_j(r) \} _{j=0}^\infty$ satisfy the system 
\begingroup
\allowdisplaybreaks
 \begin{align*}
\lsym L(r)^j \beta_j(r)+ra_{j,2} + rb_{j,2} =& (2+2r) \sum_{i=0}^j \lsym L(r)^j \alpha_{i}(r) \beta_{j-i}(r),\\
 \lsym  L(r)^j \beta_j(r) + b_{j,2}+c_{j,2}+d_{j,2} =& 4 \sum_{i=0}^j L(r)^i \alpha_{i}(r) a_{j-i,2} ,\\
   2 \lsym L(r)^j \beta_j(r) + 2 a_{j,2}=&4 \sum_{i=0}^j L(r)^i \alpha_i(r) b_{j-i,2}, \\
    \beta_j(r) + r a_{j,2} + r d_{j,2} + e_{j,2}=& (2+2r) \sum_{i=0}^j L(r)^i \alpha_i(r) c_{j-i,2}\tag{\stepcounter{equation}\theequation},\\
     a_{j,2}+c_{j,2}+e_{j,2}+f_{j,2}=&4 \sum_{i=0}^j L(r)^i \alpha_i(r) d_{j-i,2} ,\\
      \beta_{j}(r)+c_{j,2}+rd_{j,2}+r f_{j,2}=& (2+2r)\sum_{i=0}^j L(r)^i \alpha_i(r) e_{j-i,2} ,\\
       2d_{j,2}+2e_{j,2}=& 4 \sum_{i=0}^j L(r)^i \alpha_i(r) f_{j-i,2}, \\
       \end{align*}
       \endgroup
with initial data 
$$
\beta_0(r) = - \frac{1}{2}, \ a_{0,2} = - \frac{ 3r^2 + 8r+4}{9r^2 + 26 r  + 15}, \  b_{0,2}=- \frac{ 3r^2 + 7r+2}{9r^2 + 26 r  + 15}, \ c_{0,2}=  - \frac{ 6r^2 + 19r+15}{2(9r^2 + 26 r  + 15)}, $$$$ d_{0,2} =  - \frac{ 6r^2 + 19r+13}{2(9r^2 + 26 r  + 15)}, \ e_{0,2} =  - \frac{ 6r^2 + 21r+15}{2(9r^2 + 26 r  + 15)},\ f_{0,2} =  - \frac{ 3r^2 + 10r+7}{9r^2 + 26 r  + 15}
.$$
The terms of the sequence $\{ \gamma_j(r) \} _{j=0}^\infty$ satisfy the system 
\begingroup
\allowdisplaybreaks
 \begin{align*}
-\lskew L(r)^j \gamma_j(r)+ra_{j,3} =& (2+2r) \sum_{i=0}^j \lskew L(r)^j \alpha_{i}(r) \gamma_{j-i}(r)\\
  \lskew L(r)^j \gamma_j(r) +c_{j,3}+d_{j,3} =& 4 \sum_{i=0}^j L(r)^i \alpha_{i}(r) a_{j-i,3} ,\\
    \gamma_j(r) + r a_{j,3} + r d_{j,3} + e_{j,3}=& (2+2r) \sum_{i=0}^j L(r)^i \alpha_i(r) c_{j-i,3}\tag{\stepcounter{equation}\theequation},\\
     a_{j,3}+c_{j,3}+e_{j,3}=&4 \sum_{i=0}^j L(r)^i \alpha_i(r) d_{j-i,3} ,\\
      \gamma_{j}(r)+c_{j,3}+rd_{j,3}=& (2+2r)\sum_{i=0}^j L(r)^i \alpha_i(r) e_{j-i,3} ,\\
       \end{align*}
       \endgroup
with initial data
$$
\gamma_0(r) =  \frac{1}{2}, \ a_{0,3} =  \frac{ 2r+3}{9r^2 + 26 r  + 15},  \ c_{0,2}=   \frac{ 9r+15}{2(9r^2 + 26 r  + 15)}, \ d_{0,2} =   \frac{ 5r+9}{2(9r^2 + 26 r  + 15)}, $$$$e_{0,2} =   \frac{ 7r +15}{2(9r^2 + 26 r  + 15)}
.$$
\label{betagammasystemtheorem}
\end{theorem} 
\begin{proof} The initial data follow from (1.39), and again the systems of equations are obtained by applying lemma 2.4 at each junction point in $V_1$. \qed 
\end{proof} \par 
\vspace{4mm}
Again this expresses the values $\beta_j(r)$ and $\gamma_j(r)$ via a system of affine equations with the same number of equations and unknowns, so this gives a method to solve for the boundary values of $P_{j,2}^{(r)} $ and $P_{j,3}^{(r)}$. It can be seen from theorems 2.6 and 2.7 that $\alpha_j(r), \beta_j(r),$ and $\gamma_j(r)$ are rational functions. More properties of these functions are discussed in section 4. \par 
The method to compute the boundary values of monomials established in theorems 2.6 and 2.7 is naive in some sense, requiring the solution of a system of equations at each step to compute the value of the next higher degree monomial at the boundary. While this method has the advantage of also providing the values of the monomials on all points in $V_1$, one might desire explicit recurrence relations for the sequences $\{ \alpha_j(r)\}_{j=0}^\infty$, $\{ \beta_j(r)\}_{j=0}^\infty$, and $\{ \gamma_j(r)\}_{j=0}^\infty$ which don't involve the values at all points in $V_1$ like those which are presented for the monomials with respect to the standard Laplacian in \cite{nsty} and \cite{cq}. It turns out that such relations can indeed be written down. 

\begin{theorem} The terms of the sequence $\{\alpha_j(r)\}_{j=0}^\infty$ satisfy the recurrence relation
$$128(r+1)^3 \sum_{\substack{i_1 + i_2 + i_3 + i_4 + i_5 + i_6 = j\\0\leq i_1,i_2,i_3,i_4,i_5,i_6} }\alpha_{i_1} \alpha_{i_2} \alpha_{i_3} \alpha_{i_4} \alpha_{i_5} \alpha_{i_6}$$
$$- (128r^3+448r^2+512r+192)\sum_{\substack{i_1 + i_2 + i_3 + i_4 + i_5  = j\\0\leq i_1,i_2,i_3,i_4,i_5} }\alpha_{i_1} \alpha_{i_2} \alpha_{i_3} \alpha_{i_4} \alpha_{i_5} 
$$\begin{equation}+(-72r^3-88r+8r+24)\sum_{\substack{i_1 + i_2 + i_3 + i_4 = j\\0\leq i_1,i_2,i_3,i_4} }\alpha_{i_1} \alpha_{i_2} \alpha_{i_3} \alpha_{i_4} +(64r^3+164r^2+144r+60)\sum_{\substack{i_1 + i_2 + i_3  = j\\0\leq i_1,i_2,i_3} }\alpha_{i_1} \alpha_{i_2} \alpha_{i_3}
\end{equation}$$+(14r^3+2r^2-10r-18)\sum_{i=0}^{j}\alpha_{i}   \alpha_{j-i}-(6r^3+9r^2+3) \alpha_j-(8r^2+8r)\sum_{\substack{i_1 + i_2 + i_3  = j\\0\leq i_1,i_2,i_3} } \frac{\alpha_{i_1} \alpha_{i_2} \alpha_{i_3}}{L(r)^{i_3}}
$$$$-4r \sum_{i=0}^j \frac{ \alpha_i \alpha_{j-i}}{L(r)^{j-i}} +2r^2 \frac{\alpha_j}{L(r)^j}=0.$$

The terms of the sequence $\{\beta_j(r)\}_{j=0}^\infty$ satisfy the recurrence relation
$$ (384r^4 +1920r^3 +3456r^2 +2688r +768) \sum_{\substack{i_1 + i_2 + i_3 + i_4 + i_5 + i_6 = j\\0\leq i_1,i_2,i_3,i_4,i_5,i_6} }\alpha_{i_1} \alpha_{i_2} \alpha_{i_3} \alpha_{i_4} \alpha_{i_5} \beta_{i_6}
$$$$-(384r^4+2112r^3+4224r^2+3648r+1152)  \sum_{\substack{i_1 + i_2 + i_3 + i_4 + i_5 = j\\0\leq i_1,i_2,i_3,i_4,i_5} }\alpha_{i_1} \alpha_{i_2} \alpha_{i_3} \alpha_{i_4} \beta_{i_5} 
$$\begin{equation}+(-216r^4-504r^3+264r^2+1080r+528) \sum_{\substack{i_1 + i_2 + i_3 + i_4  = j\\0\leq i_1,i_2,i_3,i_4} }\alpha_{i_1} \alpha_{i_2} \alpha_{i_3} \beta_{i_4}
\label{beta recurrence}\end{equation}$$+(192r^4+684r^3+648r^2+84r-24) \sum_{\substack{i_1 + i_2 + i_3 = j\\0\leq i_1,i_2,i_3} }\alpha_{i_1} \alpha_{i_2} \beta_{i_3}
+(42r^4+54r^3-42r^2+18r-36)\sum_{i=0}^{j}\alpha_{i}   \beta_{j-i}
$$$$+(-18r^4-15r^3+30r^2-21r+6)\beta_j
-(36r^3+140 r^2+164 r+60) \sum_{\substack{i_1 + i_2 + i_3  = j\\0\leq i_1,i_2,i_3} }\frac{\alpha_{i_1} \alpha_{i_2}  \beta_{i_3}}{L(r)^{i_3}}
$$$$-(18r^2+52r+30)\sum_{i=0}^j \frac{ \alpha_i \beta_{j-i}}{L(r)^{j-i}}+(9r^3+26r^2+15r)\frac{\beta_j}{L(r)^j}=0.$$

The terms of the sequence $\{\gamma_j(r)\}_{j=0}^\infty$ satisfy the recurrence relation
\begin{equation}(9r^2+26r+15)\frac{\gamma_j}{L(r)^j}- 32(r+1)^2 \sum_{\substack{i_1 + i_2 + i_3 + i_4= j\\0\leq i_1,i_2,i_3,i_4} }\alpha_{i_1} \alpha_{i_2} \alpha_{i_3}\gamma_{i_4}+16(r+1)^2 \sum_{\substack{i_1 + i_2 + i_3= j\\0\leq i_1,i_2,i_3} }\alpha_{i_1} \alpha_{i_2} \gamma_{i_3} \label{gamma recurrence}
\end{equation}$$+(10r^2+4r+2)\sum_{i=0}^{j}\alpha_{i}   \gamma_{j-i}-(3r^2-2r+1)\gamma_j=0.$$

\end{theorem}

\begin{proof} First consider the sequence $\{\alpha_j(r) \} _{j=0}^\infty$. Introduce infinite semicirculant matrices $\boldsymbol{\alpha}$, $A_1$, $B_1$, $C_1$, $D_1$, $E_1$, $F_1$ with 
$$
\boldsymbol {\alpha}_{ij}= \begin{cases} \alpha_{i-j}(r) & i - j \geq 0, \\ 0 & i-j <0,
\end{cases} \ 
(A_1)_{ij}= \begin{cases} \frac{a_{i-j,1}}{L(r)^{i-j}} & i - j \geq 0 ,\\ 0 & i-j <0,
\end{cases} \ 
(B_1)_{ij}= \begin{cases} \frac{b_{i-j,1}}{L(r)^{i-j}}& i - j \geq 0 ,\\ 0 & i-j <0,
\end{cases}$$

\begin{equation}(C_1)_{ij}= \begin{cases} \frac{c_{i-j,1}}{L(r)^{i-j}}& i - j \geq 0 ,\\ 0 & i-j <0,
\end{cases} \ 
(D_1)_{ij}= \begin{cases} \frac{d_{i-j,1}}{L(r)^{i-j}}& i - j \geq 0 ,\\ 0 & i-j <0,
\end{cases} \ 
(E_1)_{ij}= \begin{cases} \frac{e_{i-j,1}}{L(r)^{i-j}}& i - j \geq 0, \\ 0 & i-j <0,
\end{cases} \end{equation}
$$(F_1)_{ij}= \begin{cases}\frac{f_{i-j,1}}{L(r)^{i-j}} & i - j \geq 0 ,\\ 0 & i-j <0.
\end{cases} $$
Let $I$ be the infinite identity matrix, with 1 on the diagonal and 0 elsewhere. Let $\bf{L}$ be the operator on such matrices defined by 
\begin{equation} 
\bf{L} \begin{pmatrix} x_0 & 0 & 0 &0& \dots \\
x_1 & x_0 & 0 &0& \dots \\
x_2 & x_1 & x_0&0 & \dots \\ 
x_3 & x_2 & x_1 & x_0  & \dots\\
\vdots & \vdots& \vdots  & \vdots & \ddots \\
\end{pmatrix} = 
\begin{pmatrix} x_0 & 0 & 0&0 &\dots \\
L(r)^{-1} x_1 & x_0 & 0& 0 &\dots \\
L(r)^{-2} x_2 & L(r)^{-1}x_1 & x_0 &0& \dots \\
L(r)^{-3} x_3 & L(r)^{-2} x_2 & L(r)^{-1} x_1 & x_0 & \dots \\
\vdots & \vdots & \vdots & \vdots & \ddots \end{pmatrix}.
\end{equation}
In this notation, (2.17) becomes 
 \begin{align*}
I+  \boldsymbol{\alpha}+rA_1 + rB_1 =& (2+2r)  \boldsymbol{\alpha}^2,\\
 \boldsymbol{\alpha} + B_1+C_1+D_1 =& 4\boldsymbol{\alpha} A_1,\\
   2 \boldsymbol{\alpha} + 2A_1=&4 \boldsymbol{\alpha} B_1 ,\\
  \textbf{L} \boldsymbol{\alpha} + r A_1 + r D_1 + E_1=& (2+2r)  \boldsymbol{\alpha} C_1\tag{\stepcounter{equation}\theequation},\\
     A_1 + C_1+ E_1+F_1 =&4  \boldsymbol{\alpha} D_1,\\
     \textbf{L} \boldsymbol{\alpha} +C_1+rD_1+r F_1=& (2+2r) \boldsymbol{\alpha} E_1,\\
     2D_1 + 2E_1=& 4 \boldsymbol{\alpha}F_1. \\
       \end{align*}
Manipulating these equations to obtain an equation in only the variables $\boldsymbol{\alpha}$ and $ \textbf{L} \boldsymbol{\alpha}$ yields the equation 
$$
128(r+1)^3 \boldsymbol{\alpha}^6
- (128r^3+448r^2+512r+192)\boldsymbol{\alpha}^5 +(-72r^3-88r+8r+24)\boldsymbol{\alpha}^4
+(64r^3+164r^2+144r+60)\boldsymbol{\alpha}^3$$\begin{equation}
+(14r^3+2r^2-10r-18)\boldsymbol{\alpha}^2-(6r^3+9r^2+3) \boldsymbol{\alpha} -(8r^2+8r)\boldsymbol{\alpha}^2 \textbf{L} \boldsymbol {\alpha}-4r \boldsymbol{\alpha} \textbf{L} \boldsymbol{\alpha} +2r^2 \textbf{L} \boldsymbol{\alpha}=0.\end{equation}
Rewriting (2.26) as an equation involving the entries of the involved matrices using the fact that 
$$ \left (  \begin{pmatrix} x_0 & 0 & 0 &0& \dots \\
x_1 & x_0 & 0 &0& \dots \\
x_2 & x_1 & x_0&0 & \dots \\ 
x_3 & x_2 & x_1 & x_0  & \dots\\
\vdots & \vdots& \vdots  & \vdots & \ddots \\
\end{pmatrix}  \begin{pmatrix} y_0 & 0 & 0 &0& \dots \\
y_1 & y_0 & 0 &0& \dots \\
y_2 & y_1 & y_0&0 & \dots \\ 
y_3 & y_2 & y_1 & y_0  & \dots\\
\vdots & \vdots& \vdots  & \vdots & \ddots \\  
\end{pmatrix} \right ) _{ij} = \begin{cases}  \sum_{k=0}^{i-j} x_k y_{i-j-k} & i - j \geq 0, \\ 0 & i-j < 0 \end{cases}$$
gives (2.20). \par 
To obtain the recurrence for $\{\beta_j(r)\}_{j=0}^\infty,$ introduce matrices 
$\boldsymbol{\beta}, A_2, B_2, C_2, D_2, E_2, F_2$ with 

$$\boldsymbol {\beta}_{ij}= \begin{cases} \beta_{i-j}(r) & i - j \geq 0, \\ 0 & i-j <0,
\end{cases} \ 
(A_2)_{ij}= \begin{cases} \frac{a_{i-j,2}}{L(r)^{i-j} }& i - j \geq 0, \\ 0 & i-j <0,
\end{cases} \ 
(B_2)_{ij}= \begin{cases} \frac{b_{i-j,2}}{L(r)^{i-j} } & i - j \geq 0 ,\\ 0 & i-j <0,
\end{cases} \ $$
\begin{equation}(C_2)_{ij}= \begin{cases} \frac{c_{i-j,2}}{L(r)^{i-j} } & i - j \geq 0 ,\\ 0 & i-j <0,
\end{cases} \ 
(D_2)_{ij}= \begin{cases} \frac{d_{i-j,2}}{L(r)^{i-j} } & i - j \geq 0 ,\\ 0 & i-j <0,
\end{cases} \
(E_2)_{ij}= \begin{cases} \frac{e_{i-j,2}}{L(r)^{i-j} } & i - j \geq 0 ,\\ 0 & i-j <0,
\end{cases} \end{equation}
$$(F_2)_{ij}= \begin{cases} \frac{f_{i-j,2}}{L(r)^{i-j} }& i - j \geq 0 ,\\ 0 & i-j <0.
\end{cases} $$
Then (2.18) becomes 
 \begin{align*}
\lsym(r) \boldsymbol{\beta}+rA_2 + rB_2 =& (2+2r) \lsym(r) \boldsymbol{\alpha} \boldsymbol \beta ,\\
\lsym(r) \boldsymbol{\beta} + B_2+C_2+D_2 =& 4\boldsymbol{\alpha} A_2,\\
   2  \lsym(r) \boldsymbol{\beta} + 2A_2=&4 \boldsymbol{\alpha} B_2 ,\\
  \textbf{L} \boldsymbol{\beta} + r A_2 + r D_2 + E_2=& (2+2r)  \boldsymbol{\alpha} C_2\tag{\stepcounter{equation}\theequation},\\
     A_2 + C_2+ E_2+F_2 =&4  \boldsymbol{\alpha} D_2,\\
     \textbf{L} \boldsymbol{\beta} +C_2+rD_2+r F_2=& (2+2r) \boldsymbol{\alpha} E_2,\\
     2D_2 + 2E_2=& 4 \boldsymbol{\alpha}F_2. \\
       \end{align*}
This system yields an equation relating $\boldsymbol \beta , \textbf{L} \boldsymbol \beta ,\boldsymbol \alpha$: 
$$
( 384r^4 + 1920 r^3 + 3456r^2 + 2688r + 768 ) \boldsymbol \alpha^5  \boldsymbol \beta
+ -(384r^4+2112r^3+4224r^2+3648r+1152) \boldsymbol \alpha^4  \boldsymbol \beta 
$$\begin{equation}
+(-216 r^4 -504r^3 +264r^2 +1080r + 528 ) \boldsymbol \alpha^3 \boldsymbol \beta
+ ( 192 r^4 + 684 r^3 + 648r^2 +84r -24) \boldsymbol \alpha^2 \boldsymbol \beta
\label{beta relation} \end{equation} $$
+ ( 42r^4 + 54r^3 -42r^2 +18r-36) \boldsymbol \alpha \boldsymbol \beta 
+ (-18r^4-15r^3 +30r^2 -21r + 6 ) \boldsymbol \beta 
- (36r^3 + 140r^2 + 164r + 60 ) \boldsymbol \alpha^2 \textbf{L} \boldsymbol \beta 
$$$$
- (18r^2 + 52r+30 ) \boldsymbol \alpha \textbf{L} \boldsymbol \beta 
+ (9r^3 + 26r^2 +15r) \textbf{L} \boldsymbol \beta =0.
$$
Rewriting \eqref{beta relation} in terms of the entries of the matrices yields \eqref{beta recurrence}. \par 

Finally, the same method gives the recurrence relation for $\{ \gamma_j(r)\} _{j=0}^ \infty .$ Define matrices 
$ \boldsymbol \gamma, A_3, C_3, D_3, E_3 $ with 
\begin{equation}\boldsymbol {\gamma}_{ij}= \begin{cases} \gamma_{i-j}(r) & i - j \geq 0 ,\\ 0 & i-j <0,
\end{cases} \ 
(A_3)_{ij}= \begin{cases} \frac{a_{i-j,3}}{L(r)^{i-j} }& i - j \geq 0 ,\\ 0 & i-j <0,
\end{cases} \ 
(C_3)_{ij}= \begin{cases} \frac{c_{i-j,3}}{L(r)^{i-j} } & i - j \geq 0 ,\\ 0 & i-j <0,
\end{cases} \end{equation}
$$(D_3)_{ij}= \begin{cases} \frac{d_{i-j,3}}{L(r)^{i-j} } & i - j \geq 0 ,\\ 0 & i-j <0,
\end{cases} \ 
(E_3)_{ij}= \begin{cases} \frac{e_{i-j,3}}{L(r)^{i-j} } & i - j \geq 0 ,\\ 0 & i-j <0.
\end{cases} $$

In this notation, (2.19) becomes 

\begin{align*}
-\lskew(r) \boldsymbol{\gamma}+rA_3  =& (2+2r) \lskew(r) \boldsymbol{\alpha} \boldsymbol \gamma, \\
\lskew(r) \boldsymbol{\gamma} +C_3+D_3 =& 4\boldsymbol{\alpha} A_3,\\
  \textbf{L} \boldsymbol{\gamma} + r A_3 + r D_3 + E_3=& (2+2r)  \boldsymbol{\alpha} C_3\tag{\stepcounter{equation}\theequation},\\
     A_3 + C_3+ E_3 =&4  \boldsymbol{\alpha} D_3,\\
     \textbf{L} \boldsymbol{\gamma} +C_3+rD_3=& (2+2r) \boldsymbol{\alpha} E_3.\\
          \end{align*}
          
          These equations imply the relation 
          
          \begin{equation} \label{gamma relation} 
          (9r^2 + 26 r + 15 ) \textbf L \boldsymbol \gamma - 32 (r+1)^2 \boldsymbol \alpha^3 + 16 (r+1)^2 \boldsymbol \alpha^2 + (10r^2+4r +2 ) \boldsymbol \alpha - ( 3r^2 -2r + 1 ) \boldsymbol \gamma =0 .
          \end{equation} 
          Writing \eqref{gamma relation} in terms of the entries of the matrices gives \eqref{gamma recurrence}.
\qed \end{proof} \par \vspace{4mm}

This section concludes with the calculation of the normal and tangential derivatives of the monomials at the boundary point $q_1$. This computation is necessary to determine the value of a polynomial at points in $V_\ast \setminus V_1$ since extending a polynomial from $V_m$ to $V_{m+1}$ requires knowledge of the jet at each point in $V_m$. Let $\partial_{n,r} P_{j,k}^{(r)} (q_1) = n_{j,k} (r) $ and $\partial_{T,r} P_{j,k}^{(r)} (q_1 ) = t_{j,k} (r). $ Knowledge of the jets of monomials at $q_1$ and $q_2$ is also necessary for changing between various bases for the space of polynomials as in \cite{nsty} and for constructing multiharmonic splines such as in \cite{su}. \par 

\begin{theorem} The normal and tangential derivatives of $P_{j,k} ^{(r)}$ at $q_1$ are given by 
\begin{equation} 
n_{j,k} (r) = \begin{cases}  \label{normal derivative values}
2 \sum_{i=0} ^ j \alpha _i(r) \alpha_{j-i}(r)  + 2 \sum_{i=0}^{j-1} n _{i,1} (r) \beta_{j-i}(r) - \alpha_j(r) - \delta_{j0} & k= 1,\\
2 \sum_{i=0} ^ j \alpha_i (r) \beta_{j-i}(r) + 2 \sum_{ i = 0 }^{j-1} n_{i,2} (r) \beta_{j-i}(r) - \beta_j(r) & k=2, \\
2 \sum_{ i=0}^{j-1} n_{i,3} (r) \beta_{j-i}(r) + \gamma_j(r) & k=3, 
\end{cases}
\end{equation}

\begin{equation} \label{tangential derivative values}
t_{j,k} (r)  = \begin{cases} 
\alpha_j(r) - \delta_{j0} - 2 \sum_{i=0}^{j-1} t_{i,1} (r)\gamma_{j-i} (r) & k=1 ,\\
\beta_j(r) - 2 \sum_{ i=0}^{j-1} t_{i,2} (r) \gamma_{j-i}(r) & k=2 ,\\
-\gamma_{ j}(r) - 2 \sum_{ i=0} ^{j-1} t_{i,3}(r) \gamma_{j-i}(r) &k=3.
\end{cases}
\end{equation}

\end{theorem}
\begin{proof} Let $\rho \in D_3$ be the rotation sending $q_i \mapsto q_{i+1} $. Then $\partial_{n,r} P_{j,k}^{(r)} (q_1) = \partial_{n,r} (P_{j,k} ^{(r)} \circ \rho ) (q_0)$ and $\partial_{T,r} P_{j,k}^{(r)} (q_1) = \partial_{T,r} (P_{j,k} ^{(r)} \circ \rho ) (q_0).$ Therefore the values of $n_{j,k}(r) $ and $t_{j,k} (r) $ are determined by the expression of $P_{j,k}^{(r)}  \circ \rho$ in the monomial basis. Written as a linear combination of monomials, 
\begin{align*}
P_{j,k}^{(r)}  \circ \rho=
& (P_{j,k}^{(r)}  \circ \rho)(q_0) P_{0,1}^{(r)} + \partial_{n,r} (P_{j,k}^{(r)}  \circ \rho)(q_0) P_{0,2}^{(r)} + \partial_{T,r} ( P_{j,k}^{(r)}  \circ \rho)(q_0)P_{0,3}^{(r)} \\
& + (P_{j-1,k}^{(r)}  \circ \rho)(q_0) P_{1,1}^{(r)} + \partial_{n,r}  ( P_{j-1,k}^{(r)}  \circ \rho)(q_0)  P_{1,2}^{(r)} + \partial_{T,r} ( P_{j-1,k}^{(r)}  \circ \rho)(q_0) P_{1,3}^{(r)} \\
& + \dots   \tag{\stepcounter{equation}\theequation}\\
& + (P_{0,k}^{(r)}  \circ \rho)(q_0) P_{j,1}^{(r)} + \partial_{n,r} ( P_{0,k}^{(r)}  \circ \rho)(q_0)  P_{j,2}^{(r)} + \partial_{T,r} ( P_{0,k}^{(r)}  \circ \rho)(q_0) P_{j,3}^{(r)} \\
=& P_{j,k}^{(r)} (q_1) P_{0,1} ^{(r)} + n_{j,k}(r)P_{0,2}^{(r)} + t_{j,k}(r) P_{0,3}^{(r)} \\  
& + P_{j-1,k}^{(r)} (q_1)P_{1,1}^{(r)} + n_{j-1,k}(r)P_{1,2}^{(r)} + t_{j-1,k}(r) P_{1,3}^{(r)} \\
& + \dots \\
& +  P_{0,k}^{(r)} (q_1)P_{j,1}^{(r)} + n_{0,k}(r)P_{j,2}^{(r)} + t_{0,k}(r) P_{j,3}^{(r)}.
\end{align*}
Therefore 
\begin{equation}
(P_{j,k}^{(r)}  \circ \rho)(q_1) =
 \sum_{i=0}^j (P_{i,k}^{(r)}  (q_1) \alpha_{j-i}(r)  +n_{i,k}(r)  \beta_{j-i}(r)  + t_{i,k}(r) \gamma_{j-i}(r) ),
\end{equation}
and 
\begin{equation}
(P_{j,k}^{(r)}  \circ \rho)(q_2) =
 \sum_{i=0}^j (P_{i,k}^{(r)}  (q_2) \alpha_{j-i}(r)  +n_{i,k}(r)  \beta_{j-i}(r)  - t_{i,k}(r) \gamma_{j-i}(r) ).
 \end{equation}
 Adding and subtracting these two equations yields 
 \begin{equation} 
 (P_{j,k}^{(r)}  \circ \rho)(q_1)+ (P_{j,k}^{(r)}  \circ \rho)(q_2) = \sum_{i=0}^j (2n_{i,k} (r)\beta_{j-i}(r) +  \alpha_{j-i}(r) (P_{i,k} ^{(r)} (q_1) + P_{i,k}^{(r)} (q_2) )), 
 \end{equation} 
  \begin{equation} 
 (P_{j,k}^{(r)}  \circ \rho)(q_1)-(P_{j,k}^{(r)}  \circ \rho)(q_2) = 2\sum_{i=0}^j t_{i,k} (r) \gamma_{j-i}(r) .
 \end{equation} 
Since $(P_{j,k}^{(r)}  \circ \rho)(q_1)= P_{j,k} ^{(r)} (q_2)$ and $(P_{j,k} ^{(r)} \circ \rho ) (q_2) = P_{j,k} ^{(r)} (q_0) $, applying lemma 2.2 shows that 
\begin{equation}  \alpha_j(r) + \delta _{j0} = 2 \sum_{i=0}^j( \alpha_{i}(r)  \alpha_{j-i} (r) + n_{i,1}(r) \beta_{j-i}(r)) , 
\end{equation} 
\begin{equation}  \beta_j(r)  = 2 \sum_{i=0}^j( \alpha_{i}(r)  \beta_{j-i} (r) + n_{i,2}(r) \beta_{j-i}(r) ), 
\end{equation} 
\begin{equation}  -\gamma_j(r)  = 2 \sum_{i=0}^j n_{i,3}(r) \beta_{j-i}(r) , 
\end{equation} 
\begin{equation} \label{alpha gamma relation}
\alpha_j(r) - \delta_{j0} = 2 \sum_{i=0}^j t_{i,1} (r) \gamma _{j-i}(r) , 
\end{equation}
\begin{equation}
\beta_j(r)  = 2 \sum_{i=0}^j t_{i,2} (r) \gamma _{j-i}(r) , 
\end{equation}
\begin{equation}
-\gamma_j(r) = 2 \sum_{i=0}^j t_{i,3} (r) \gamma _{j-i}(r).
\end{equation}
Recalling that $\beta_{0}(r) = - \frac12$ and $\gamma_0(r) = \frac12$ and solving for $n_{j,k}(r),$ $t_{j,k}(r)$ gives \eqref{normal derivative values} and \eqref{tangential derivative values}.
\qed
\end{proof}

\section{Properties of Polynomials}
This section discusses the generalizations to the family of self-similar symmetric Laplacians of several results of \cite{nsty} regarding general polynomials, monomials, and the derivatives of the monomials at the boundary. Some of these results are purely consequences of the symmetry of the standard Laplacian and hence continue to hold for any value of $r$, while some results only hold for $r=1$. \par 

Let $R_0, R_1, R_2\in D_3$ be the reflections of $SG$ which fix $q_0, q_1, q_2$, respectively. Let $\rho \in D_3$ be the rotation sending $q_i \mapsto q_{i+1}$.\par 
\begin{theorem} For all $j \geq 0$, and for all $0 <r < \infty$,
\begin{equation}P_{j,3}^{(r)} (x) + P_{j,3}^{(r)} (\rho x)+P_{j,3}^{(r)} (\rho^{-1}x)=0.\end{equation} \end{theorem}
\begin{proof} For $j=0,$ the function 
$$ P_{0,3}^{(r)}+ P_{0,3}^{(r)}\circ \rho+P_{0,3}^{(r)}\circ \rho^{-1}$$ is harmonic, and by the skew-symmetry of $P_{j,3}^{(r)},$ the value of this function on each boundary point is zero, hence 
$$ P_{0,3}^{(r)}+ P_{0,3}^{(r)}\circ \rho+P_{0,3}^{(r)}\circ \rho^{-1}= 0$$ identically.
Assume that 
$$P_{j-1,3}^{(r)} (x) + P_{j-1,3}^{(r)} (\rho x)+P_{j-1,3}^{(r)} (\rho^{-1}x)=0.$$
Then 
$$\Delta_r (  P_{j,3}^{(r)}+ P_{j,3}^{(r)}\circ \rho +P_{j,3}^{(r)}\circ \rho ^{-1})= P_{j-1,3}^{(r)} + P_{j-1,3}^{(r)} \circ \rho +P_{j-1,3}^{(r)} \circ \rho^{-1}=0$$
by the symmetry of $\Delta_r$, so $$ P_{j,3}^{(r)}+ P_{j,3}^{(r)}\circ \rho +P_{j,3}^{(r)}\circ \rho ^{-1}$$ is harmonic and vanishes on the boundary, hence it vanishes identically. \qed \end{proof}  \par \vspace{4mm} 

\begin{theorem} Any polynomial $P$, with $\Delta_r^{n+1} P=0$ for some $n$, satisfies the identity 
\begin{equation}P(x) + P ( \rho x ) + P(\rho^{-1} x ) = P(R_0x) + P(R_1x) + P (R _2 x) ,\end{equation}
and the local version 
\begin{equation} P(x_0) + P(x_1) + P(x_2) = P (y_0) + P ( y_1) + P (y_2 ) \end{equation}
where $$x_ 0 = F_w x, \ x_1 = F_w \rho x ,\ x_2 = F_w \rho ^{-1} x,\ y_0 = F_w R_0 x ,\ y_1 = F_w R_1x ,\ y_2 = F_w R_2 x .$$ \end{theorem}

\begin{proof} It suffices to consider the case where $P$ is a monomial. If $P = P_{j,1}^{(r)} $ or $P= P_{j,2}^{(r)}$, then $P(x) = P(R_0 x )$, so $P( \rho x) = P (R_1 x ),$ and $P( \rho^{-1} x ) = P( R_2 x )$ since 
$R_0 \circ \rho = R_1$ and $ R_0 \circ \rho^{-1} = R_2$. If $P = P_{j,3}^{(r)} ,$ then $P(x) =- P(R_0 x )$, $P( \rho x) = -P (R_1 x ),$ and $P( \rho^{-1} x ) = -P( R_2 x )$, so \begin{equation} P(x) + P(\rho x) + P (\rho^{-1} x ) = - (P(R_0x) + P(R_1x) + P (R _2 x)), \end{equation} but by theorem 3.1 the left hand side is zero, so 
$$P(x) + P ( \rho x ) + P(\rho^{-1} x ) = P(R_0x) + P(R_1x) + P (R _2 x).$$
The local version follows from the fact that $P \circ F_w $ is a polynomial. \qed \end{proof} \par \vspace{4mm}

\begin{corollary} Any polynomial $P$ satisfies 
$$ \partial_{T,r} P (q_0) + \partial_{T,r} P(q_1) + \partial_{T,r} P(q_2) =0.$$
More generally, the sum of the tangential derivatives at the boundary points of any cell vanishes. \end{corollary}
\begin{proof} Let $x = F_{00}^m(q_1)$ and apply the previous theorem to obtain 
$$ P (F_{00}^m q_1 ) + P (F_{11}^mq_2 ) + P(F_{22}^mq_0) - (P(F_{00}^m q_2 ) + P (F_{11}^m q_0 ) + P ( F_{22}^m q_1) ) = 0 .$$ 
Rearranging terms gives 
$$ P (F_{00}^m q_1 ) - P(F_{00}^m q_2 )+ P (F_{11}^mq_2 )-P (F_{11}^m q_0 ) + P(F_{22}^mq_0)  - P ( F_{22}^m q_1)  = 0,$$
and multiplying by $(\lambda_{\text{skew}}(r))^{-m}$ and taking the limit as $m \to \infty$ shows 
$$\partial_{T,r} P (q_0) + \partial_{T,r} P(q_1) + \partial_{T,r} P(q_2) =0.\ \qed$$ \end{proof} \par \vspace{4mm}

The skew-symmetry of $P_{j,3}^{(r)}$ implies that $\partial_{T,r} P_{j,3}^{(r)}(q_1) = \partial_{T,r} P_{j,3}^{(r)} (q_2) $, and combining this fact with corollary 3.3 implies that for $j\geq 1$, 
\begin{equation} 
\partial_{T,r} P_{j,3}^{(r)} (q_1) = t_{j,3}(r) = 0.
\end{equation} 
This fact can also be seen directly from \eqref{tangential derivative values}, by substituting $\gamma_0(r) = \frac12 $ and $t_{0,3} (r) = - \frac{1}{2} $ and inductively assuming that $t_{i,3}(r) =0$ for $1 \leq i <j$. 

These identities are the direct extensions of the analogous results for the standard polynomials since their proofs use nothing other than the symmetries of the Laplacian and the monomials. The next result is again the natural generalization of a result about the standard polynomials and is reminiscent of the property of the monomials on the real line that the derivative of $\frac{x^n}{n!}$ at $x=1$ is the value of the monomial of degree $n-1$ at $x=1$. 
\begin{theorem} For $j\geq 1$,
\begin{equation} 
n_{j,2} (r) = - \alpha_j(r) .
\end{equation}
\end{theorem} 

\begin{proof} For $j=1$, \eqref{normal derivative values} says 
$$ n_{1,2}(r) = 2 \alpha_0(r) \beta_1(r) +2 \alpha_1(r) \beta_0(r)  + 2 n_{0,2} (r) \beta_1(r) - \beta_{1}(r). $$
Substituting $ \alpha_0(r) = 1, \beta_0(r) = -\frac{1}{2} , n_{0,2} (r) =- \frac12$ gives 
$$ n_{1,2} (r) = - \alpha_1(r) .$$
Assume inductively that $n_{i,2} (r) = - \alpha_i(r) $ for $1 \leq i < j$. Then \eqref{normal derivative values} becomes 
$$ n_{j,2}(r) = 2 \sum_{ i=0} ^j \alpha_i(r) \beta_{j-i} (r) - 2 \sum_{ i=1}^{j-1} \alpha_i(r) \beta_{j-i} (r)+ 2 n _{0,2} (r) \beta_j(r) - \beta_{j}(r) $$
$$ = 2 \alpha_0(r) \beta_j(r) + 2 \alpha_j(r) \beta_0(r) + 2 n_{0,2} (r) \beta_j(r) -2 \beta_j(r) .
$$
Substituting the values of  $ \alpha_0(r),\  \beta_0(r) ,$ and $ n_{0,2} (r)$ simplifies this to $- \alpha_j(r).$ 
\qed
\end{proof}

\begin{figure} 
\begin{center} 
\includegraphics[width=50mm]{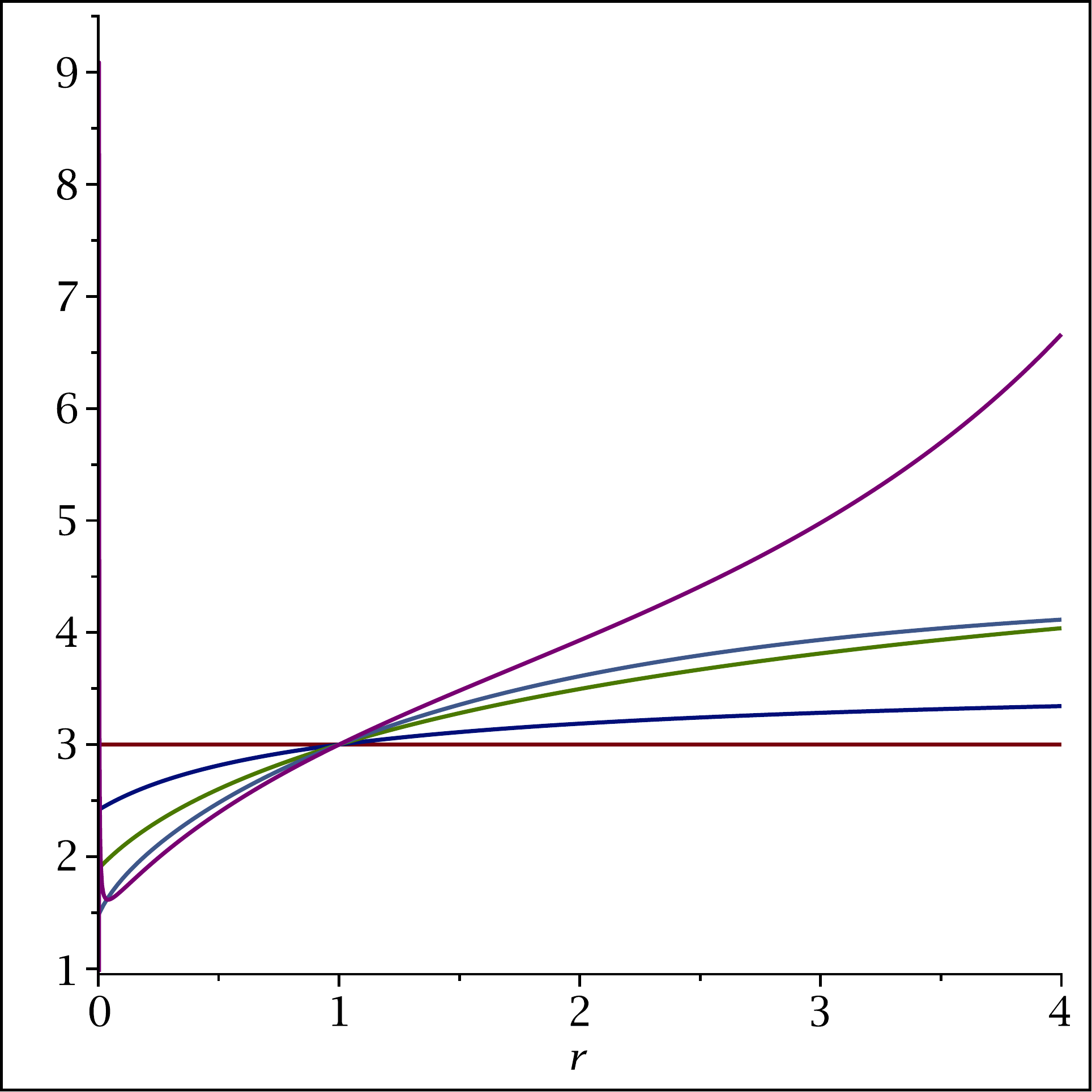}
\caption{Graphs of $\frac{ \gamma_j(r)}{\alpha_{j+1}(r) } $ for $j =0,1,2,3,4$. }
\end{center}
\end{figure}

So far, the theme of this section seems to be that the polynomials with respect to $\Delta_r$ behave quite like the standard polynomials in many ways. However, the results presented so far seem to be about the extent of this similarity. In fact, the standard polynomials seem to be quite exceptional in this more general context, and sections 4 and 5 explore some observed differences between the general behavior of polynomials with respect to $\Delta_r$ and the standard polynomials. This section concludes with a generalization of a result from \cite{nsty} which surprisingly does not extend as one might expect to this setting. \par 

In theorem 2.3 of \cite{nsty}, it is noted that $\gamma_j(1) = 3 \alpha_{j+1} (1)$ and that this is a simple consequence of the fact that $\gamma_j(1)$ and $\alpha_{j+1}(1)$ satisfy the same recurrence relation with initial data $\alpha_1(1) = \frac16 = \frac{1}{3}\gamma_0(1) .$ Here, the recurrence relations obtained for $\alpha_{j+1}(r)$ and $\gamma_{j} (r)$ are quite different, and the ratio $\frac{ \gamma_{j}(r) }{\alpha_{j+1} (r) } $ varies with $j$, as seen in figure 3.1. \par

The generalization of the equation $\gamma_j(1) = 3 \alpha_{j+1} (1)$ which holds for all values of $r$ is \eqref{alpha gamma relation}. Note that $t_{0,1} (r) =0$ since $P_{0,1}^{(r)} $ is constant, so this says 
\begin{equation} \label{alpha gamma relation generalization}
 \alpha_{j+1} = 2 \sum_{i=0}^{j+1} t_{j+1-i,1}(r) \gamma_i(r) = 2 \sum_{ i=0} ^j t_{j+1-i,1}(r) \gamma_i(r).
\end{equation}
When $r=1$, it happens that $t_{j,1}(r)$ vanishes for all $j \neq 1$, and $t_{1,1} (1) =\frac16$, so the only term which remains in the sum on the right side of \eqref{alpha gamma relation generalization} is $2 t_{1,1}(1) \gamma_j(1) = \frac13 \gamma_j(1)$, but $t_{j,1}(r)$ does not vanish for all values of $r$, so the relationship between $\alpha_{j+1}(r)$ and $\gamma_j(r)$ is much less simple. Indeed, when $r=1$ this simple relationship implies that the long term behavior of the sequences $\{ \alpha_j(1) \}$ and $\{ \gamma_j(1)\}$ is exactly the same, but the next two sections will demonstrate that for other values of $r$ the long term behavior of $\{ \alpha_j(r) \} $ seems to be different from that of $\{ \gamma_j(r)\}$.

\section{Numerical Data}

\begin{figure}[p]
    \centering
    \includegraphics[width=30mm]{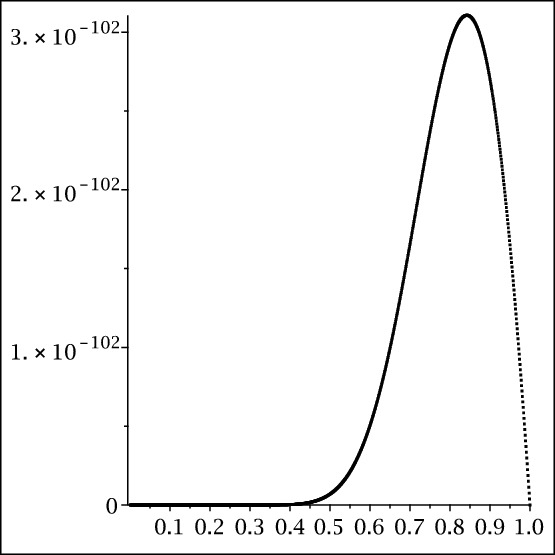}    \includegraphics[width=30mm]{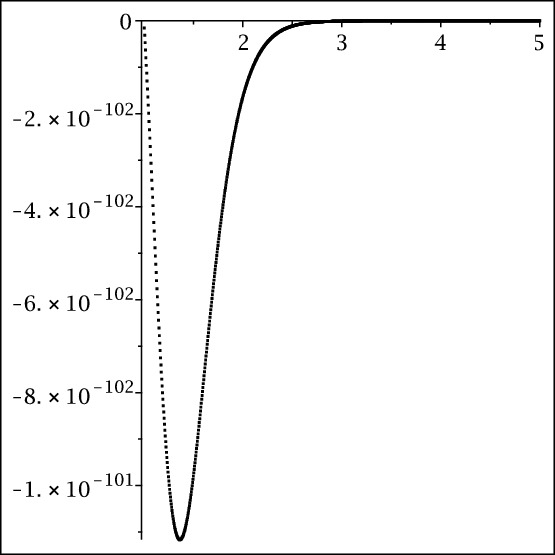}    \includegraphics[width=30mm]{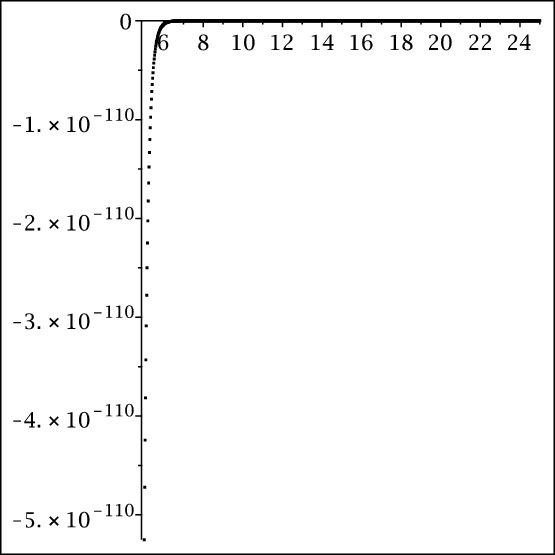}\\
    \includegraphics[width=30mm]{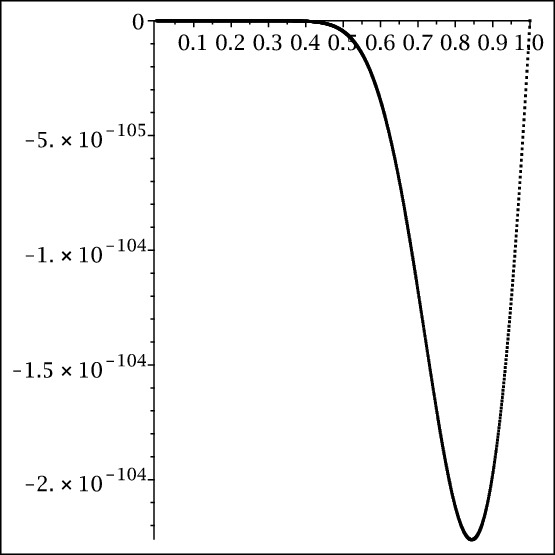}    \includegraphics[width=30mm]{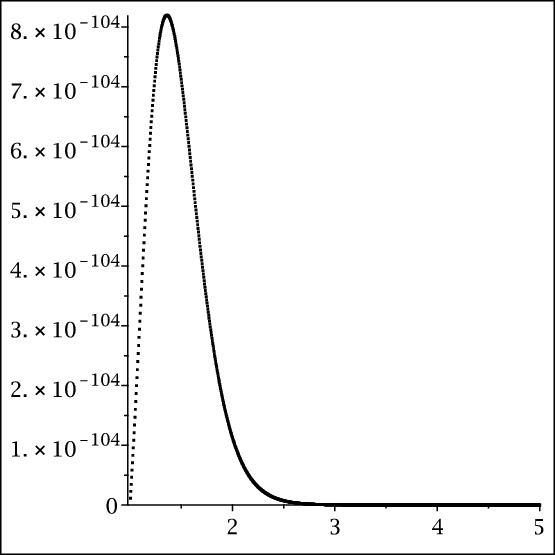}    \includegraphics[width=30mm]{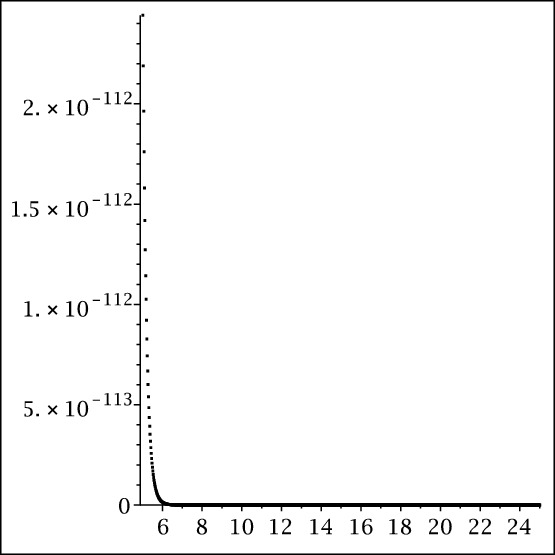}\\    \includegraphics[width=30mm]{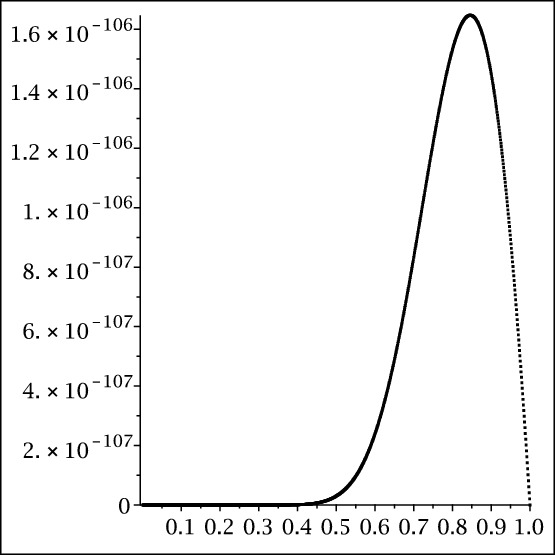}    \includegraphics[width=30mm]{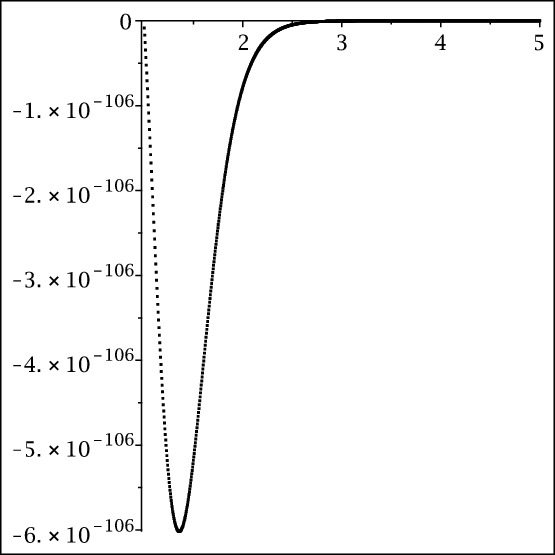}   \includegraphics[width=30mm]{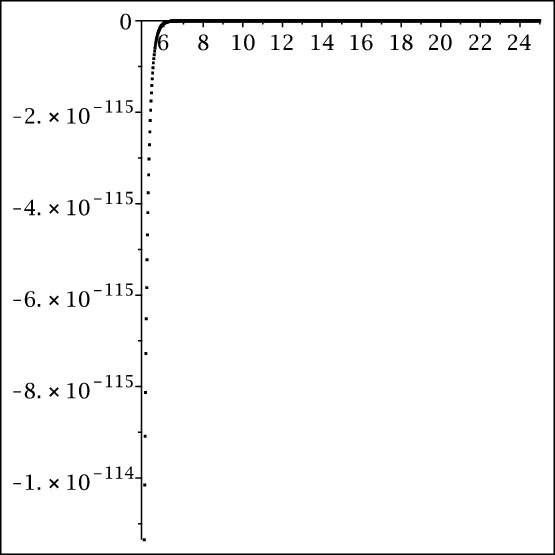}
    \caption{Numerical graphs of $\alpha_{47}(r)$, $\alpha_{48}(r)$, and $\alpha_{49}(r)$ as functions of $r$. For $\alpha_0(r),\dots,\alpha_{46}(r)$ see \protect\url{http://pi.math.cornell.edu/~reuspurweb/numericalAlphas.html} }
    \label{alphagraphs}
\end{figure}
\begin{figure}[p]
    \centering
    \includegraphics[width=30mm]{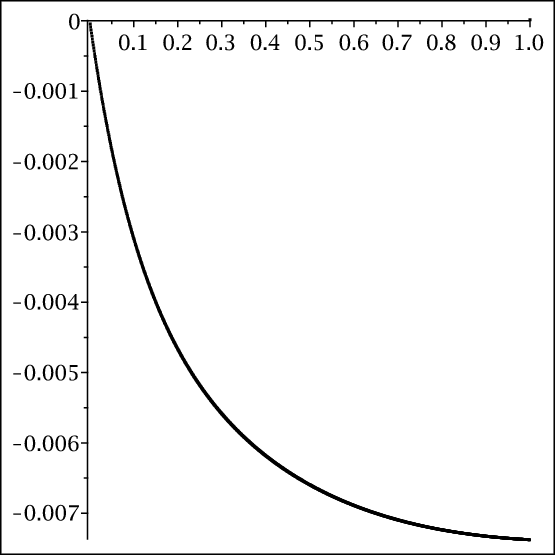}    \includegraphics[width=30mm]{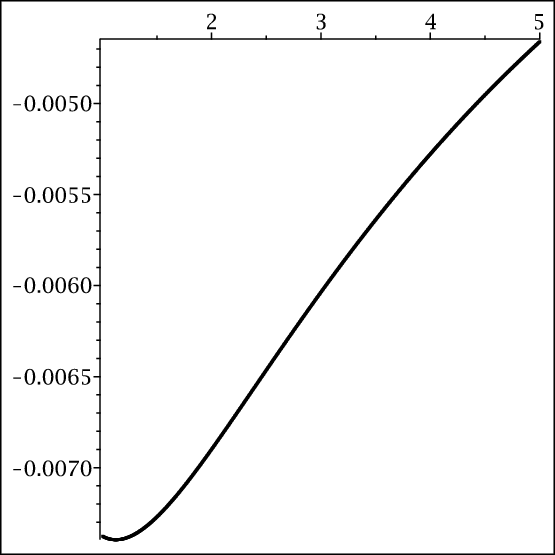}    \includegraphics[width=30mm]{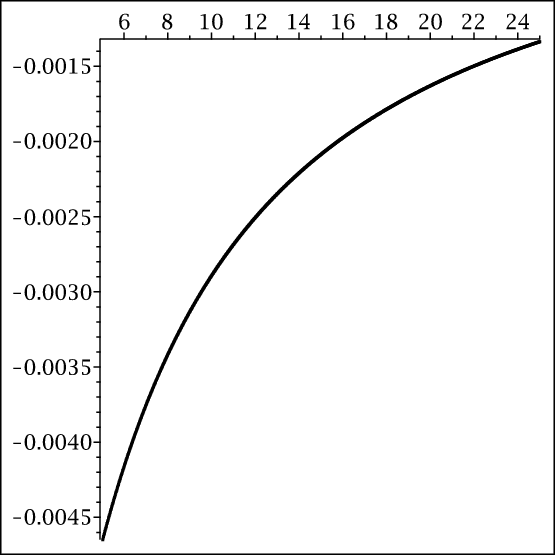}\\
    \includegraphics[width=30mm]{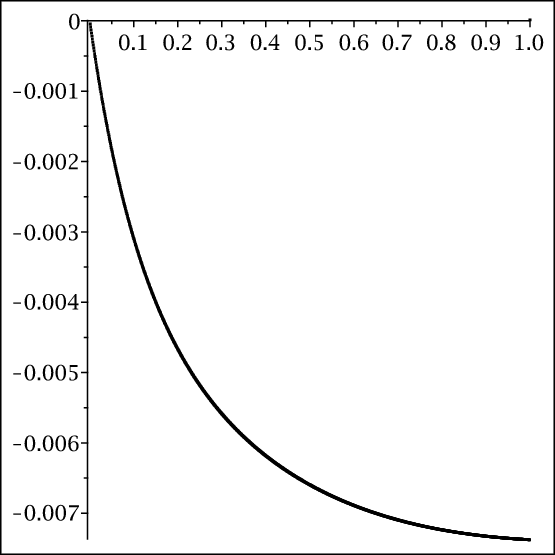}    \includegraphics[width=30mm]{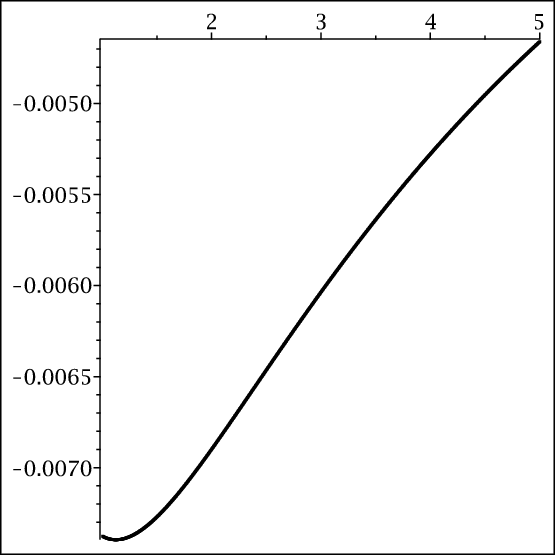}    \includegraphics[width=30mm]{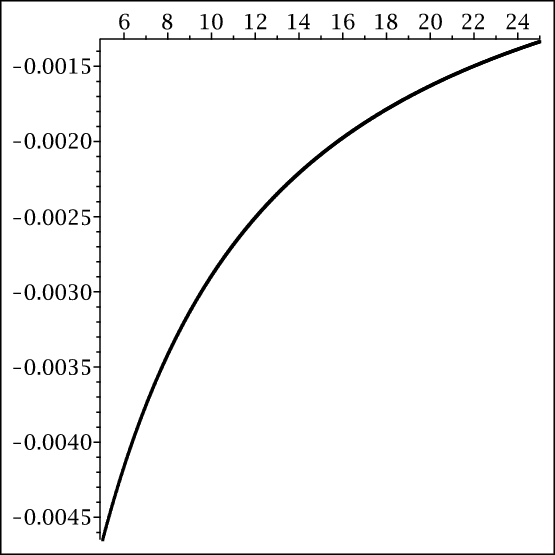}\\    \includegraphics[width=30mm]{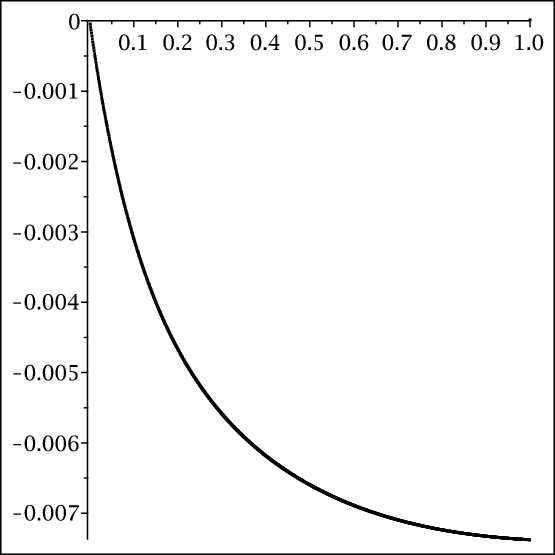}    \includegraphics[width=30mm]{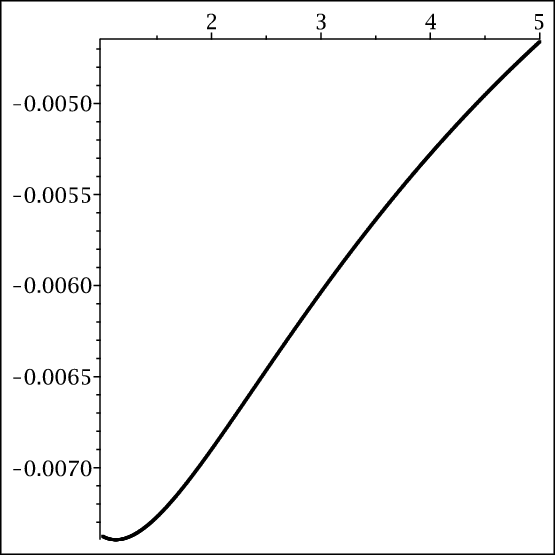}   \includegraphics[width=30mm]{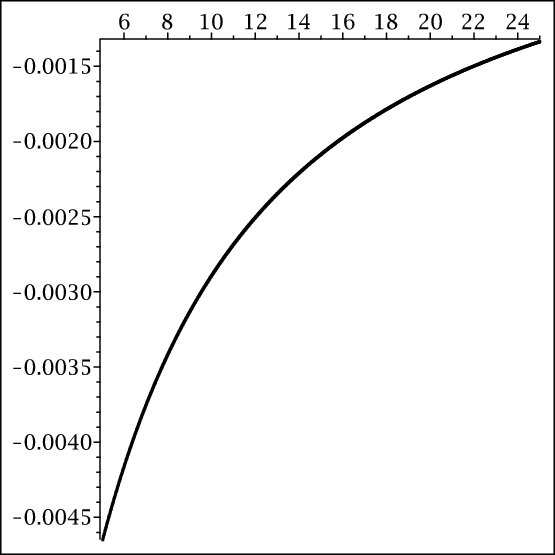}
    \caption{Numerical graphs of $\frac{\alpha_{48}(r)}{\alpha_{47}(r)}$, $\frac{\alpha_{49}(r)}{\alpha_{48}(r)}$, and $\frac{\alpha_{50}(r)}{\alpha_{49}(r)}$. For $\frac{ \alpha_1(r)}{\alpha_0(r)} , \dots , \frac{\alpha_{47}(r) } { \alpha_{46}(r)}$, see \protect\url{http://pi.math.cornell.edu/~reuspurweb/numericalAlphaRatios.html} }
    \label{alpharatios}
\end{figure}
\begin{figure}[p]
    \centering
    \includegraphics[width=30mm]{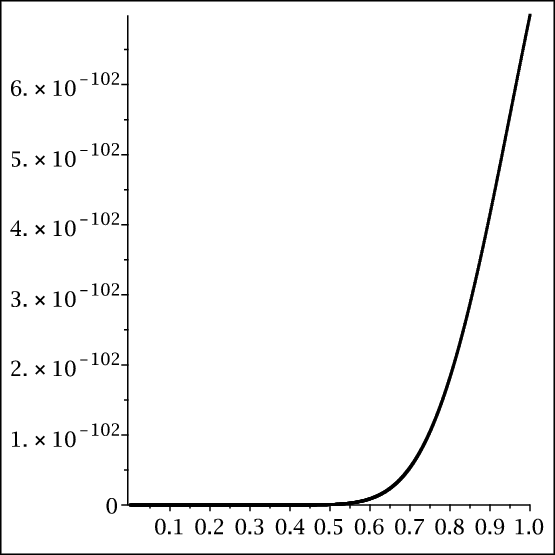}    \includegraphics[width=30mm]{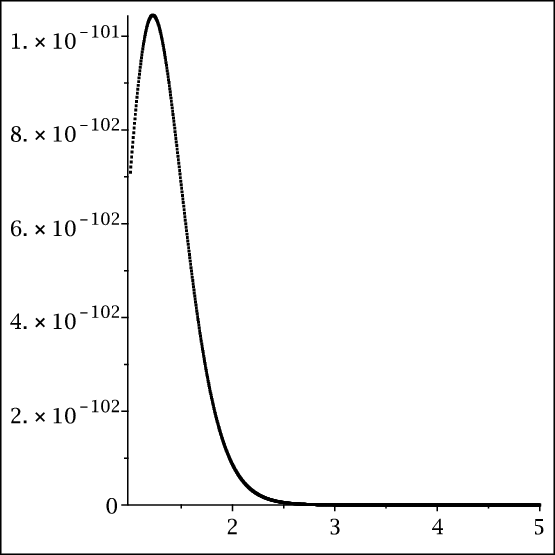}    \includegraphics[width=30mm]{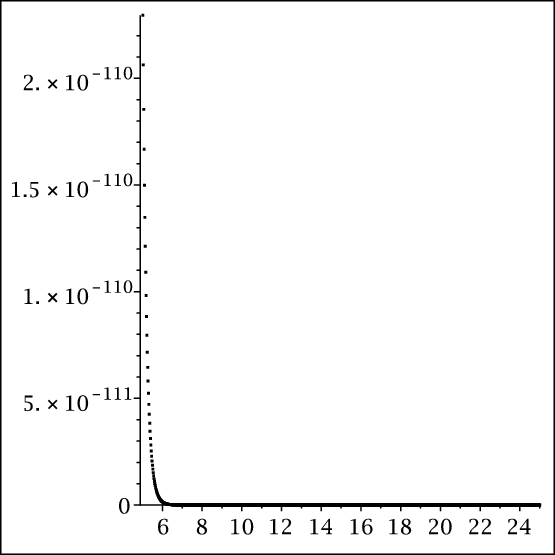}\\
    \includegraphics[width=30mm]{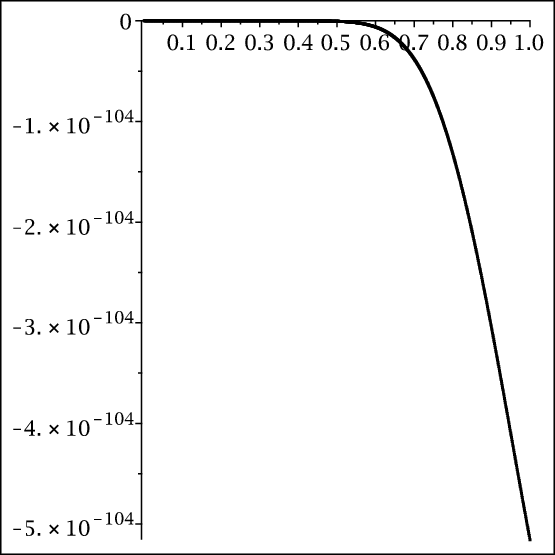}    \includegraphics[width=30mm]{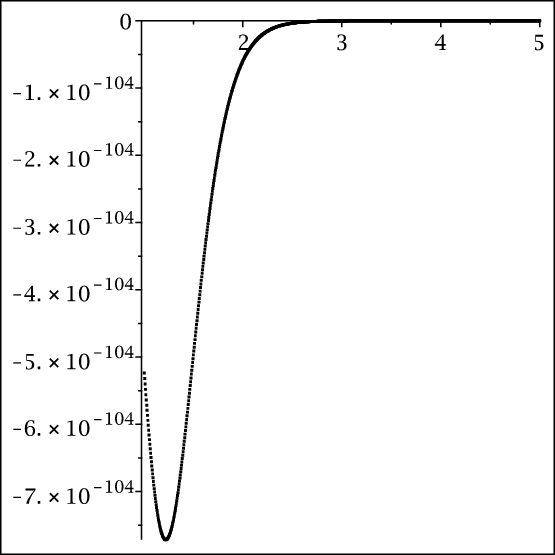}    \includegraphics[width=30mm]{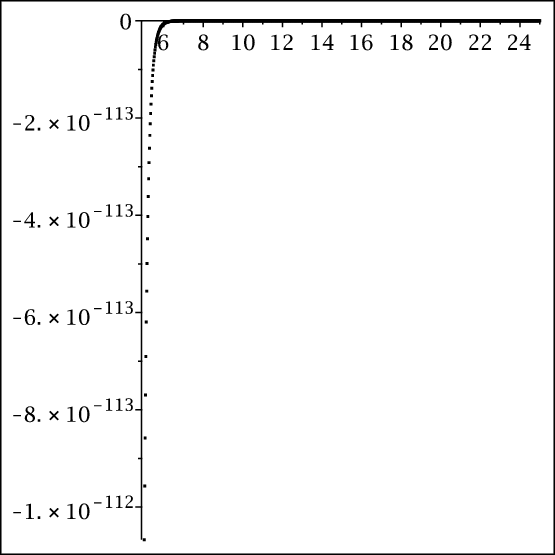}\\    \includegraphics[width=30mm]{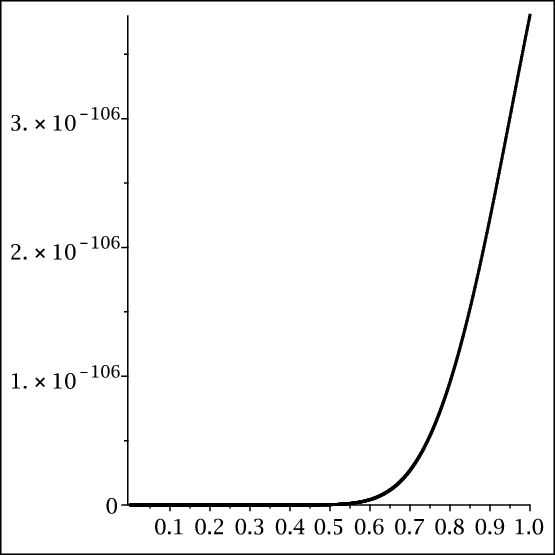}    \includegraphics[width=30mm]{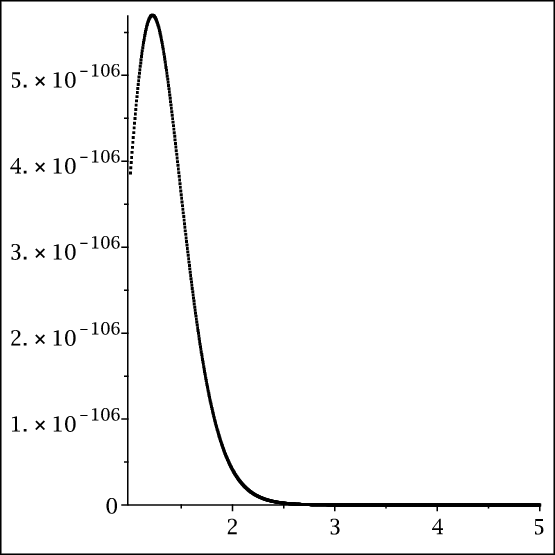}   \includegraphics[width=30mm]{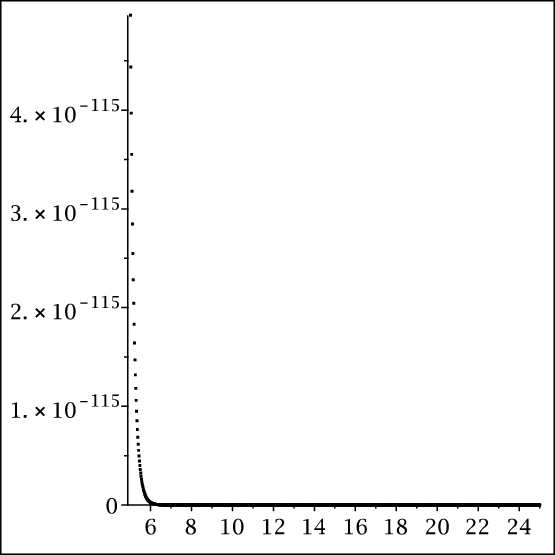}
    \caption{Numerical graphs of $\beta_{47}(r)$, $\beta_{48}(r)$, and $\beta_{49}(r)$. For $\beta_0(r),\dots,\beta_{46}(r)$ see \protect\url{http://pi.math.cornell.edu/~reuspurweb/numericalBetas.html} }
    \label{betagraphs}
\end{figure}

\begin{figure}[p]
    \centering
    \includegraphics[width=30mm]{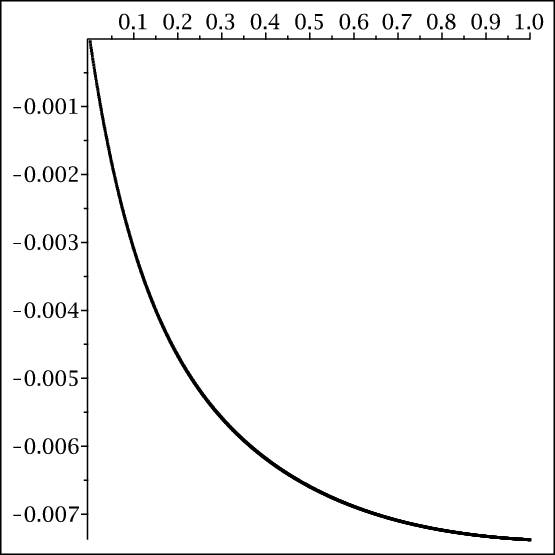}    \includegraphics[width=30mm]{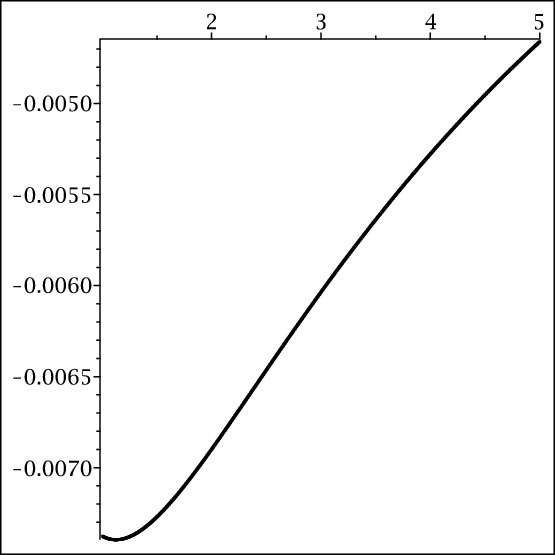}    \includegraphics[width=30mm]{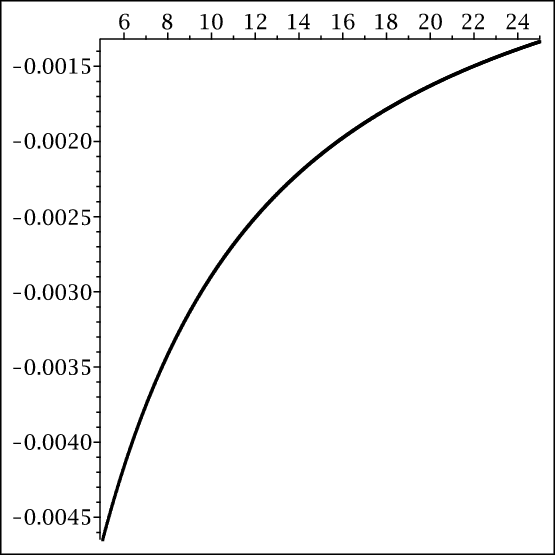}\\
    \includegraphics[width=30mm]{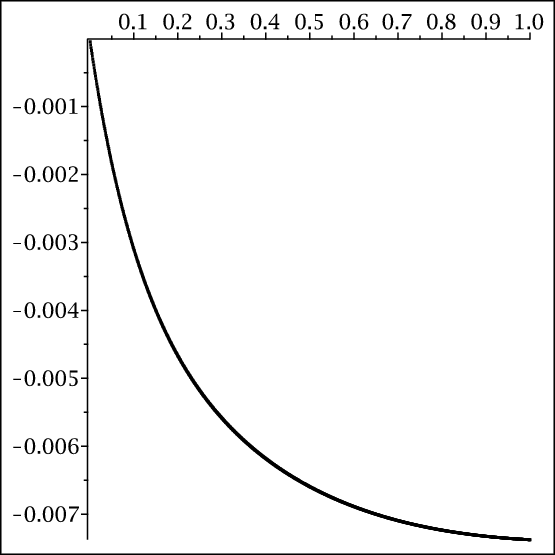}    \includegraphics[width=30mm]{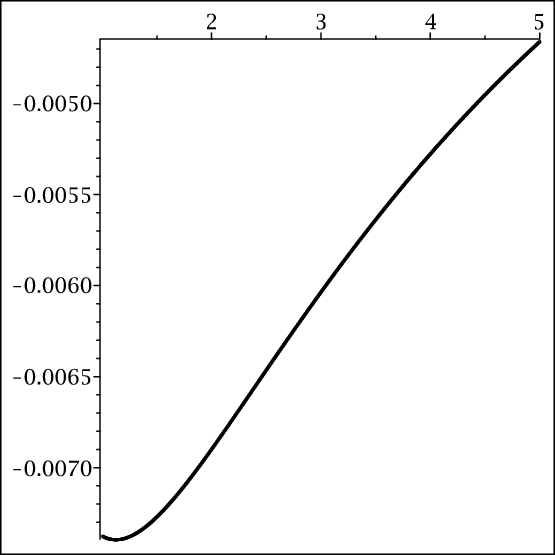}    \includegraphics[width=30mm]{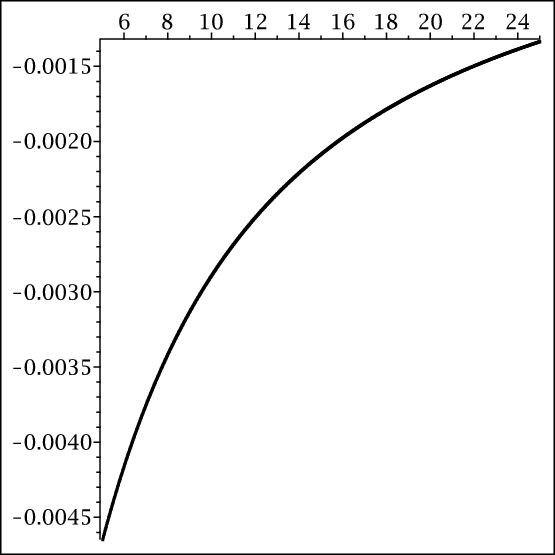}\\    \includegraphics[width=30mm]{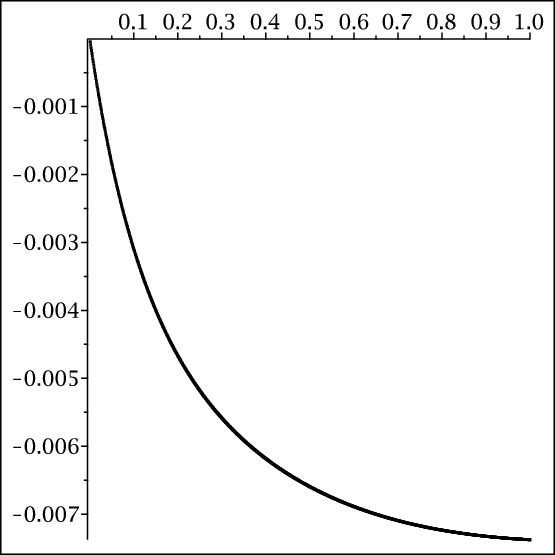}    \includegraphics[width=30mm]{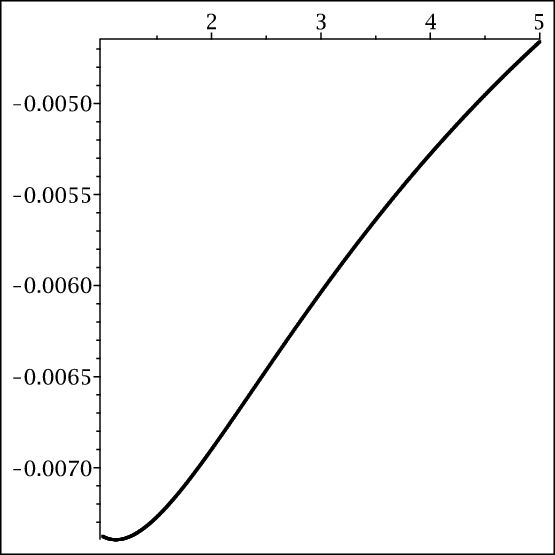}   \includegraphics[width=30mm]{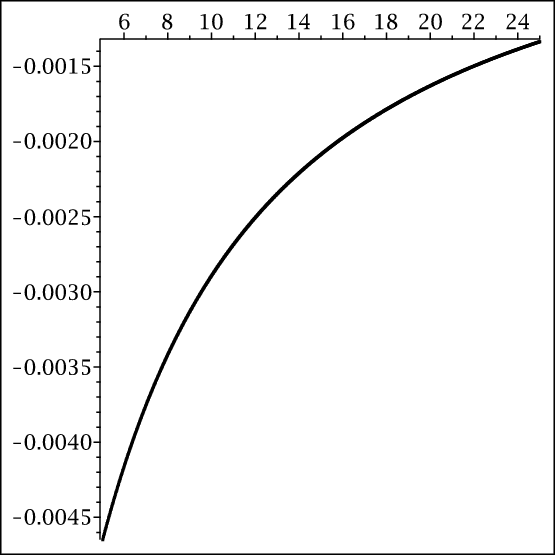}
    \caption{Numerical graphs of $\frac{\beta_{48}(r)}{\beta_{47}(r)}$, $\frac{\beta_{49}(r)}{\beta_{48}(r)}$, and $\frac{\beta_{50}(r)}{\beta_{49}(r)}$. For $\frac{\beta_1(r)}{\beta_{0}(r)},\dots,\frac{\beta_{47}(r)}{\beta_{46}(r)}$ see \protect\url{http://pi.math.cornell.edu/~reuspurweb/numericalBetaRatios.html}}
    \label{betaratios}
\end{figure}
\begin{figure}[p]
    \centering
    \includegraphics[width=30mm]{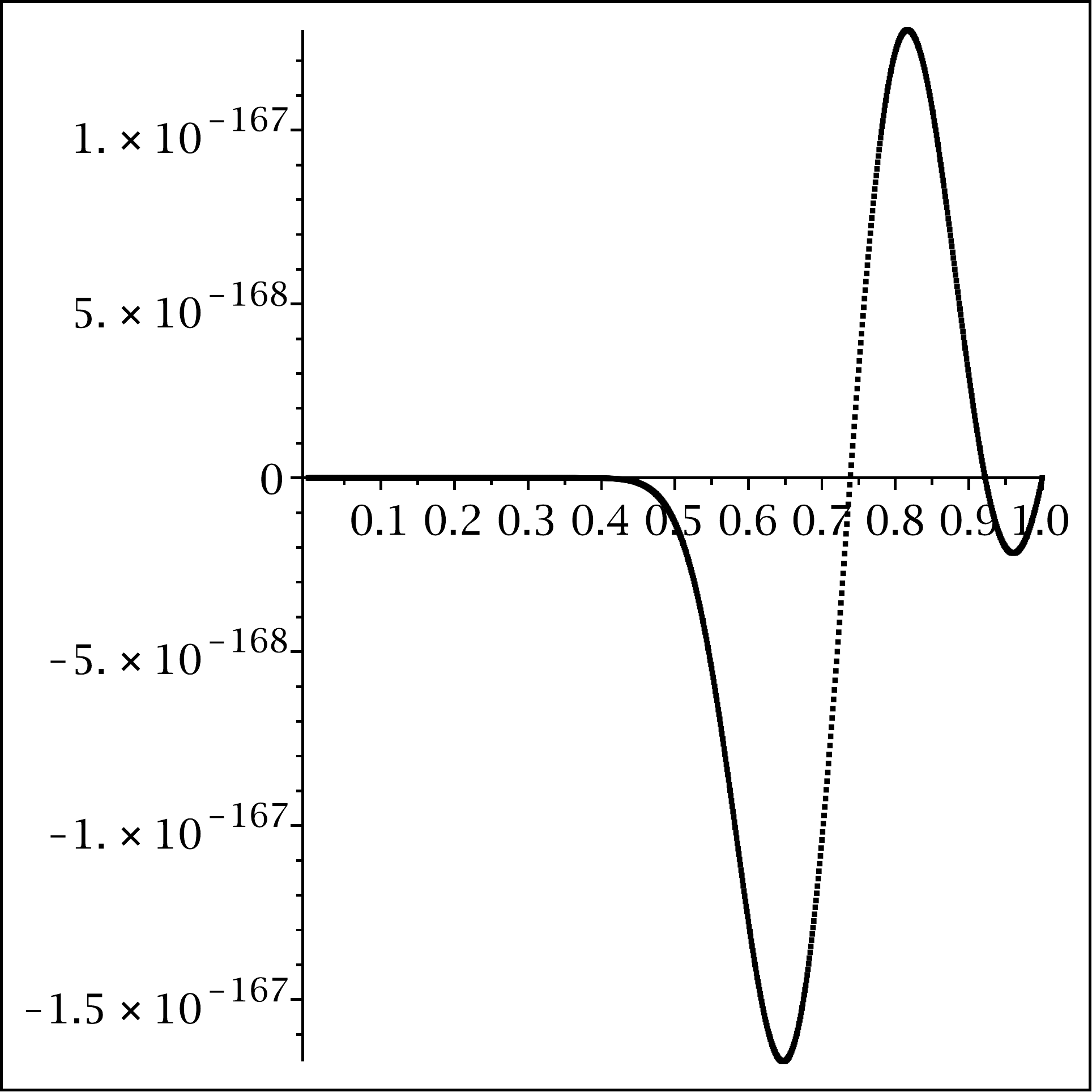}    \includegraphics[width=30mm]{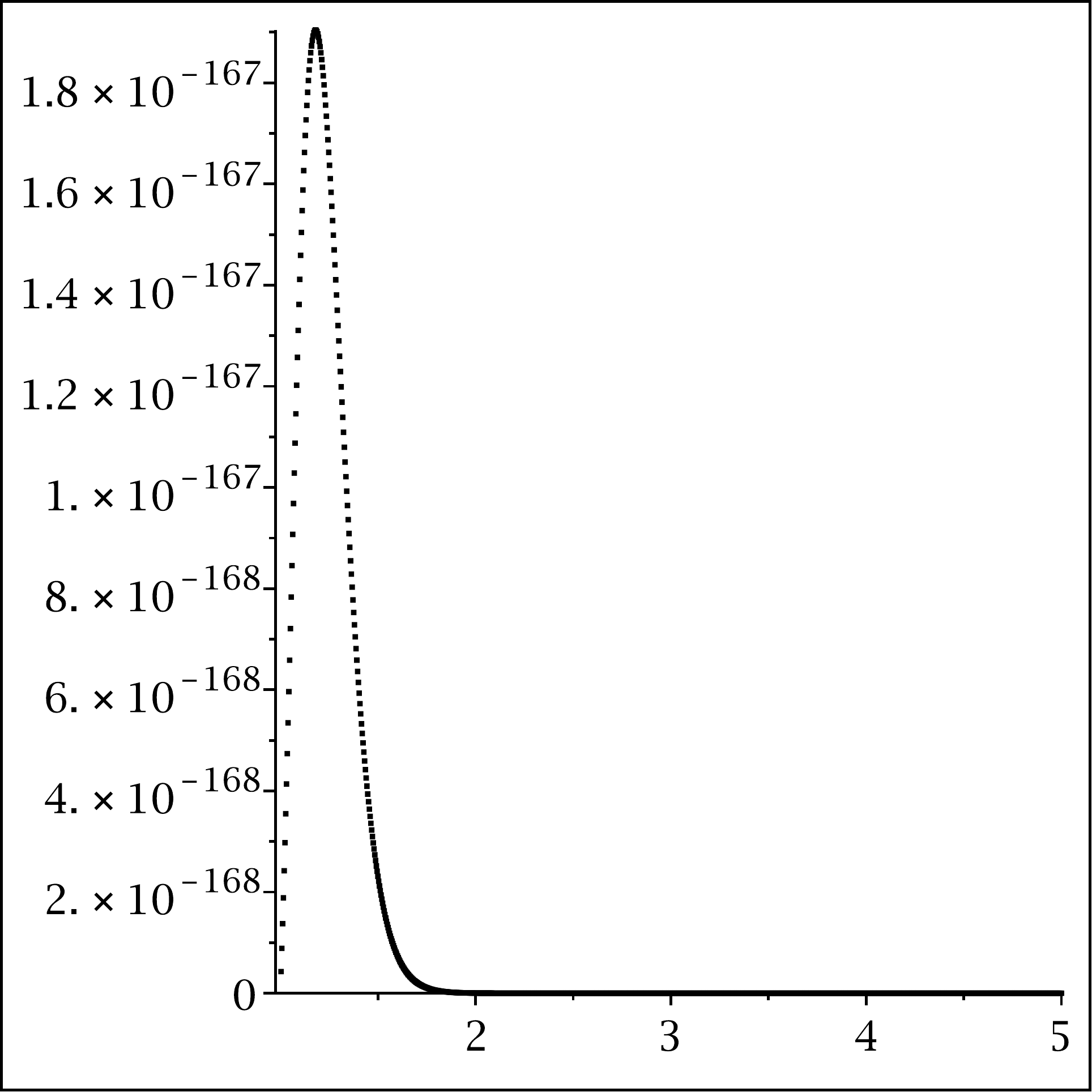}    \includegraphics[width=30mm]{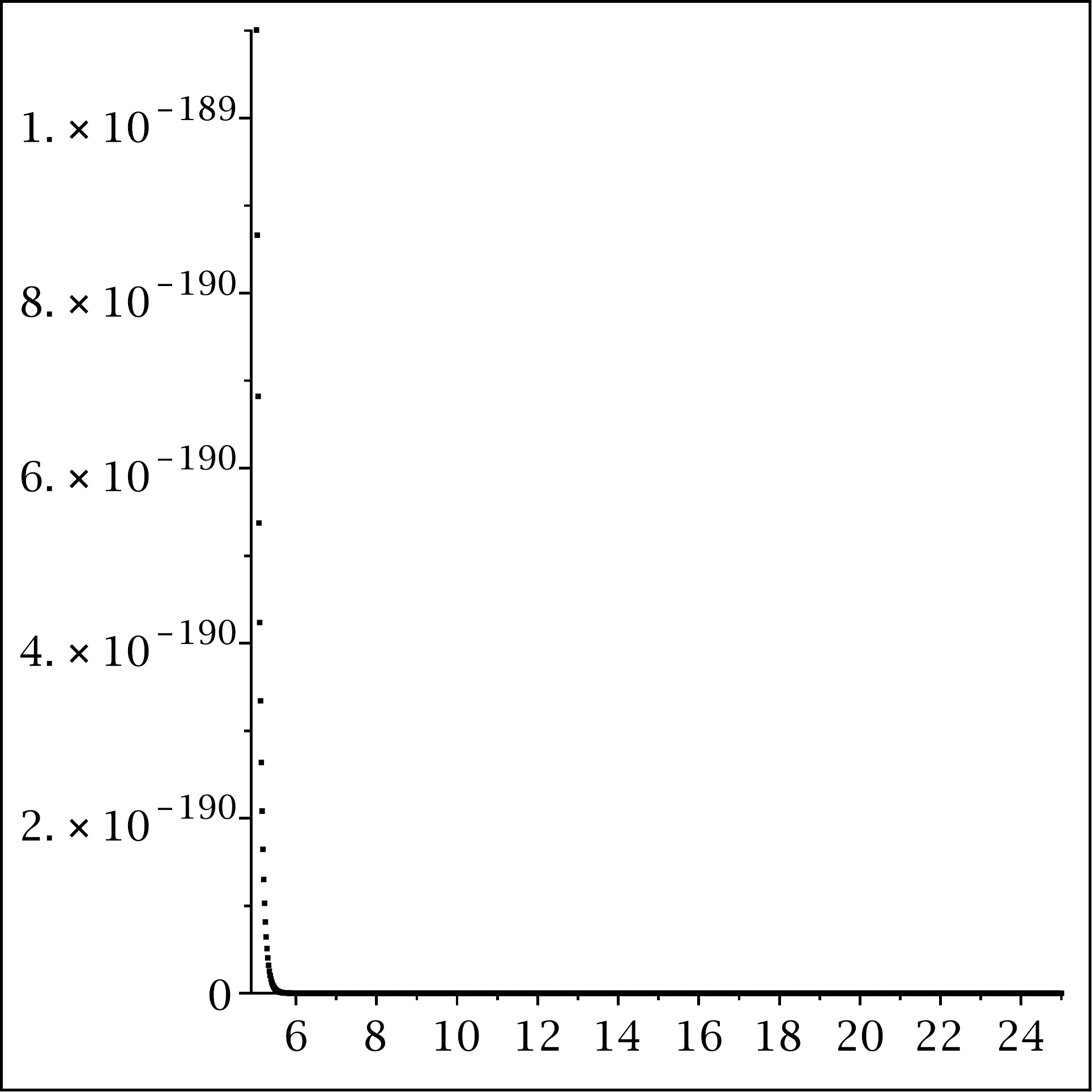}\\
    \includegraphics[width=30mm]{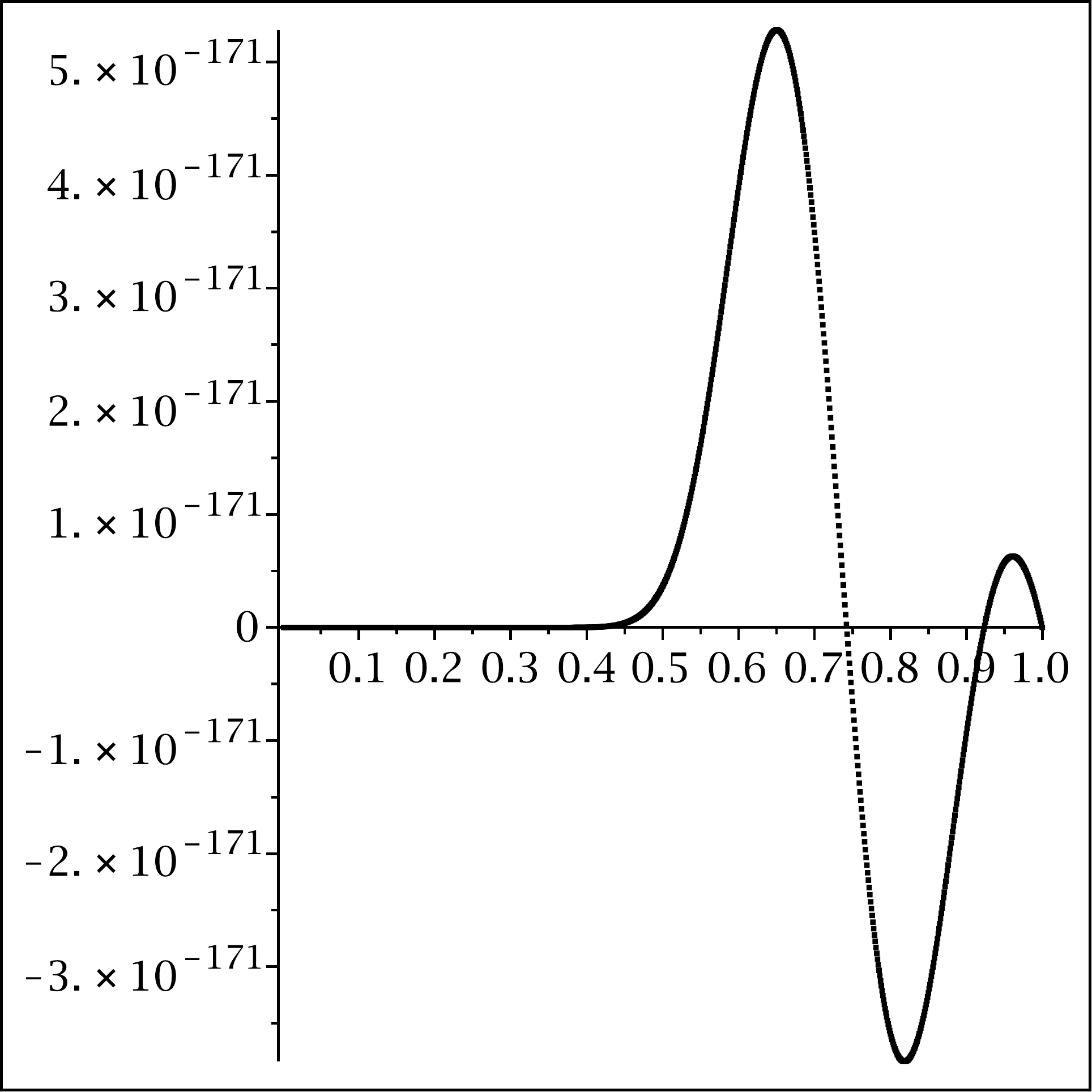}    \includegraphics[width=30mm]{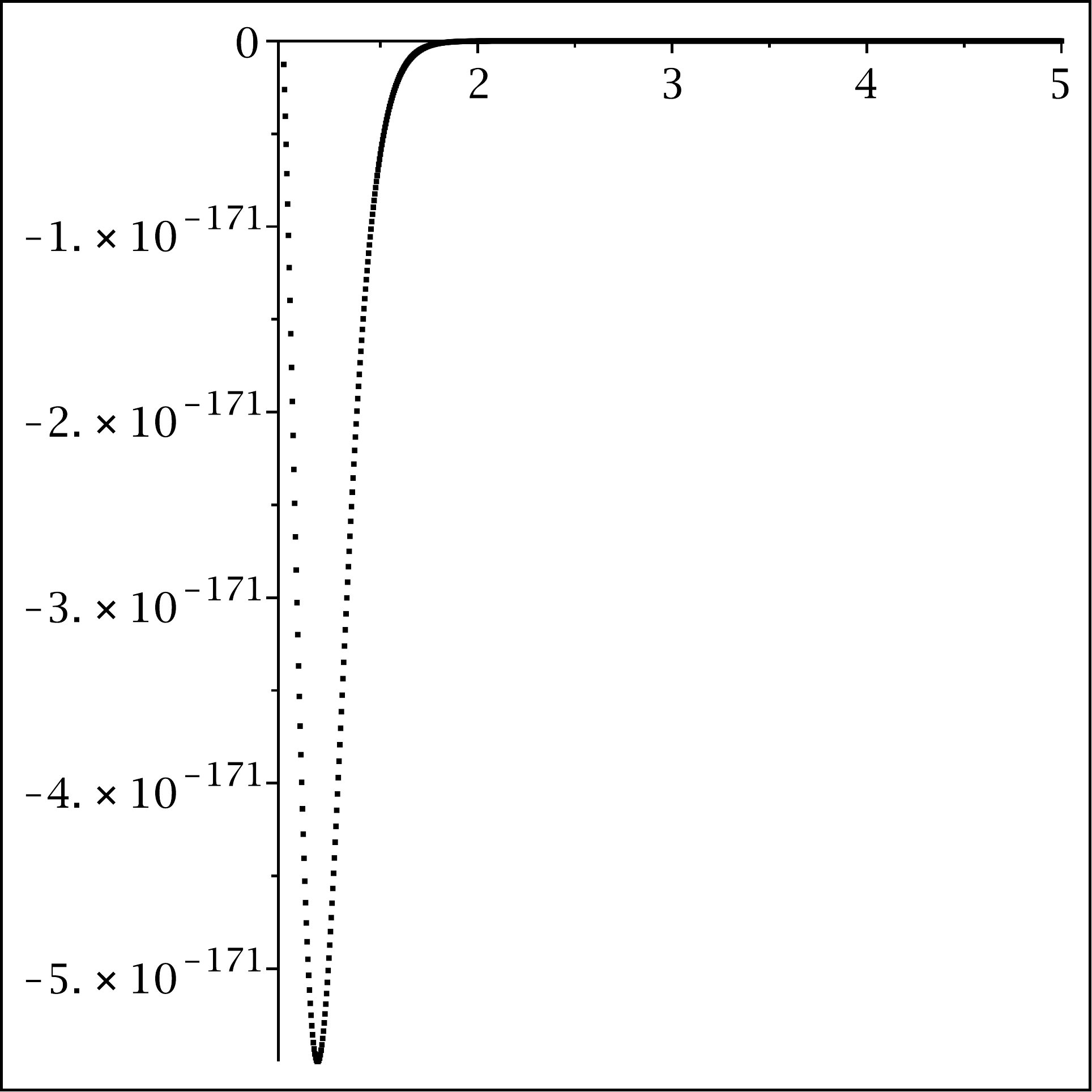}    \includegraphics[width=30mm]{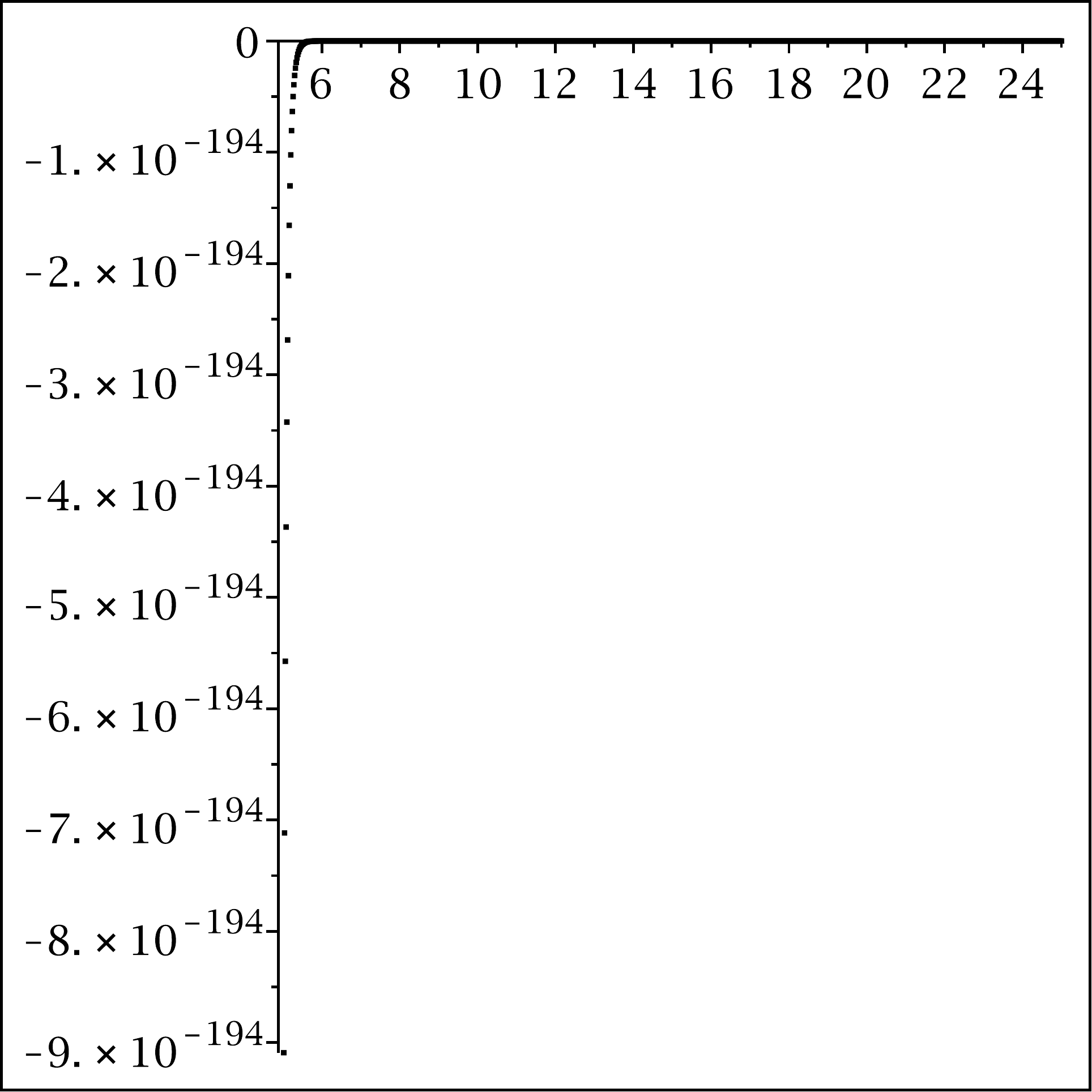}\\    \includegraphics[width=30mm]{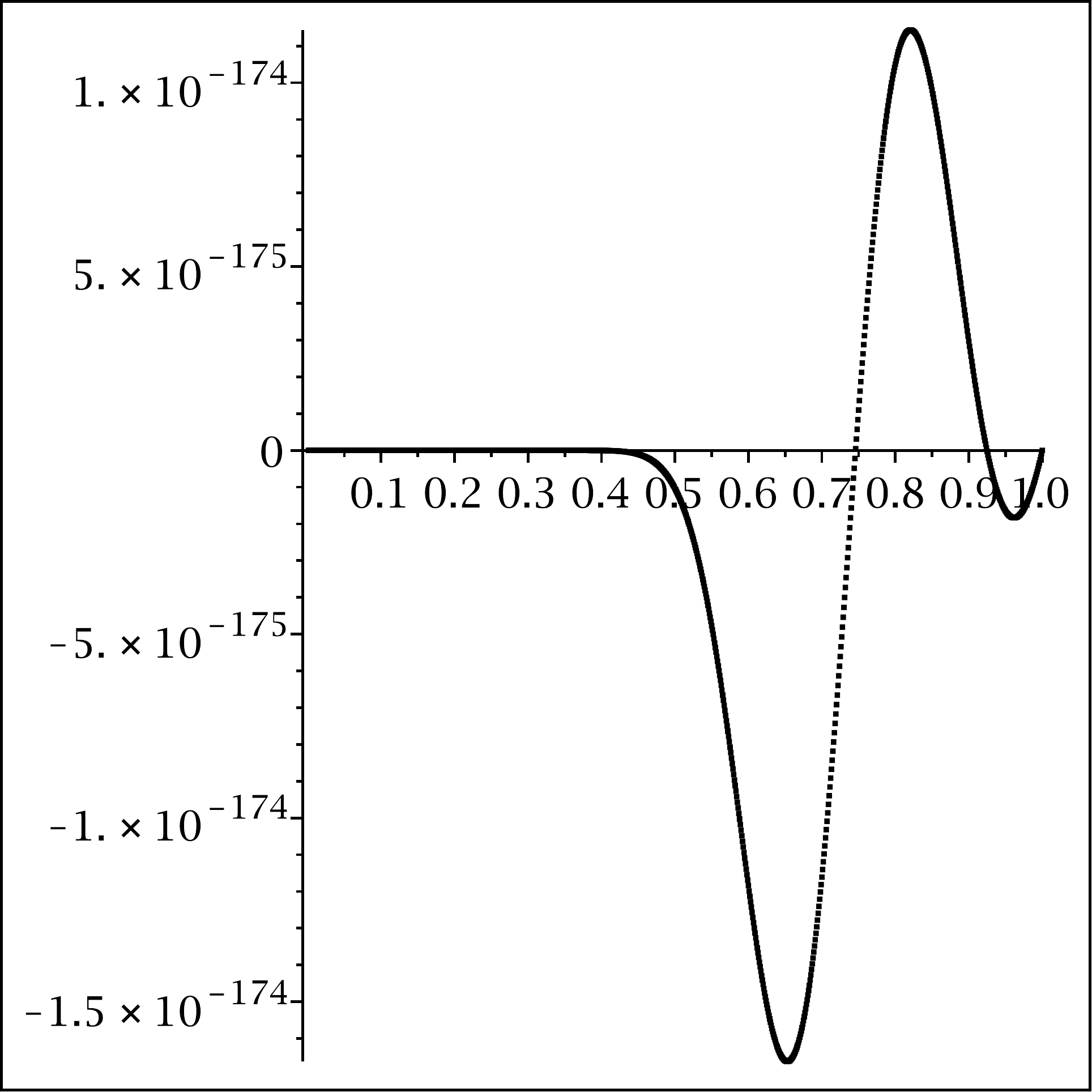}    \includegraphics[width=30mm]{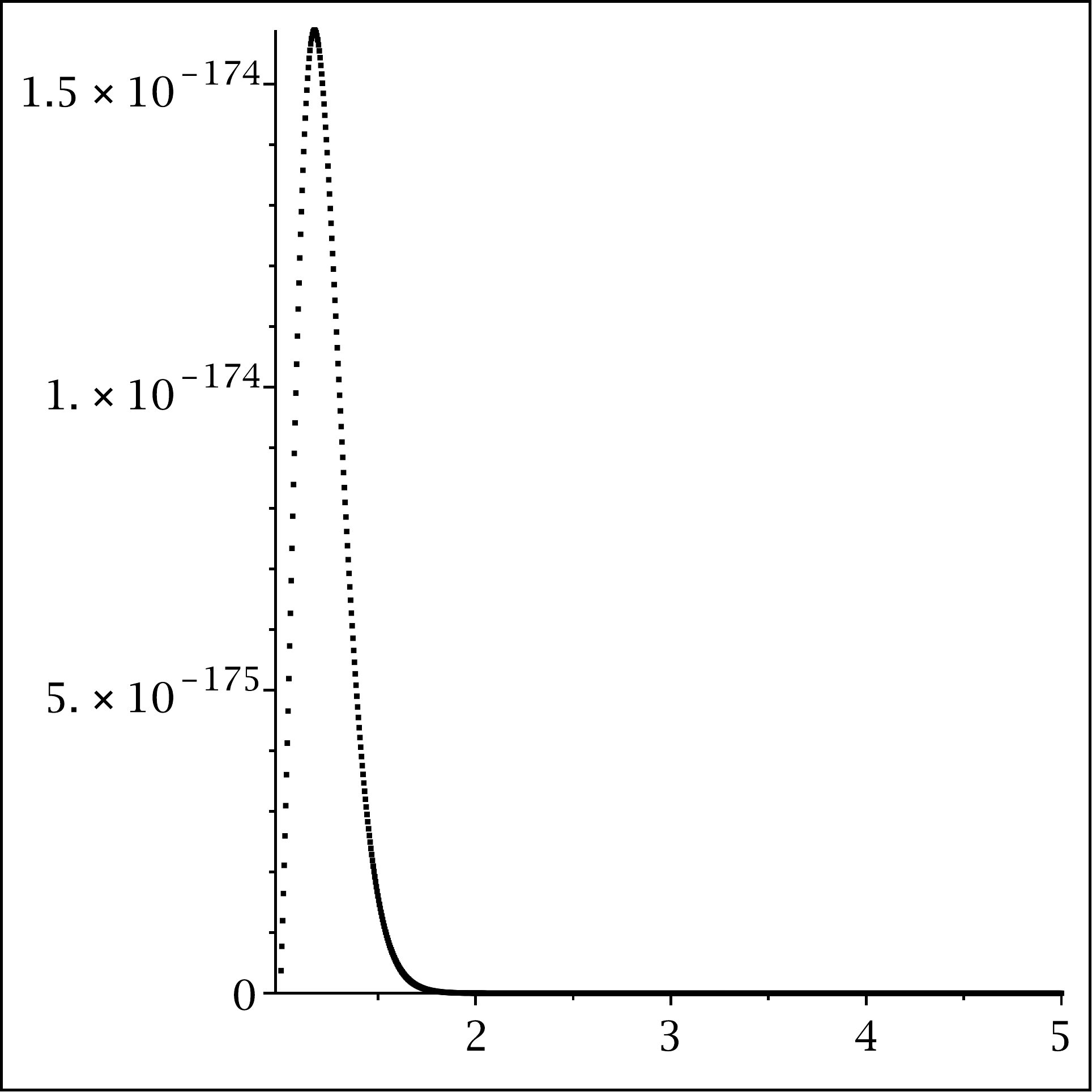}   \includegraphics[width=30mm]{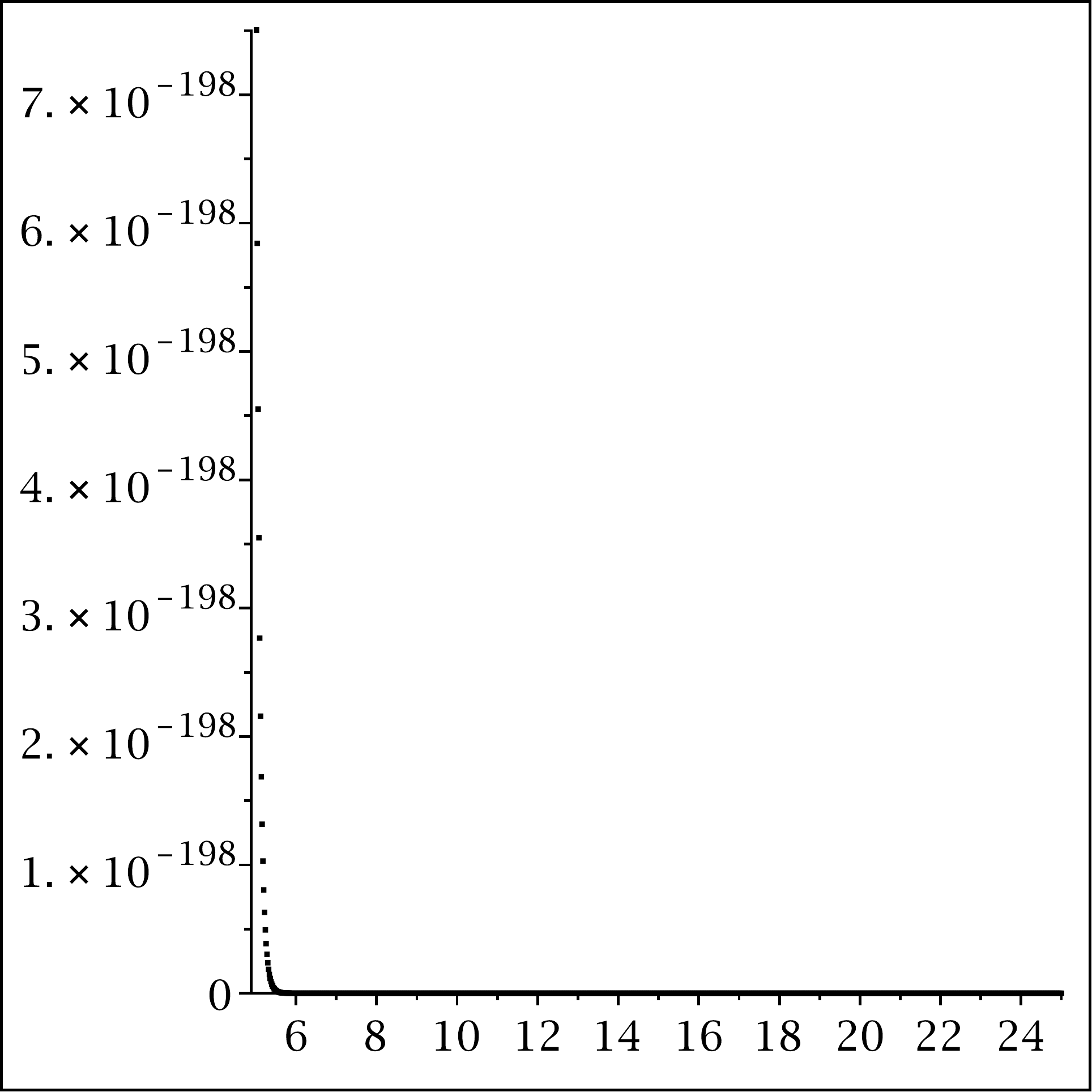}
    \caption{Numerical graphs of $\gamma_{47}(r)$, $\gamma_{48}(r)$, and $\gamma_{49}(r)$. For $\gamma_{0}(r),\dots,\gamma_{47}(r)$ see \protect\url{http://pi.math.cornell.edu/~reuspurweb/gammaPlotsNumerical.html}}
    \label{gammagraphs}
\end{figure}

\begin{figure}[p]
    \centering
    \includegraphics[width=30mm]{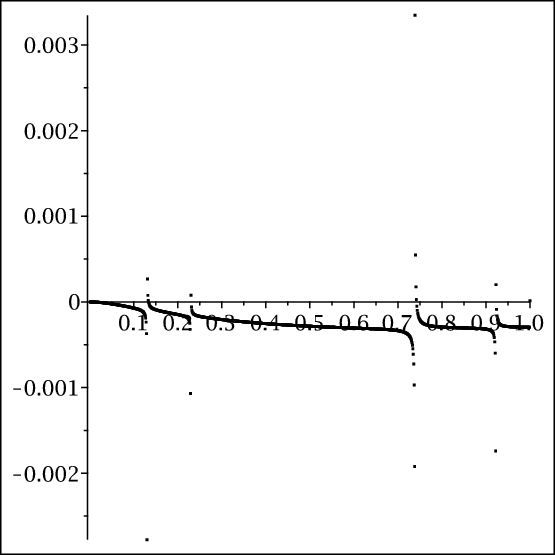}    \includegraphics[width=30mm]{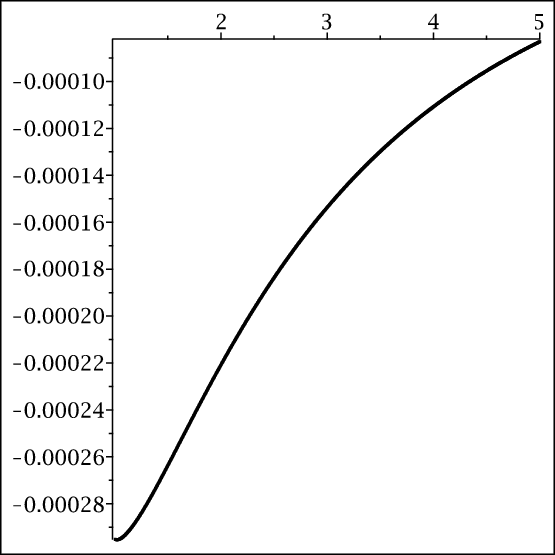}    \includegraphics[width=30mm]{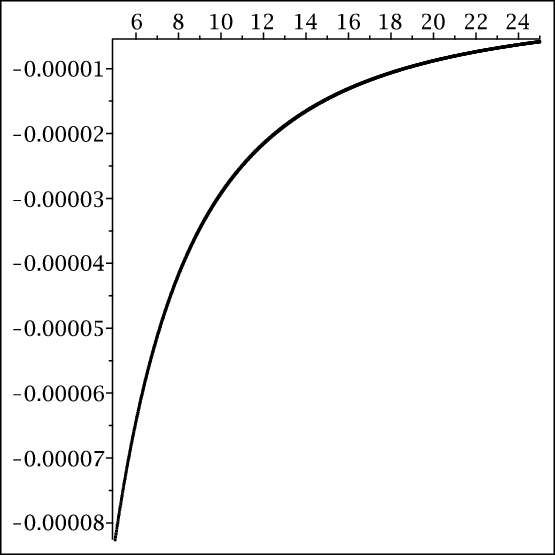}\\
    \includegraphics[width=30mm]{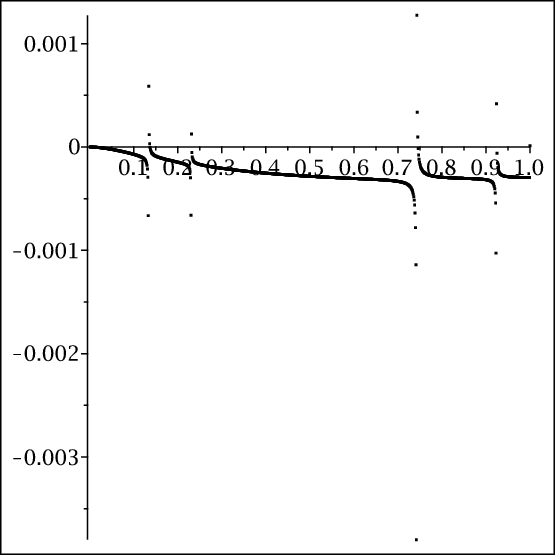}    \includegraphics[width=30mm]{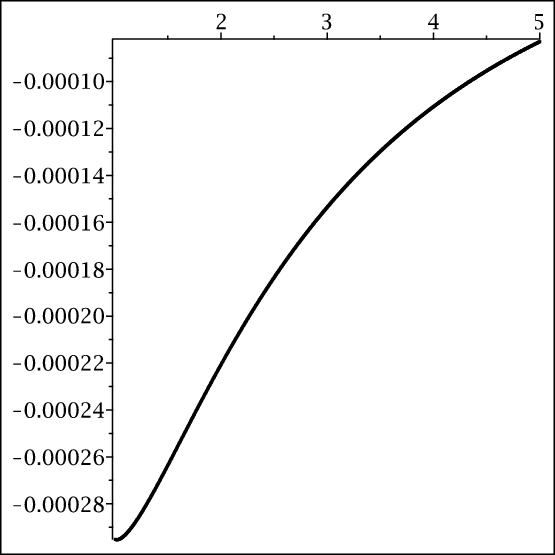}    \includegraphics[width=30mm]{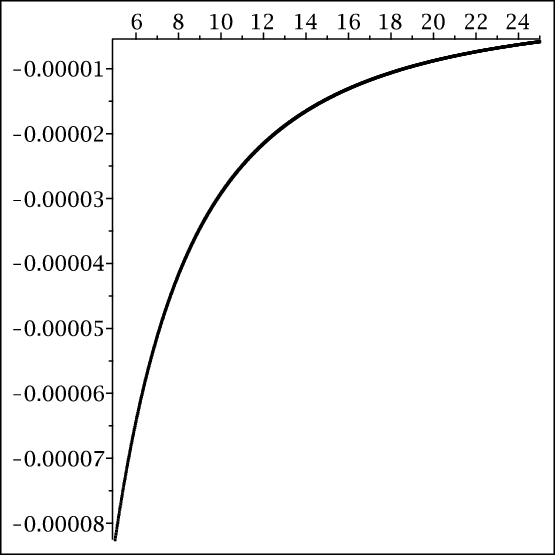}\\    \includegraphics[width=30mm]{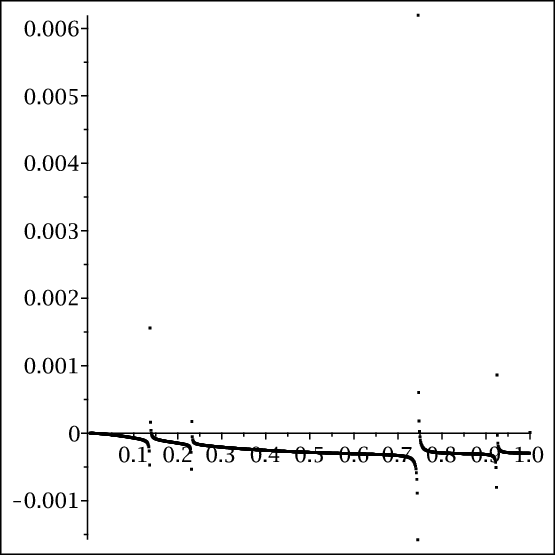}    \includegraphics[width=30mm]{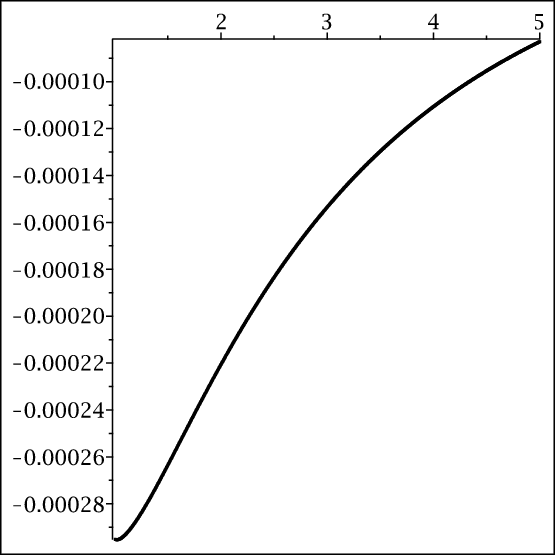}   \includegraphics[width=30mm]{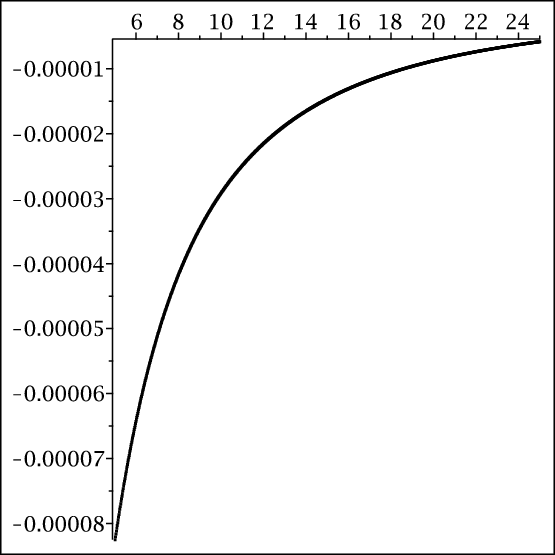}
    \caption{Numerical graphs of $\frac{\gamma_{48}(r)}{\gamma_{47}(r)}$, $\frac{\gamma_{49}(r)}{\gamma_{48}(r)}$, and $\frac{\gamma_{50}(r)}{\gamma_{49}(r)}$. For $\frac{\gamma_1(r)}{\gamma_{0}(r)},\dots,\frac{\gamma_{47}(r)}{\gamma_{46}(r)}$ see \protect\url{http://pi.math.cornell.edu/~reuspurweb/numericalGammaRatios.html}}
    \label{gammaratios}
\end{figure}
\begin{figure}[p]
    \centering
    \includegraphics[width=30mm]{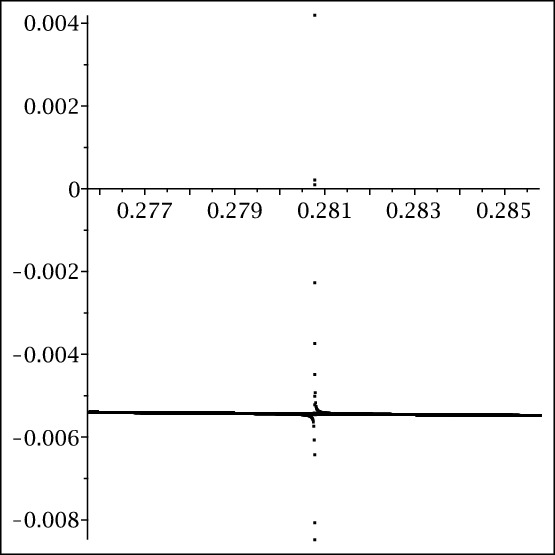}
    \includegraphics[width=30mm]{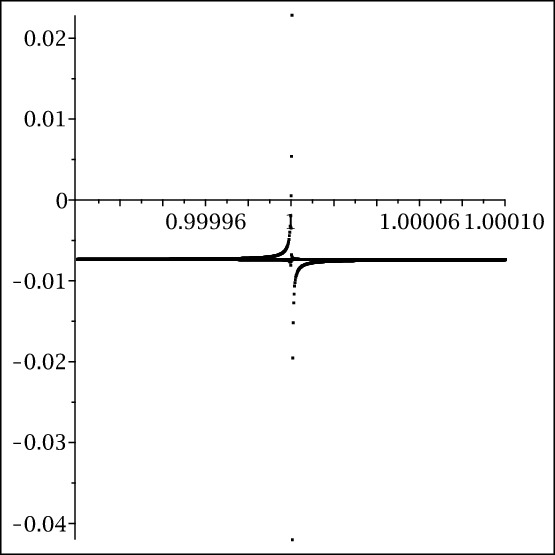}
    \caption{Simultaneous graphs of $\frac{ \alpha_{j+1}(r)}{\alpha_j(r)}$ for $j= 14, \dots, 24$ near  $r=\frac{\sqrt{17}-3}{4}$ and $r=1$.}
    \label{alphazoomplot}
\end{figure}
\begin{figure}[p]
    \centering
    \includegraphics[width=30mm]{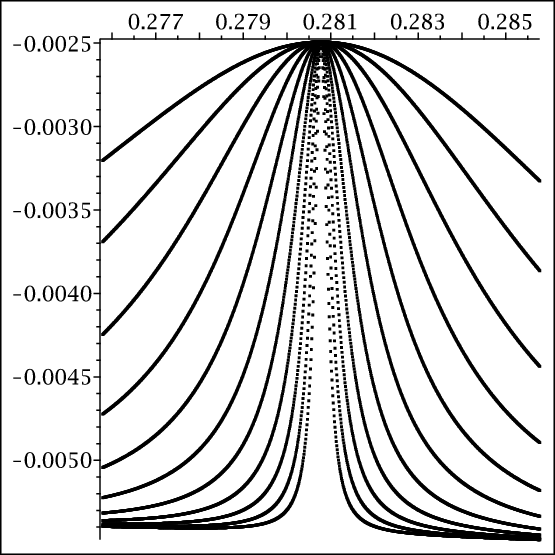}
    \includegraphics[width=30mm]{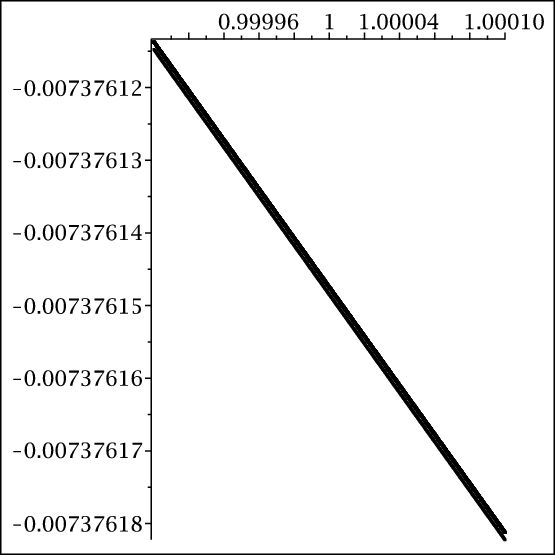}
    \caption{Simultaneous graphs of $\frac{ \beta_{j+1}(r)}{\beta_j(r)}$ for $j= 14, \dots, 24$ near  $r=\frac{\sqrt{17}-3}{4}$ and $r=1$.}
    \label{betazoomplot}
\end{figure}
\begin{figure}[p]
    \centering
    \includegraphics[width=50mm]{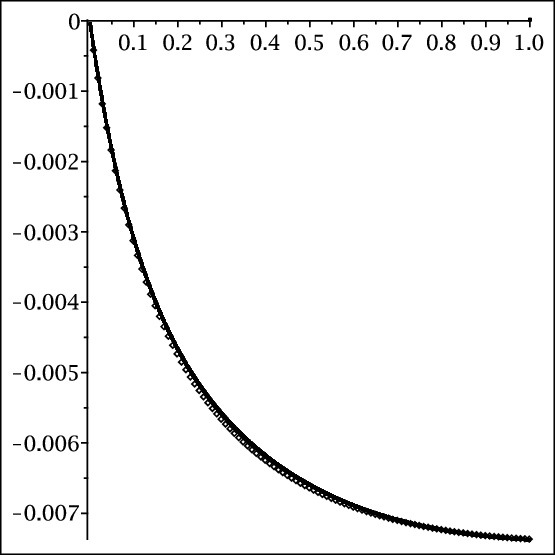}
\includegraphics[width=50mm]{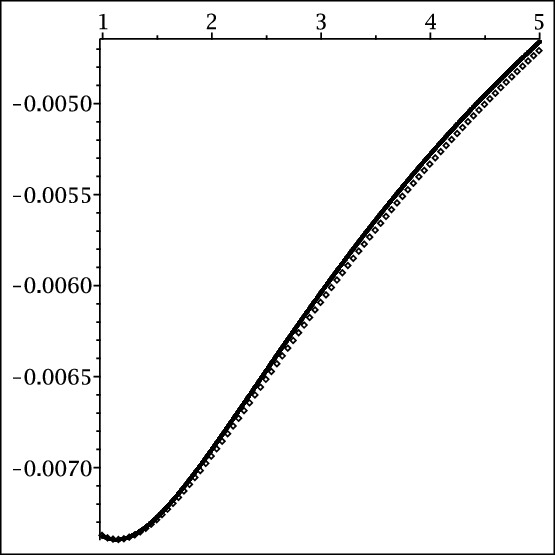}
\includegraphics[width=50mm]{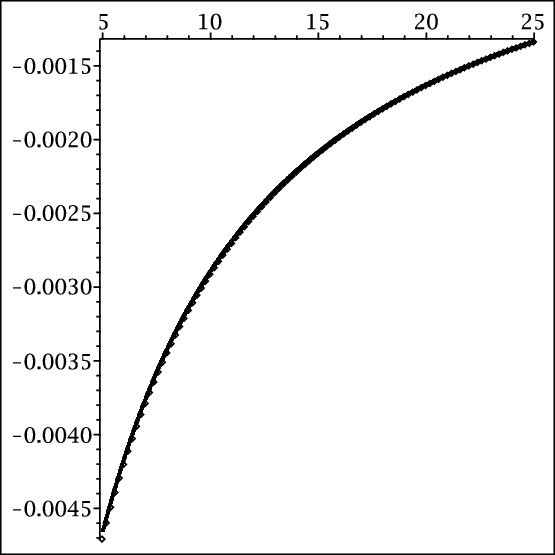}
    \caption{Simultaneous graphs of $\frac{ \alpha_{j+1}(r)}{\alpha_j(r)}$ for $j=39, \dots,49$ along with $\frac{-1}{2\lambda_3(r)-\lambda_2(r)}$. The black lines correspond to $\frac{\alpha_{j+1} (r) }{ \alpha_j(r)}$, and the dots correspond to $\frac{-1}{2\lambda_3(r)-\lambda_2(r)}$.}
    \label{comboplot}
\end{figure}
\begin{figure}[p]
    \centering
    \includegraphics[width=50mm]{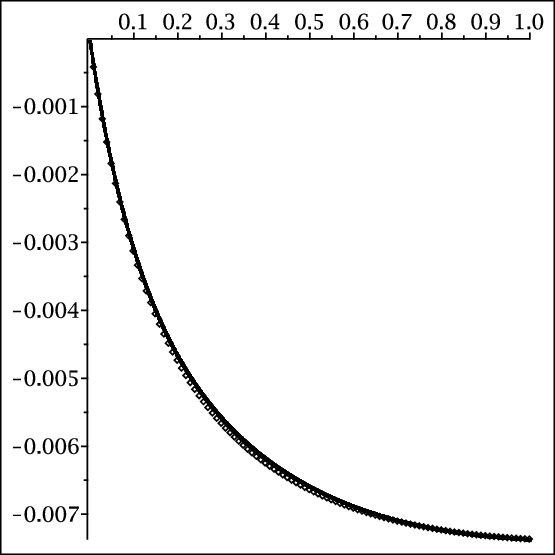}
\includegraphics[width=50mm]{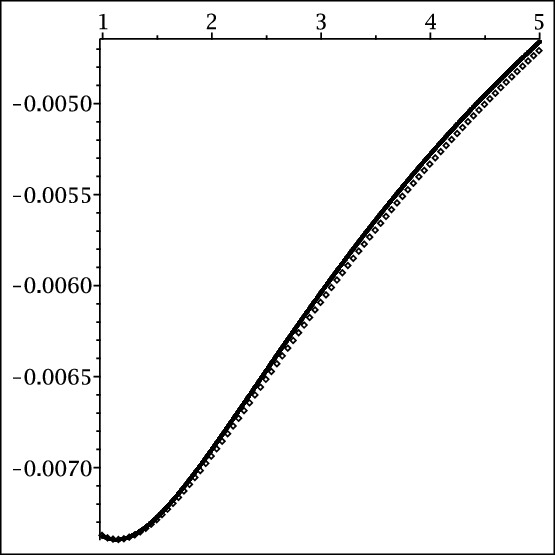}
\includegraphics[width=50mm]{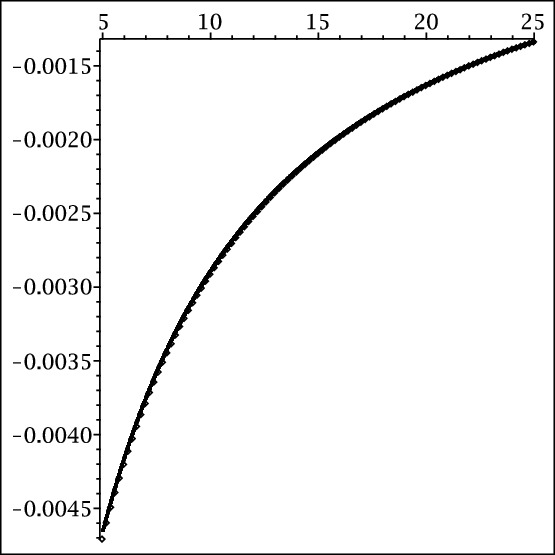}
    \caption{Simultaneous graphs of $\frac{ \beta_{j+1}(r)}{\beta_j(r)}$ for $j=39, \dots, 49$ along with $\frac{-1}{2\lambda_3(r)-\lambda_2(r)}$. The black lines correspond to $\frac{\beta_{j+1} (r) }{ \beta_j(r)}$, and the dots correspond to $\frac{-1}{2\lambda_3(r)-\lambda_2(r)}$.}
    \label{betacomboplot}
\end{figure}

Using the systems of equations (2.17), (2.18), and (2.19), one can obtain explicit formulas for the functions $\alpha_j(r)$, $\beta_j(r)$, and $\gamma_j(r)$, which are all rational functions of $r$. The first nonconstant terms in these sequences are 
$$ \alpha_2(r) = \frac{r(45r^4+233r^3+420r^2+305r+77)}{9(324r^6+2160r^5+5543r^4+7016r^3+4712r^2+1620r+225)},$$
$$\beta_1(r)= -\frac{r (81 r^4+495 r^3+1066 r^2+947 r+291)}{9 (162 r^5+999 r^4+2272 r^3+2372 r^2+1170 r+225)},
$$
$$ \gamma_1(r) =\frac{ r (27 r^3+116 r^2+155 r+62)}{3 (162 r^5+999 r^4+2272 r^3+2372 r^2+1170 r+225)},$$
and beyond these terms the expressions are too large to fit on the page and prohibitively complicated, as the degrees of the both the numerators and denominators grow rapidly. In fact, computing the explicit formulas of $\alpha_j(r)$, $\beta_j(r)$, and $\gamma_j(r)$ for $j$ up to about 15 seemed to be pushing the limit of the available computational power. However, computing $\alpha_j(r), \beta_j(r)$, and $\gamma_j(r)$ for a fixed value of $r$ is much more computationally efficient for larger values of $j$, so the strategy employed here to investigate the behavior of the sequences $\alpha_j(r)$, $\beta_j(r)$, and $\gamma_j(r)$ is to graph these functions by computing their values recursively for many fixed values of $r$, rather than trying to work with the explicit formulas. These graphs were made for the first 50 terms of each of these sequences, $j= 0, \dots, 49$. \par 

In addition to graphing the functions $\alpha_j(r)$, $\beta_j(r)$, and $\gamma_j(r)$ themselves, investigating the ratios between consecutive terms of these sequences is necessary to gain insight about their long term behavior and decay rates. The same method of sampling many values of $r$ was applied to make graphs of $\frac{ \alpha_{j+1} (r)}{\alpha_j(r)}$, $\frac{\beta_{j+1}(r) } {\beta_j(r)}$, and $\frac{ \gamma_{j+1} (r) } {\gamma_j(r)}$ for $j =0, \dots, 49. $ \par 

Patterns emerge in all of these graphs. Beyond about $j=12$, the graphs of $\alpha_j(r)$ all appear qualitatively similar, besides alternating in sign. The $\beta_j(r)$ seem to fall into a similar pattern after about $j=9$. The graphs of $\gamma_j(r)$ all seem qualitatively similar for $j<28$, but then, surprisingly, they undergo a dramatic change before falling into a different pattern. For $j$ sufficiently large, $\alpha_j(r)$ has roots near $r=0.2807764$ and $r=1$, and it appears that these roots converge to $\frac{ \sqrt{17} - 3 }{4} $ and $1$. Once the $\beta_j(r)$ establish a pattern, they have no roots for $r>0$. For large $j$, $\gamma_j(r)$ has several roots, one of which appears to converge to 1. \par 

It appears that the ratios of consecutive terms of these three sequences do not approach 0, except where the functions have roots. This is surprising, since both $\alpha_j(1)$ and $\gamma_j(1)$ are known to approach 0 faster than any exponential. This behavior does not extend to all values of $r$, however, and the fast decay of $\alpha_j(1)$ and $\gamma_j(1)$ is explained by the fact that $\alpha_j(r)$ and $\gamma_j(r)$ both have a root appearing to converge to $1$. \par 

The ratios $\frac{ \alpha_{j+1}(r)}{ \alpha_j(r)} $ and $\frac{\gamma_{j+1} (r) }{ \gamma_j(r)} $ are not continuous functions, as they have vertical asymptotes at the roots of $\alpha_j(r)$ and $\gamma_j(r)$. But, assuming the convergence of the sequences of functions $\{ \frac{ \alpha_{j+1} (r) }{ \alpha_j(r) } \}$ and $\{ \frac{ \gamma_{j+1} (r) }{ \gamma_j(r) } \}$ to functions that are continuous away from the limits of the roots of the denominators, if these roots do indeed converge it would imply that the limits of these ratios converge to functions which have only removable discontinuities corresponding to each of these roots. On the other hand, $\frac{ \beta_{j+1} (r) }{ \beta_j(r)} $ is continuous, since $\beta_j(r) \neq 0$ for $r \in (0 , \infty)$. Interestingly, $\frac{ \beta_{j+1} (r) }{\beta_j(r)} $ has a local maximum near $r=0.2807764$ which appears to converge to $\frac{ \sqrt{17} -3}{4} $. It seems that at this maximum, $\frac{ \beta_{j+1} (r) }{\beta_j(r)}$ attains a value of about $-0.0025$. The width of this peak decreases as $j$ increases, and it appears that the local minimum which occurs at the base of this peak also converges to $\frac{ \sqrt{17}-3}{4}$. This peak gets so thin that it does not appear in the numerical plots for large $j$. If these extrema do indeed converge, it implies that the limit of $\frac{ \beta_{j+1} (r) } {\beta_j(r) } $, if it exists, has at least a removable discontinuity at $r=\frac{ \sqrt{17}-3}{4} $. \par 

The functions $\frac{ \alpha_{j+1}(r)}{ \alpha_j(r)} $ and $\frac{ \beta_{j+1} (r) } {\beta_j(r) } $ appear remarkably similar for large $j$. Moreover, both of these sequences of functions closely approach the function 
$\frac{-1} { 2 \lambda_3 (r) - \lambda_2(r)}$ where $\lambda_2(r)$ is the Neumann eigenvalue of $\Delta_r$ which is the second smallest nonzero Neumann eigenvalue when $r<1$ and $\lambda_3(r)$ is the third smallest Neumann eigenvalue when $r<1$. It is interesting to note that the multiplicity of $\lambda_3(r)$ is 2 and $\lambda _2(r)$ has multiplicity 1, corresponding with their respective coefficients in this expression. 

In addition to the graphs of $\alpha_j(r), \beta_j(r), $ and $\gamma_j(r)$ and the ratios of consecutive terms, it is interesting to look at graphs of the monomials themselves. These graphs were made by calculating the jets of the monomials at each point in $V_1$, and using this information to compute their values at each point in $V_2$, and then repeating this process to compute the values at each point in $V_5$. \par 

Observing these graphs further demonstrates that $r=1$ is an exceptional value. When $r=1$, the graphs of $P_{j,1}^{(r)}$ and $P_{j,2}^{(r)}$ have noticeable qualitative differences, but for $r \neq 1$, and $j$ sufficiently large, $P_{j,1}^{(r)}$ and $P_{j,2}^{(r)}$ look quite similar, even for $r$ close to 1. \par 

Making these graphs requires strenuous computation, so it has only been done for a small selection of values of $r$, and for relatively small $j$. \par 

In addition to the selection of graphs presented here, there are many more which can be found at \url{http://pi.math.cornell.edu/~reuspurweb/}.

\begin{figure}[p]
    \centering
    \includegraphics[width=30mm]{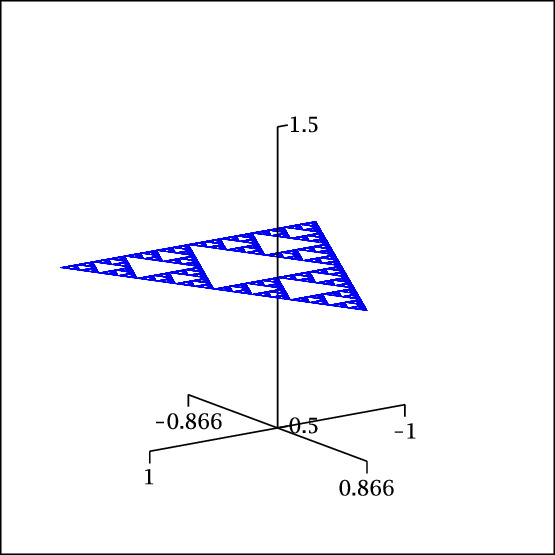}   
    \includegraphics[width=30mm]{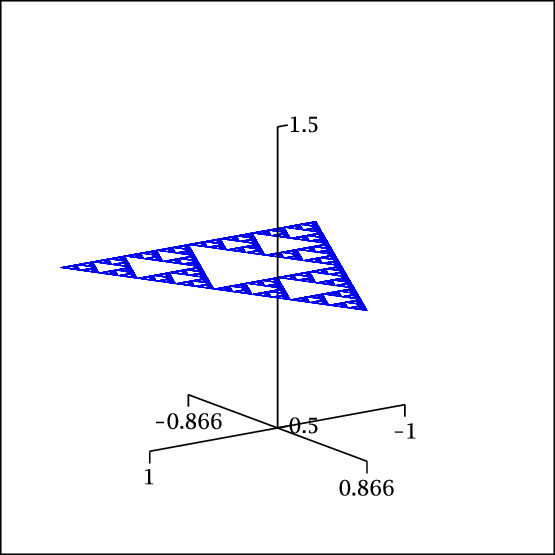}    
    \includegraphics[width=30mm]{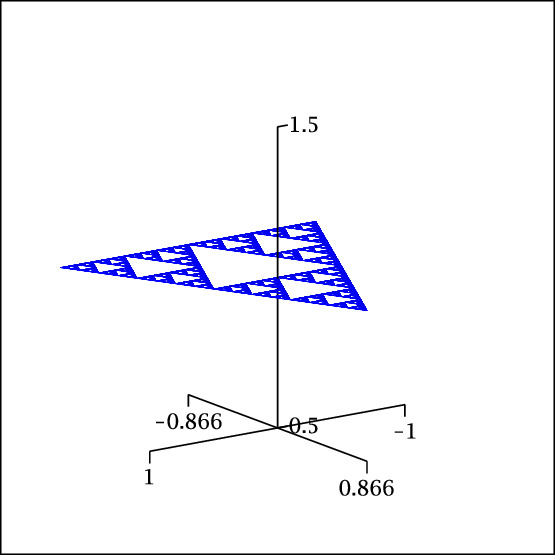}\\
    \includegraphics[width=30mm]{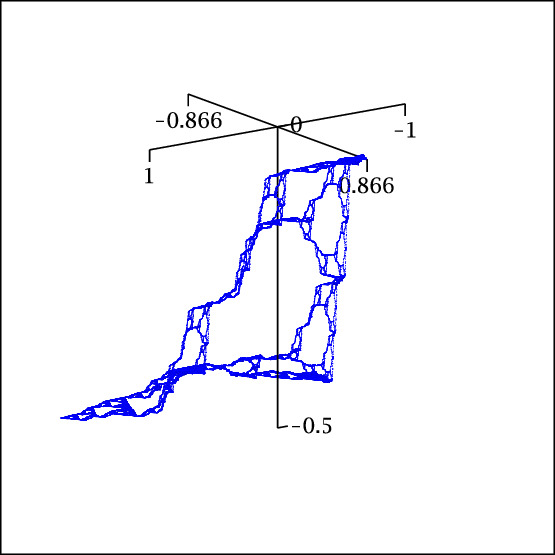}  
     \includegraphics[width=30mm]{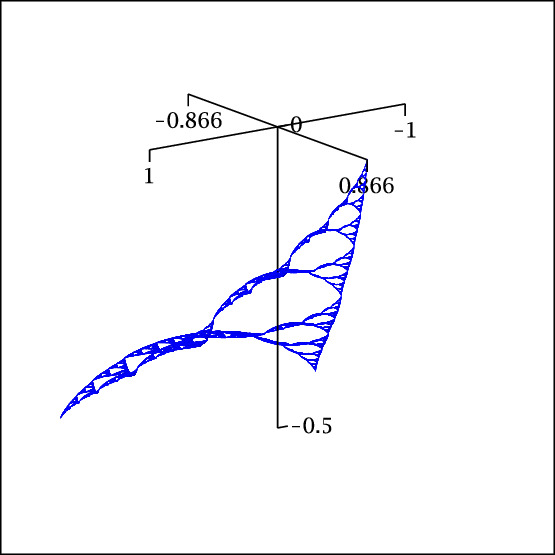}   
      \includegraphics[width=30mm]{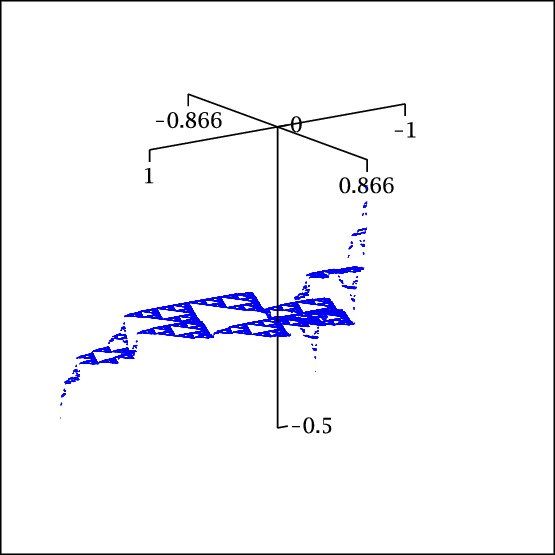}\\  
     \includegraphics[width=30mm]{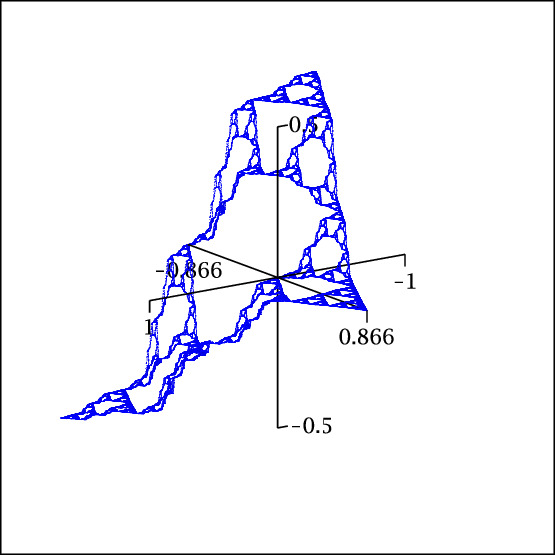}   
      \includegraphics[width=30mm]{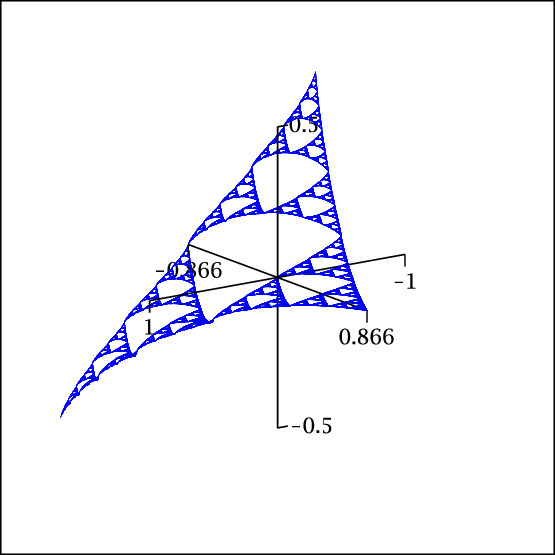}   
        \includegraphics[width=30mm]{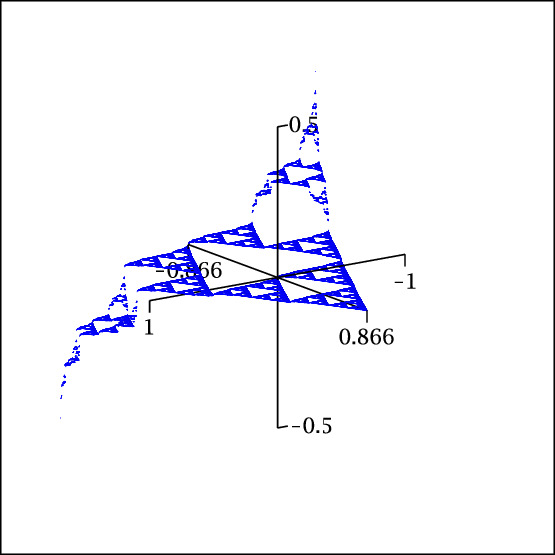}
    \caption{Graphs of $P_{0,1}^{(r)}$, $P_{0,2}^{(r)} \ , P_{0,3}^{(r)}$ (top to bottom) with $r=\frac1{10}$, $r=1$, $r=10$ (left to right).}
\end{figure}

\begin{figure}[p]
    \centering
    \includegraphics[width=30mm]{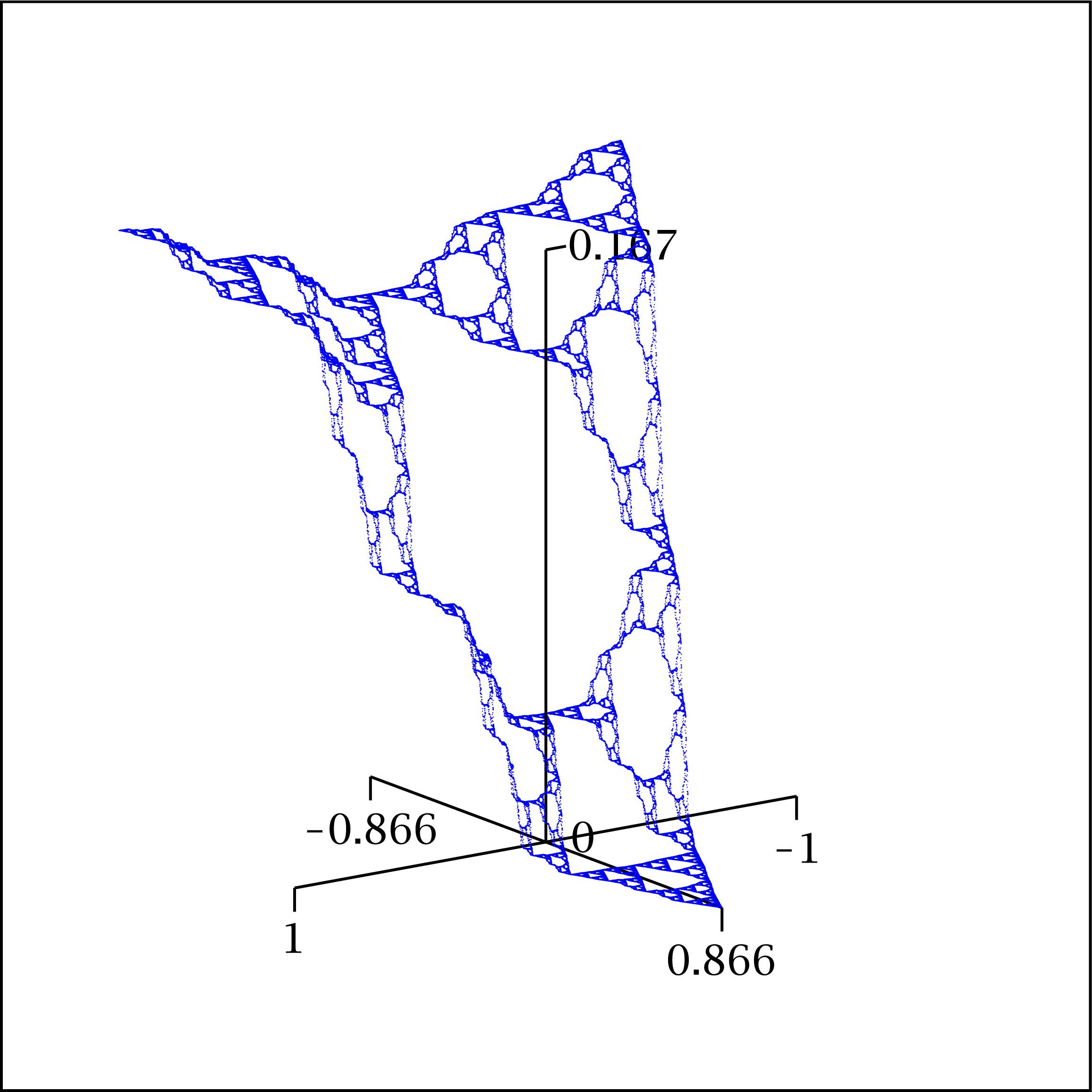}   
    \includegraphics[width=30mm]{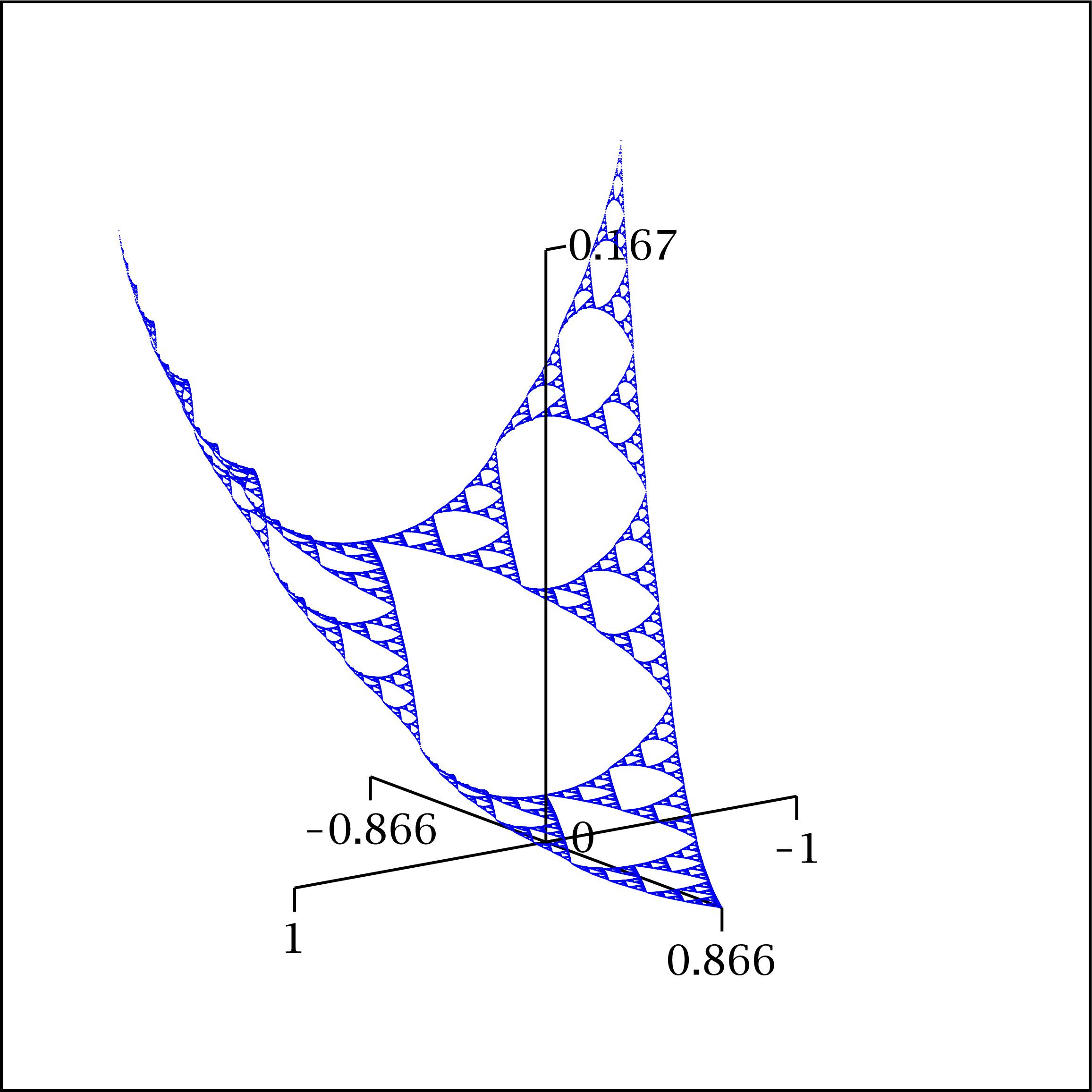}    
    \includegraphics[width=30mm]{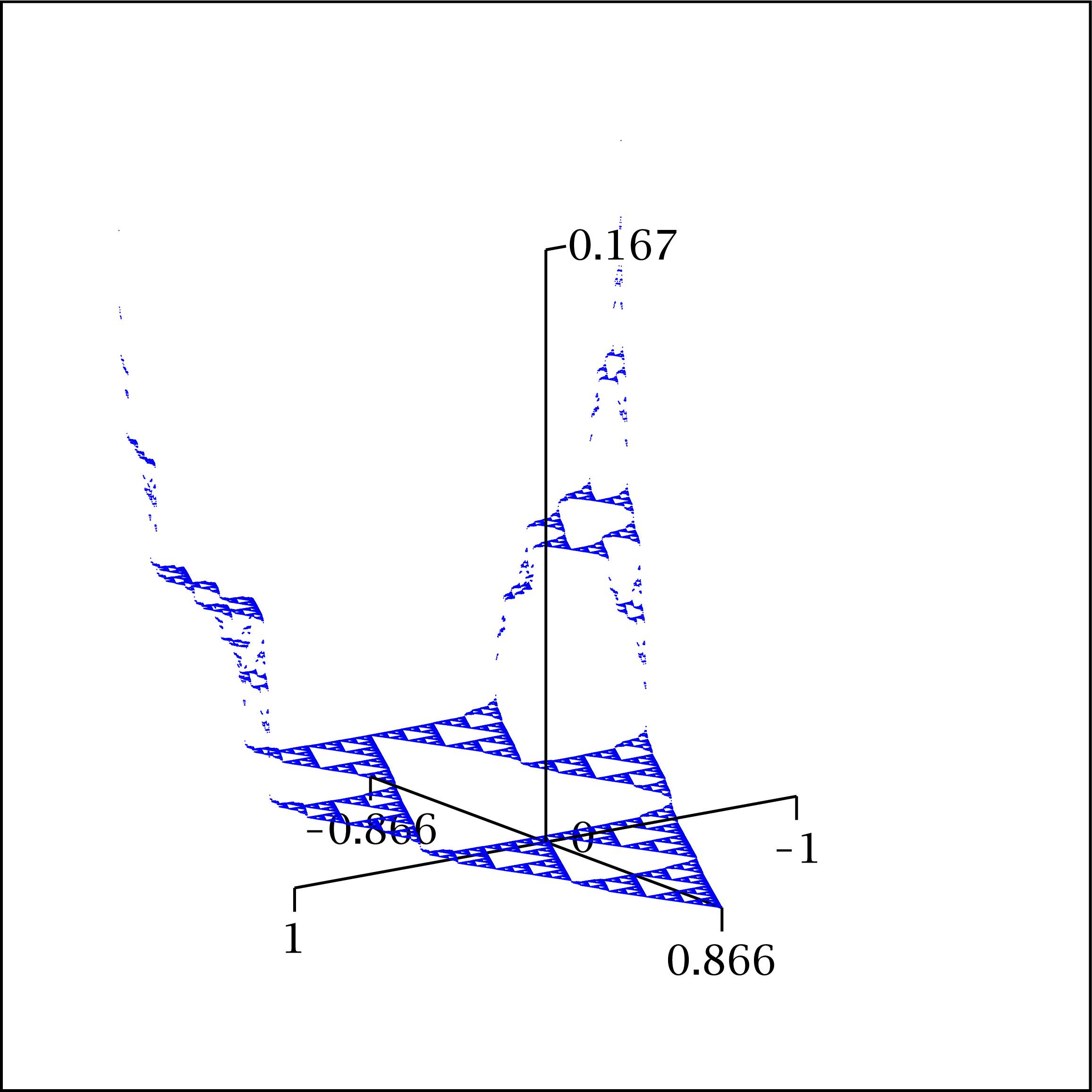}\\
    \includegraphics[width=30mm]{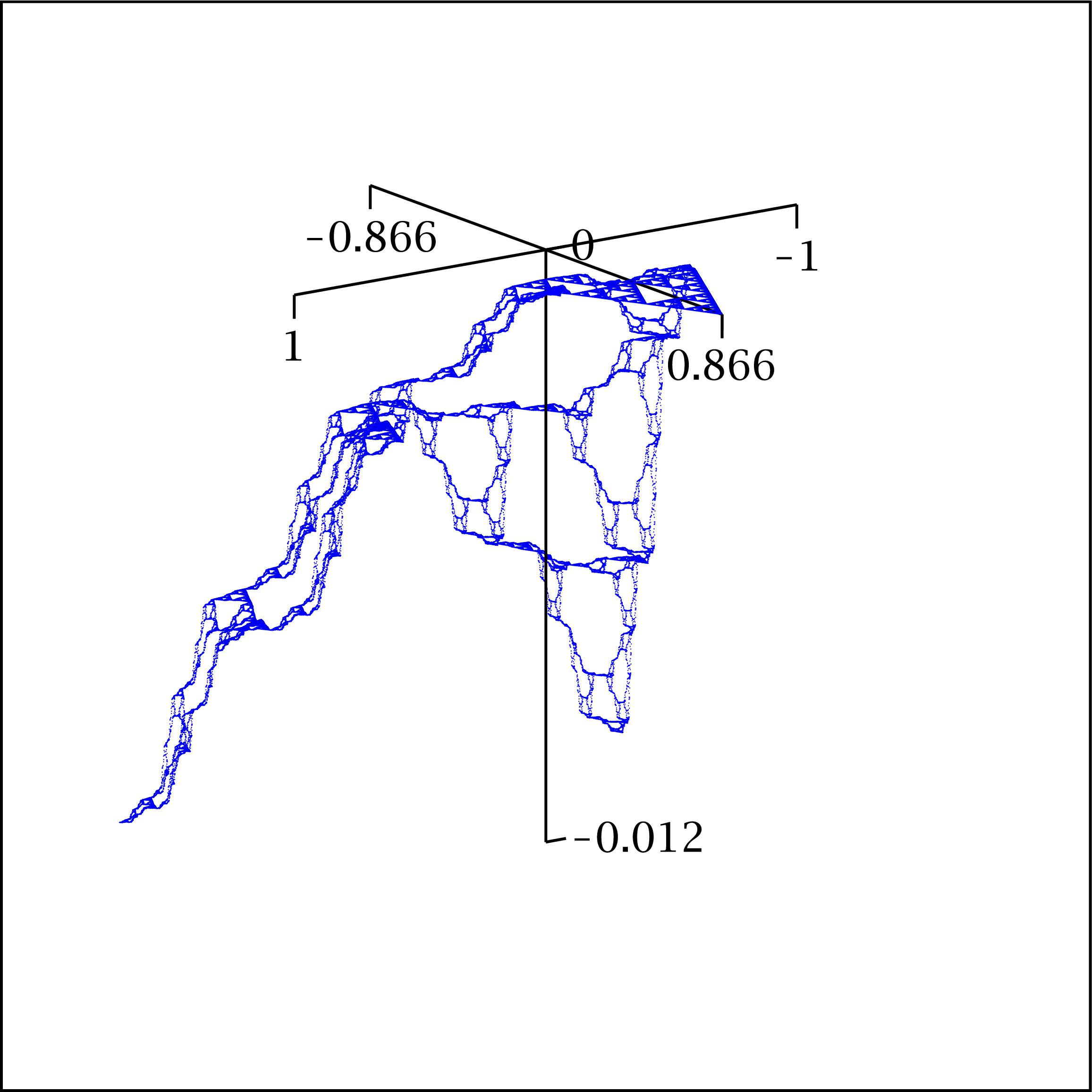}  
     \includegraphics[width=30mm]{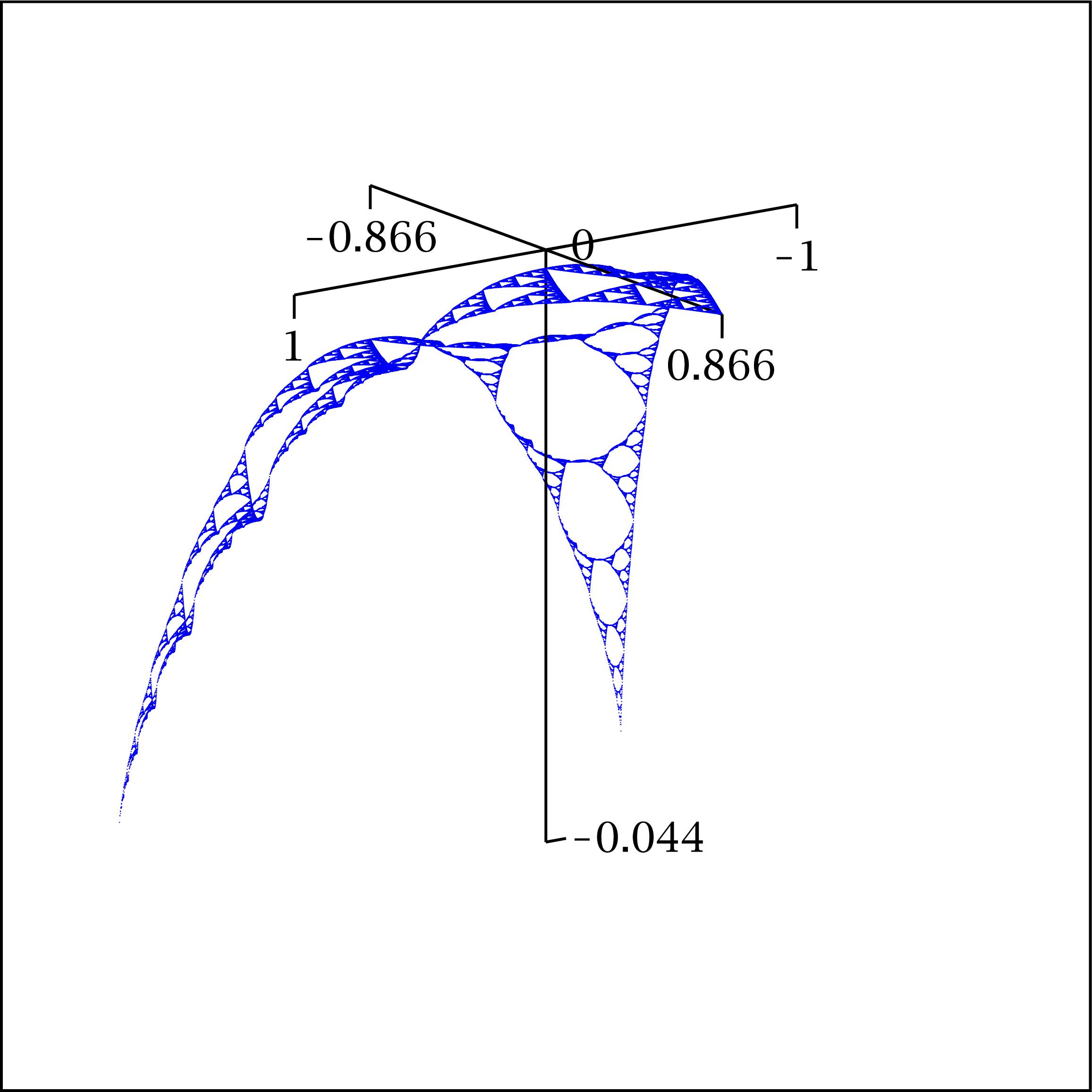}   
      \includegraphics[width=30mm]{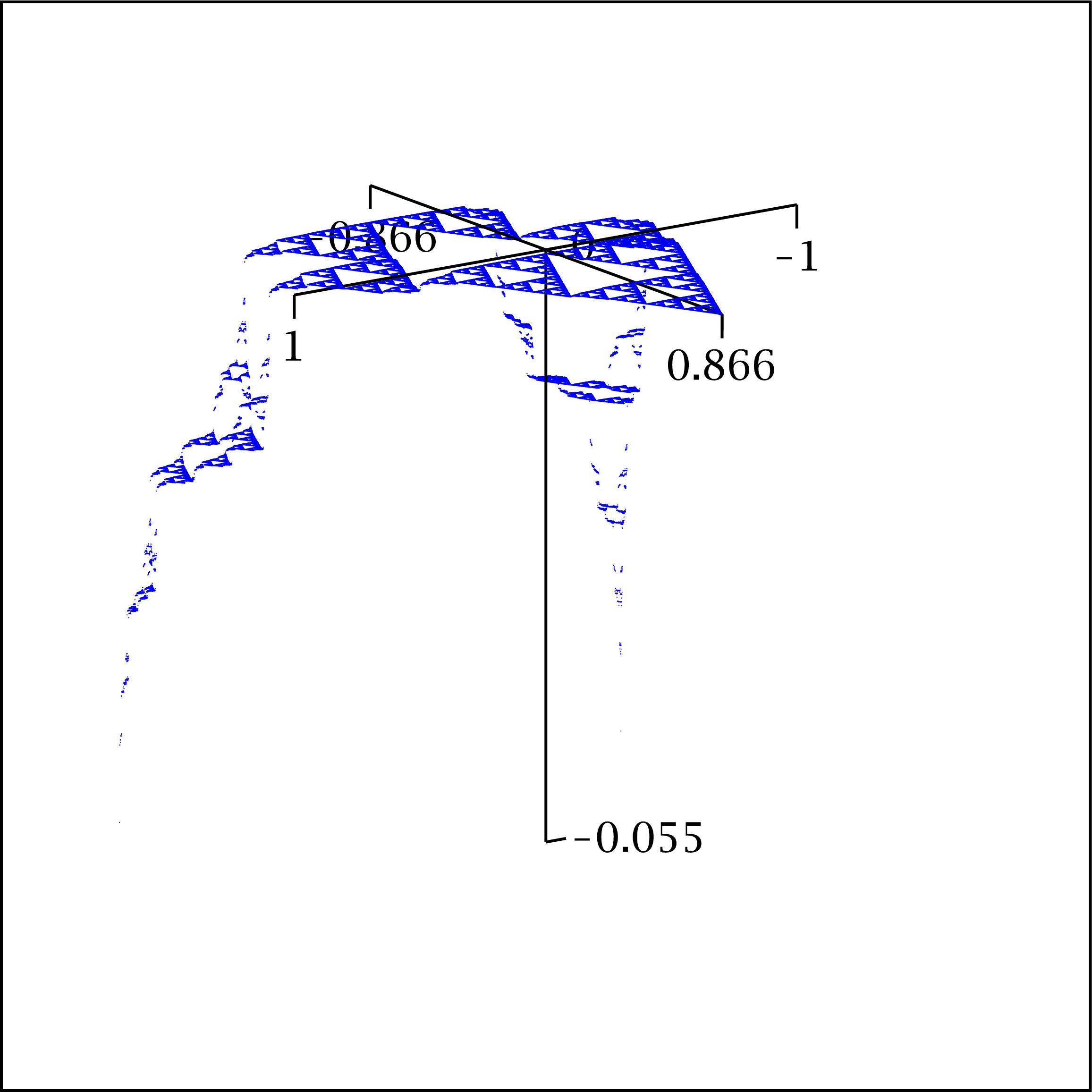}\\  
     \includegraphics[width=30mm]{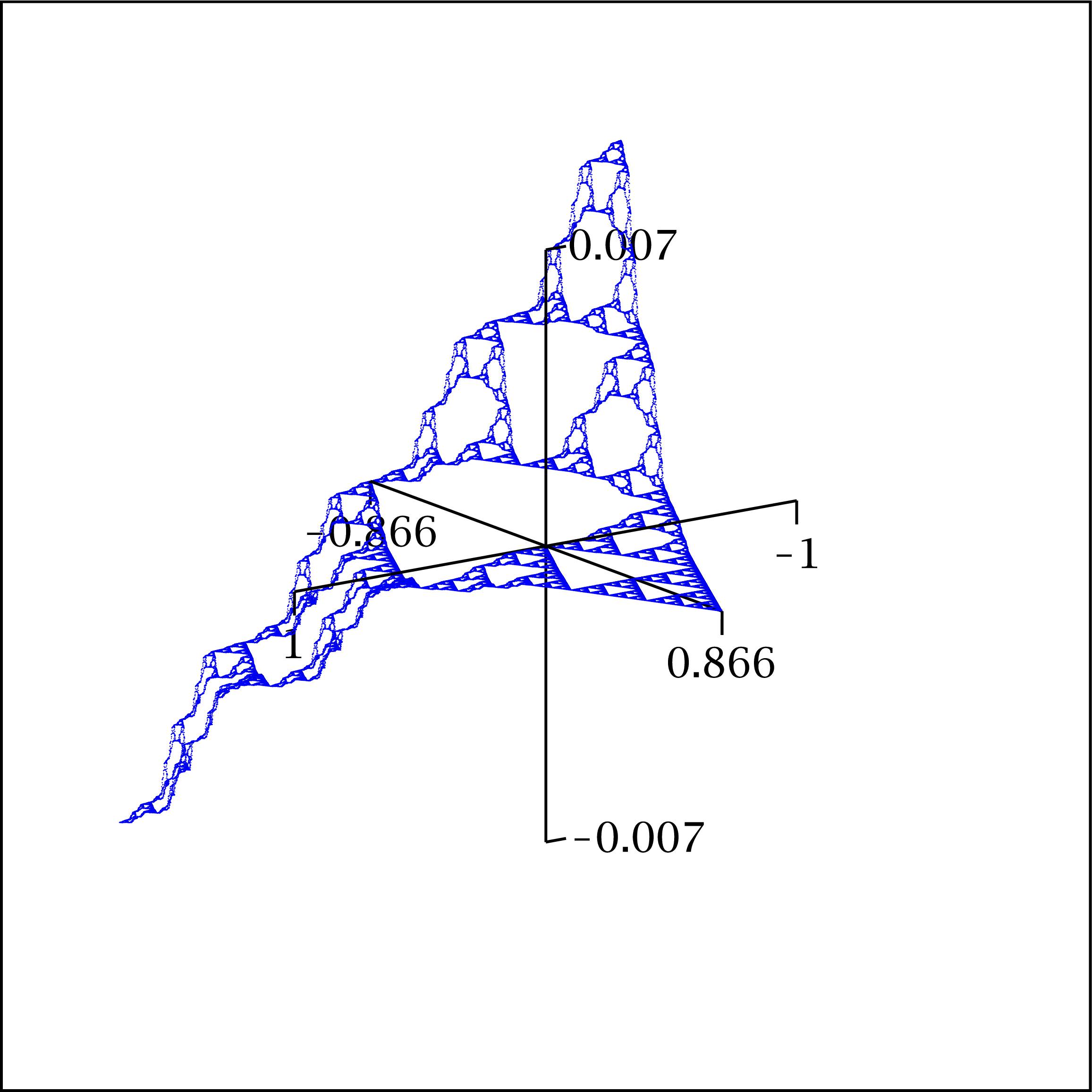}   
      \includegraphics[width=30mm]{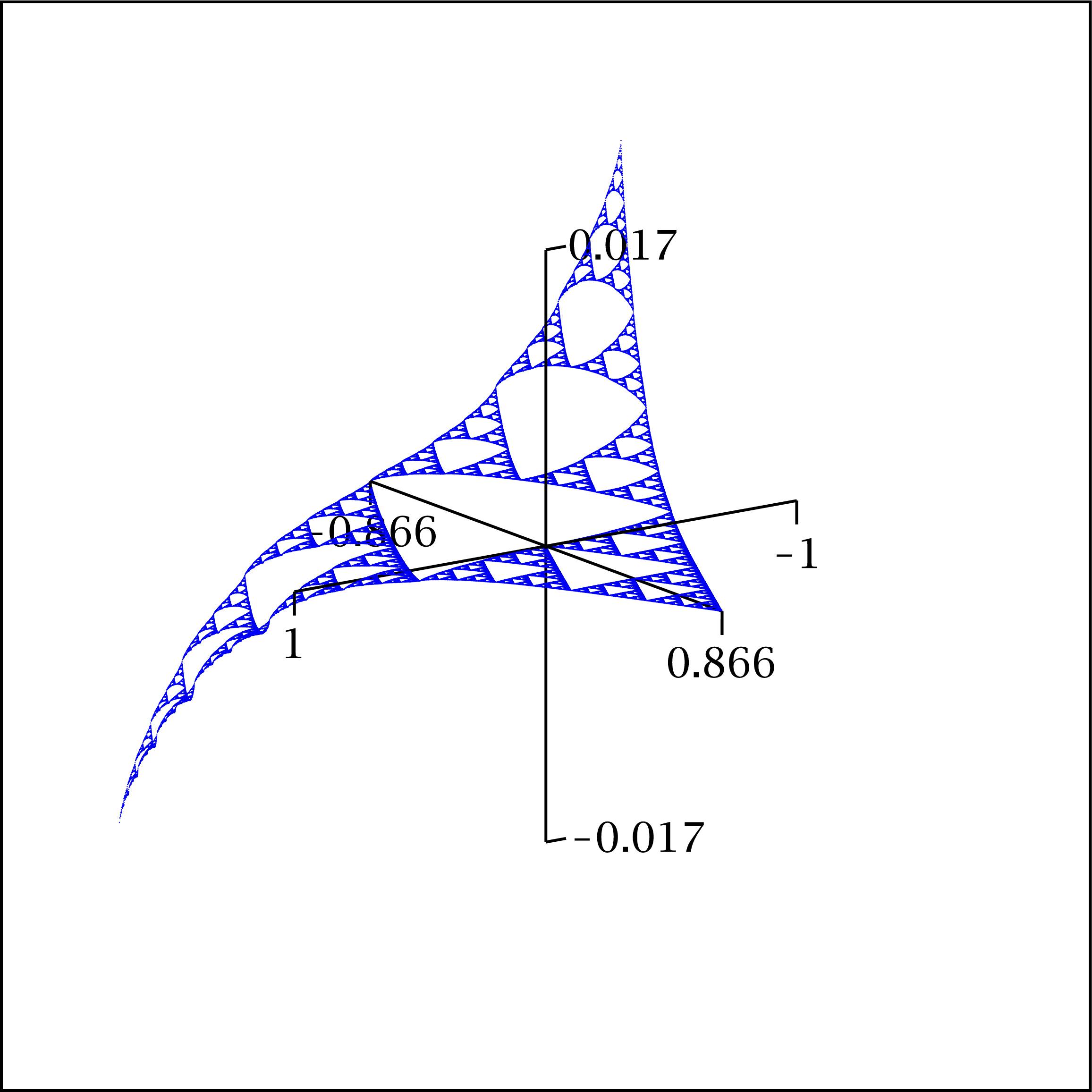}   
        \includegraphics[width=30mm]{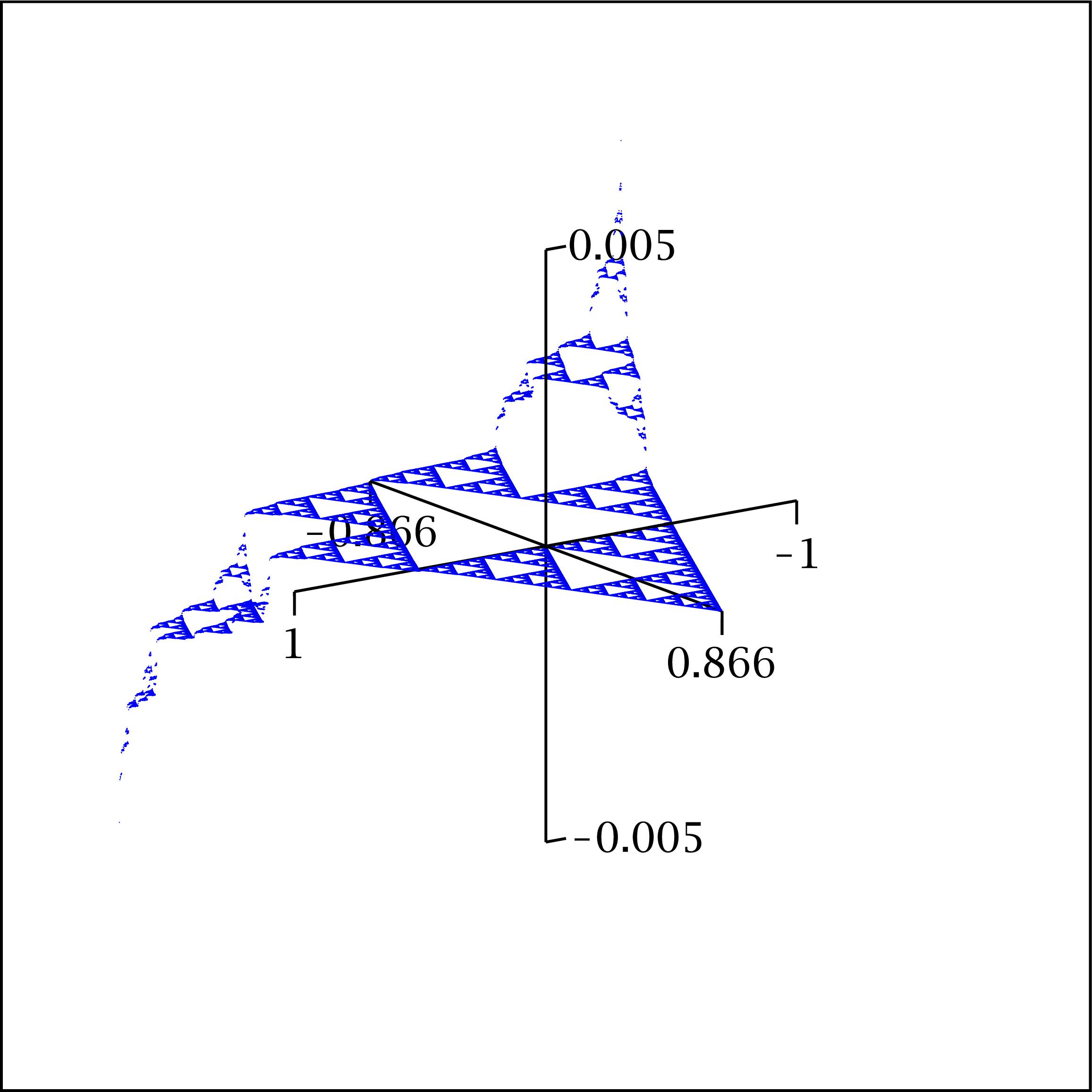}
    \caption{Graphs of $P_{1,1}^{(r)}$, $P_{1,2}^{(r)} \ , P_{1,3}^{(r)}$ (top to bottom) with $r=\frac1{10}$, $r=1$, $r=10$ (left to right).}
\end{figure}

\begin{figure}[p]
    \centering
    \includegraphics[width=30mm]{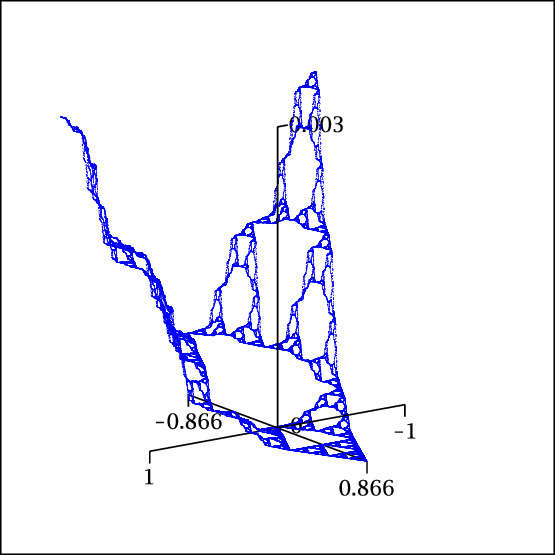}   
    \includegraphics[width=30mm]{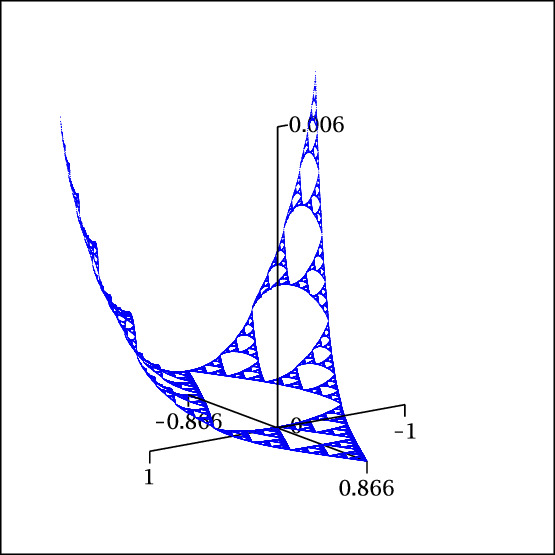}    
    \includegraphics[width=30mm]{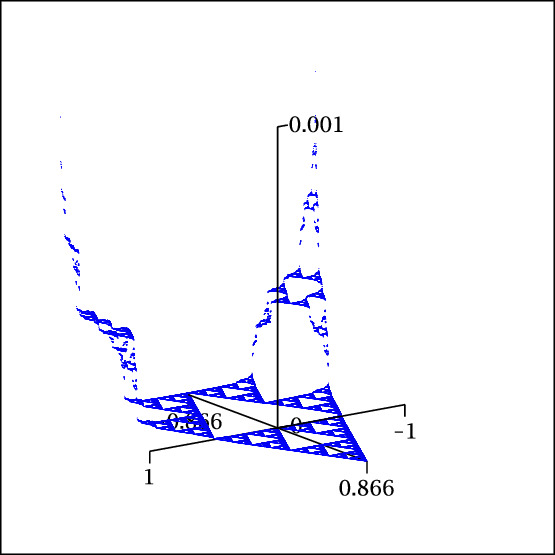}\\
    \includegraphics[width=30mm]{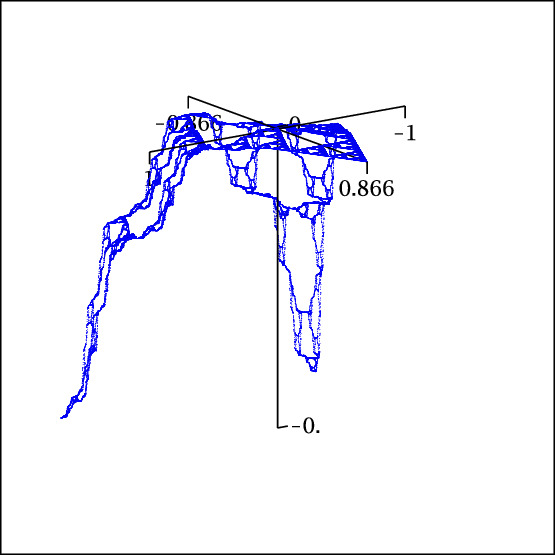}  
     \includegraphics[width=30mm]{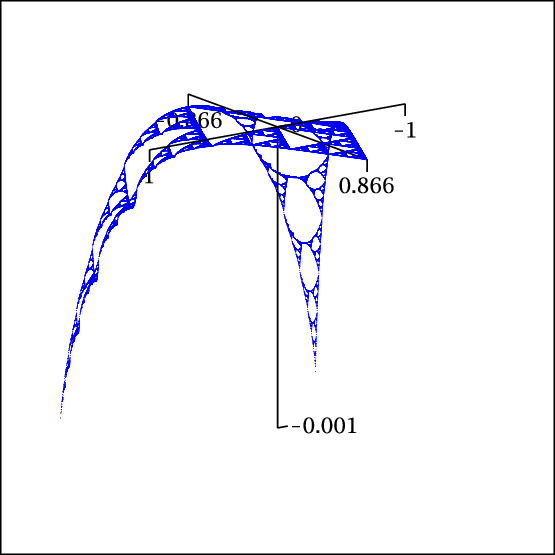}   
      \includegraphics[width=30mm]{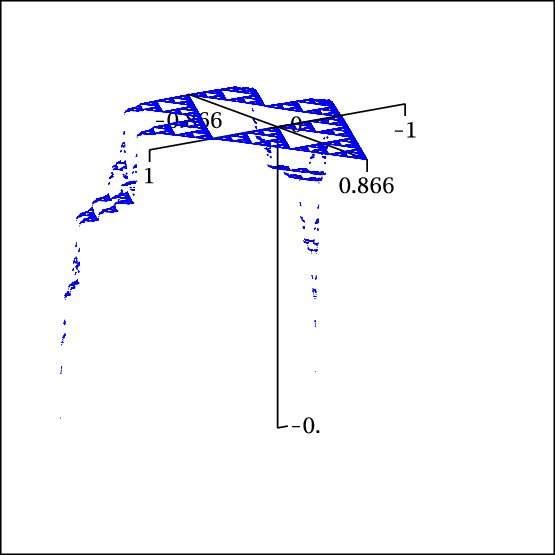}\\  
     \includegraphics[width=30mm]{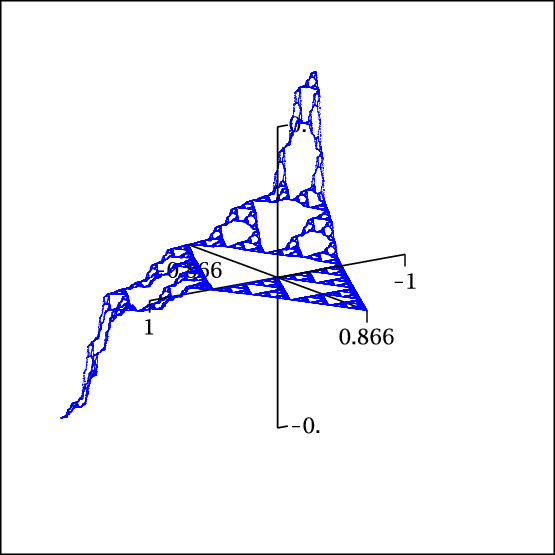}   
      \includegraphics[width=30mm]{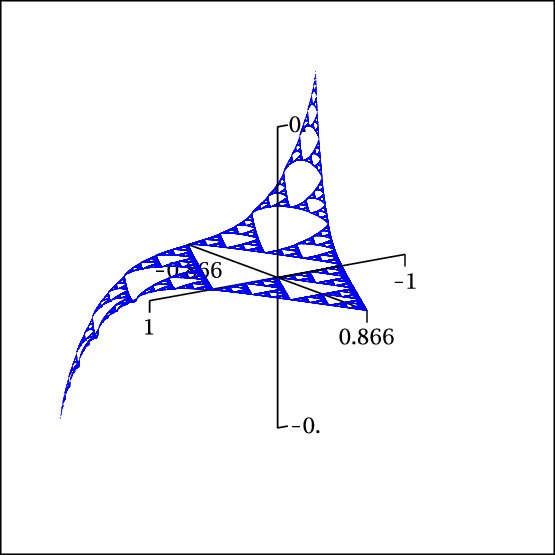}   
        \includegraphics[width=30mm]{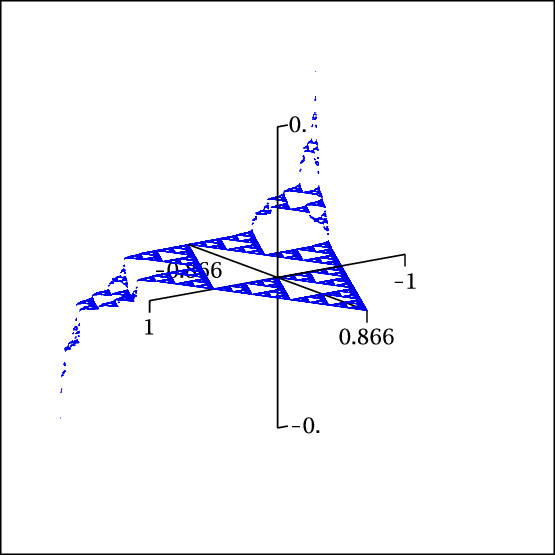}
    \caption{Graphs of $P_{2,1}^{(r)}$, $P_{2,2}^{(r)} \ , P_{2,3}^{(r)}$ (top to bottom) with $r=\frac1{10}$, $r=1$, $r=10$ (left to right).}
\end{figure}

\begin{figure}[p]
    \centering
        \includegraphics[width=30mm]{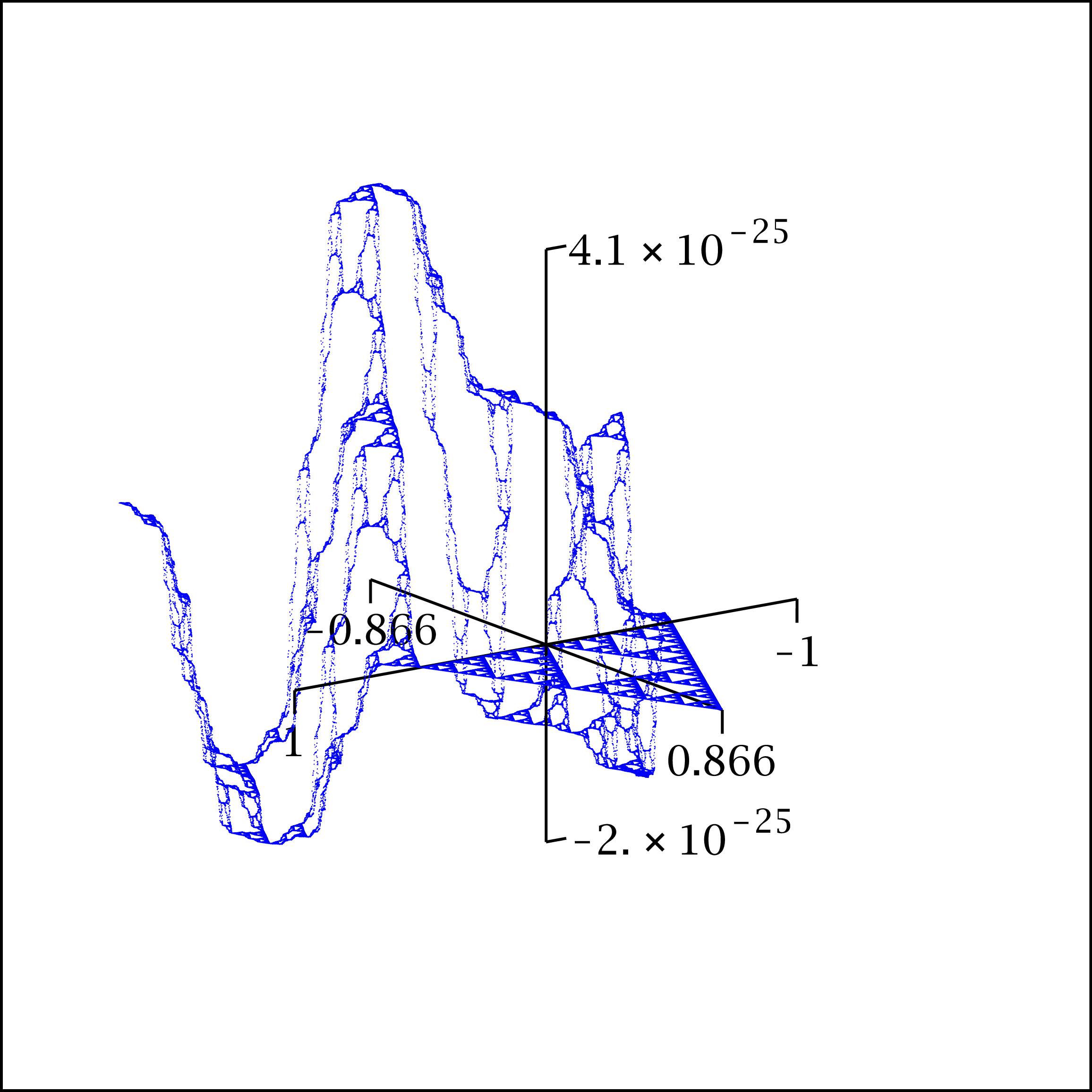}   
    \includegraphics[width=30mm]{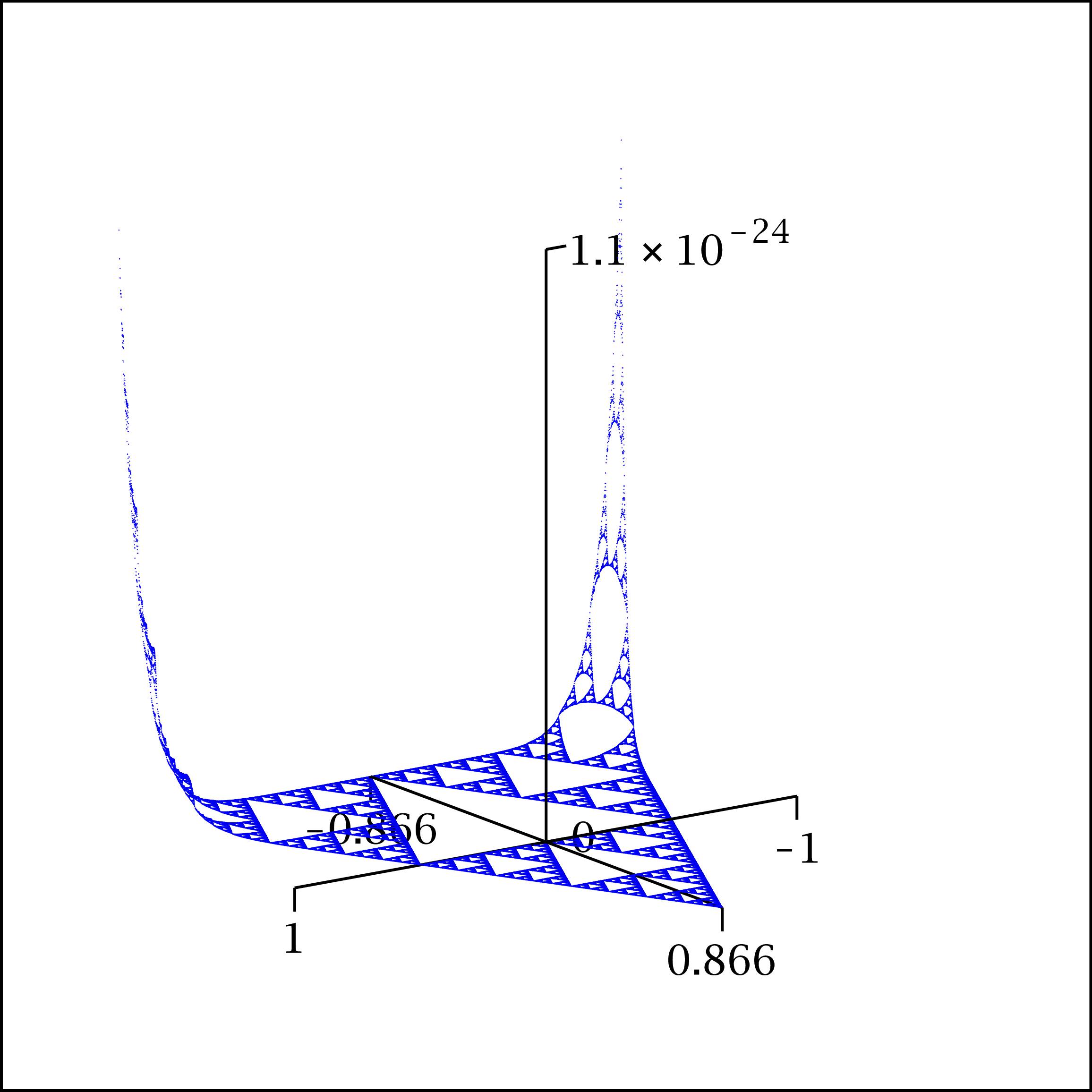}    
    \includegraphics[width=30mm]{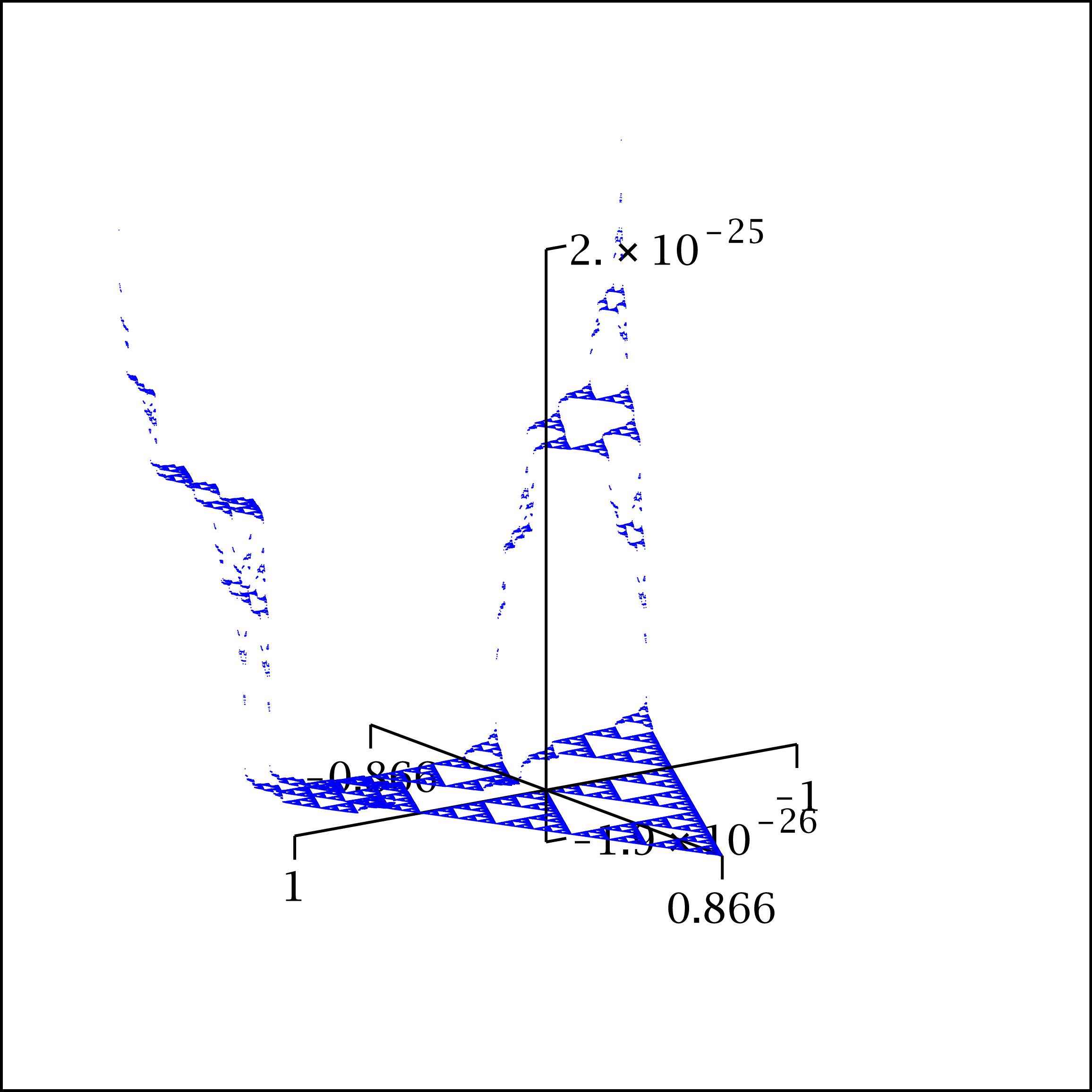}\\
    \includegraphics[width=30mm]{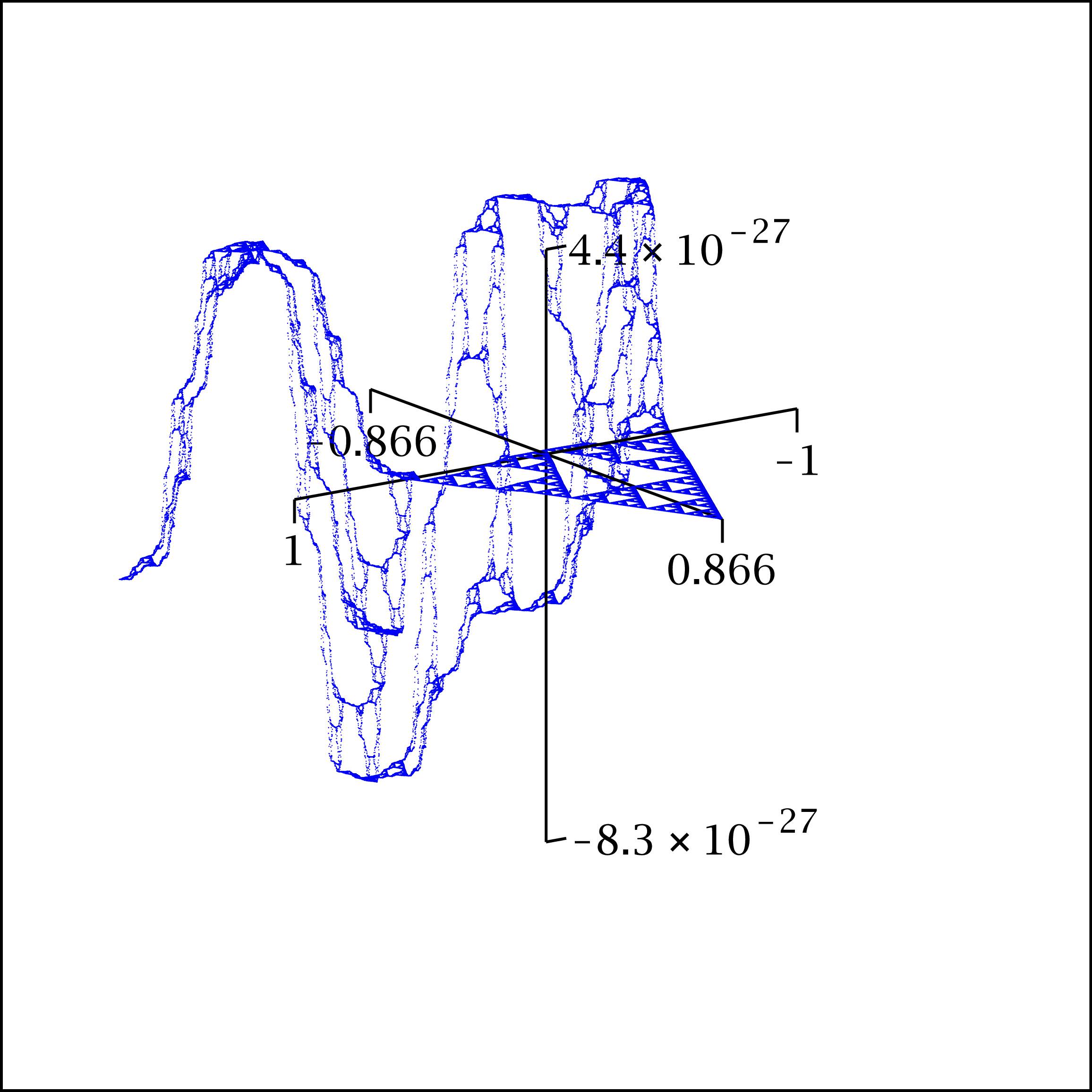}  
     \includegraphics[width=30mm]{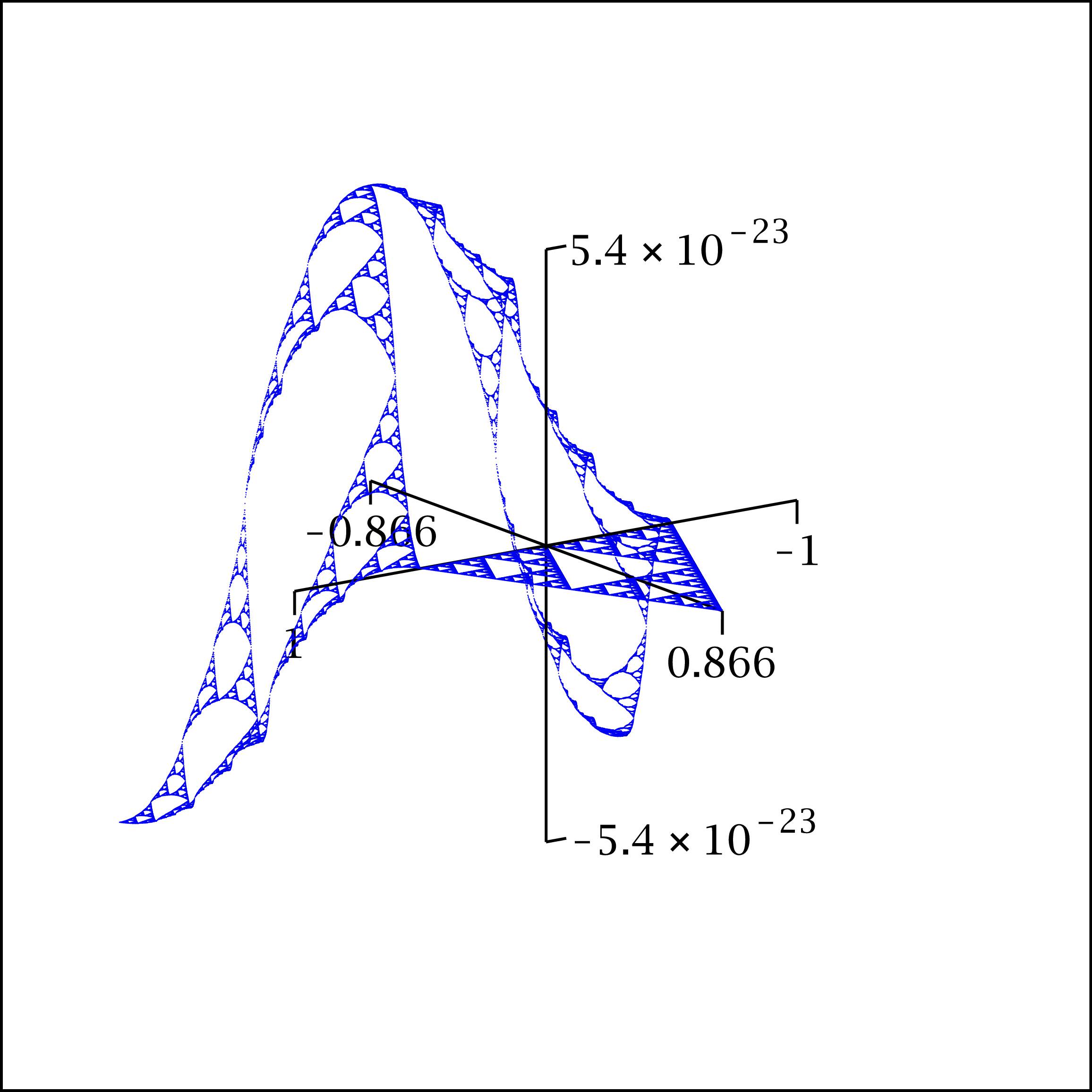}   
      \includegraphics[width=30mm]{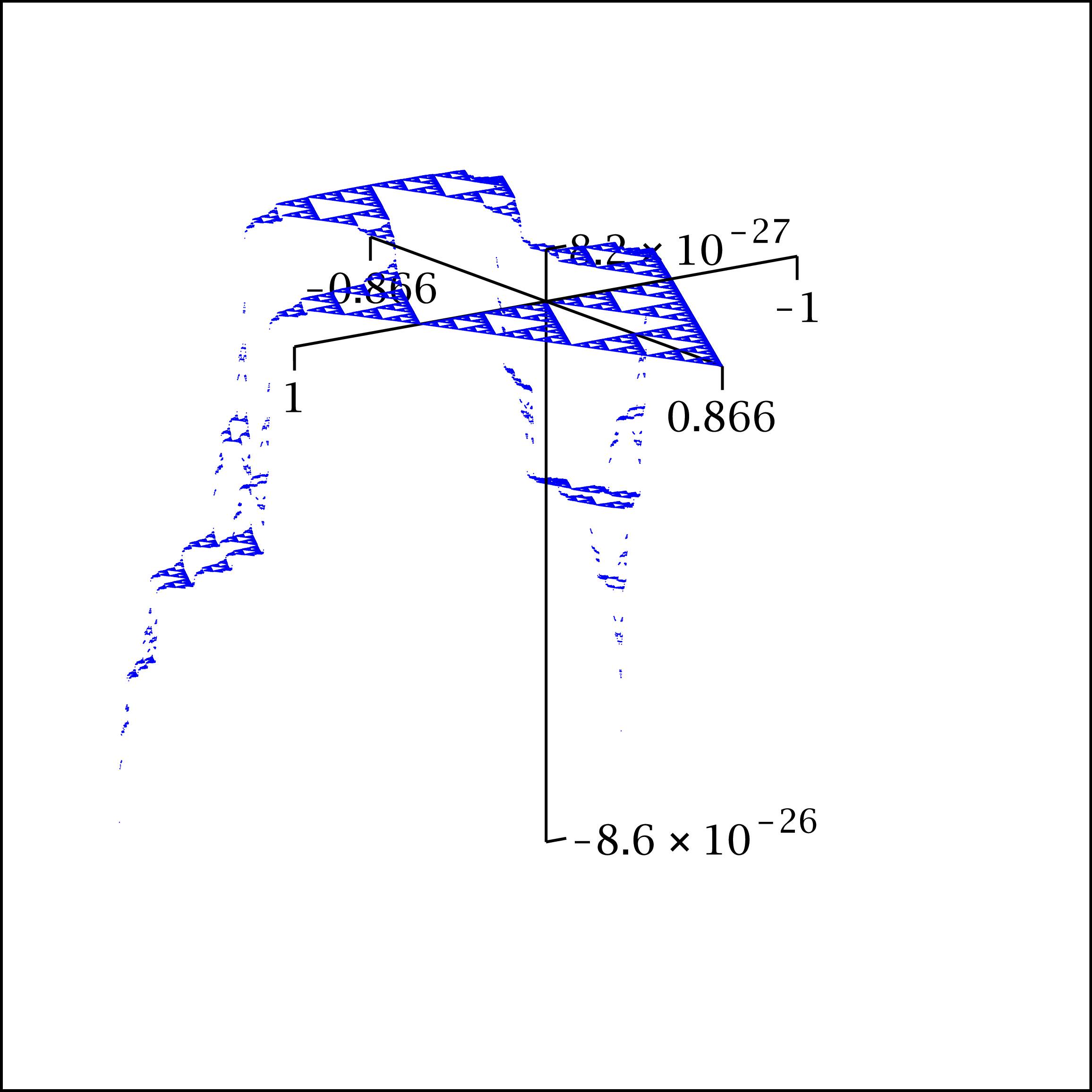}\\  
     \includegraphics[width=30mm]{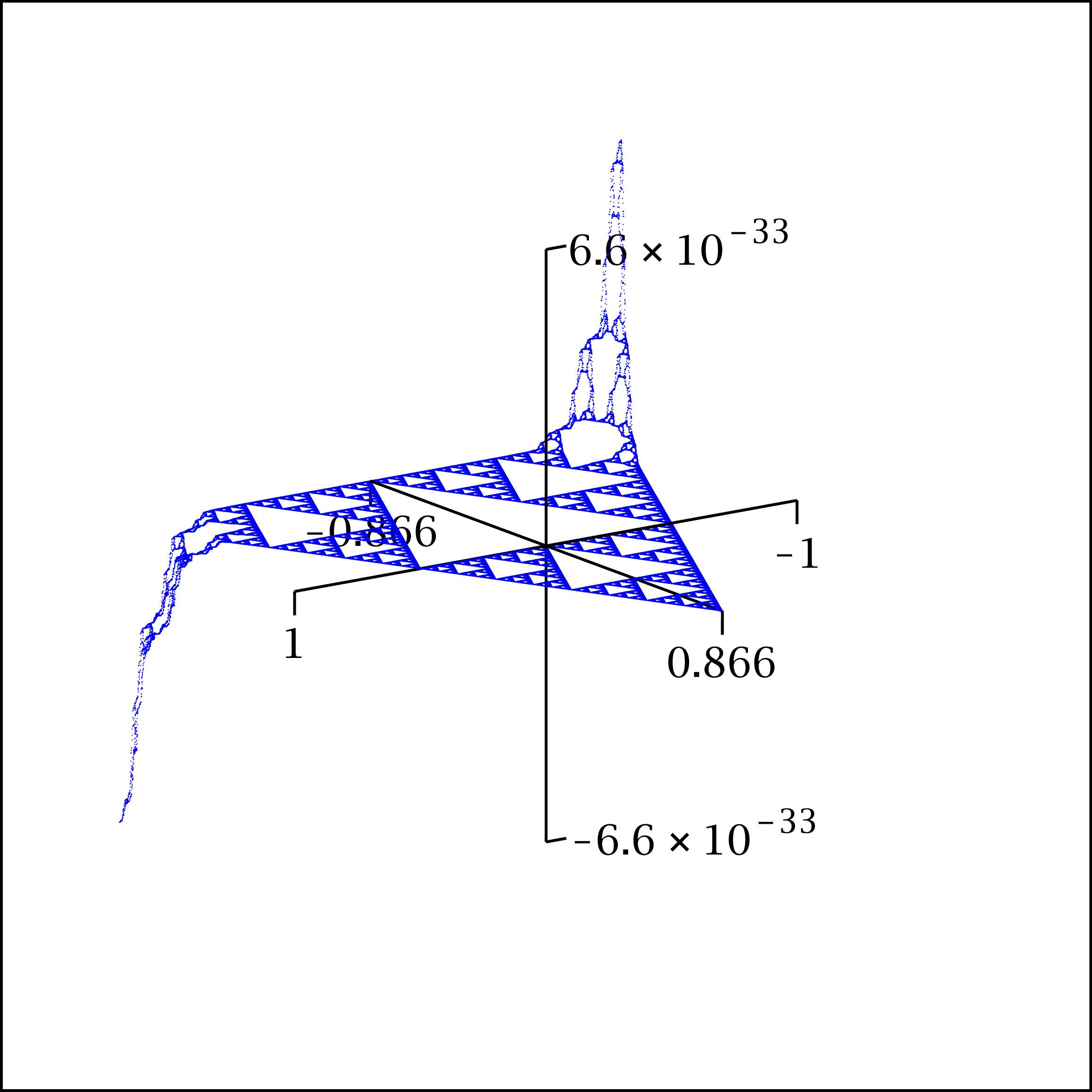}   
      \includegraphics[width=30mm]{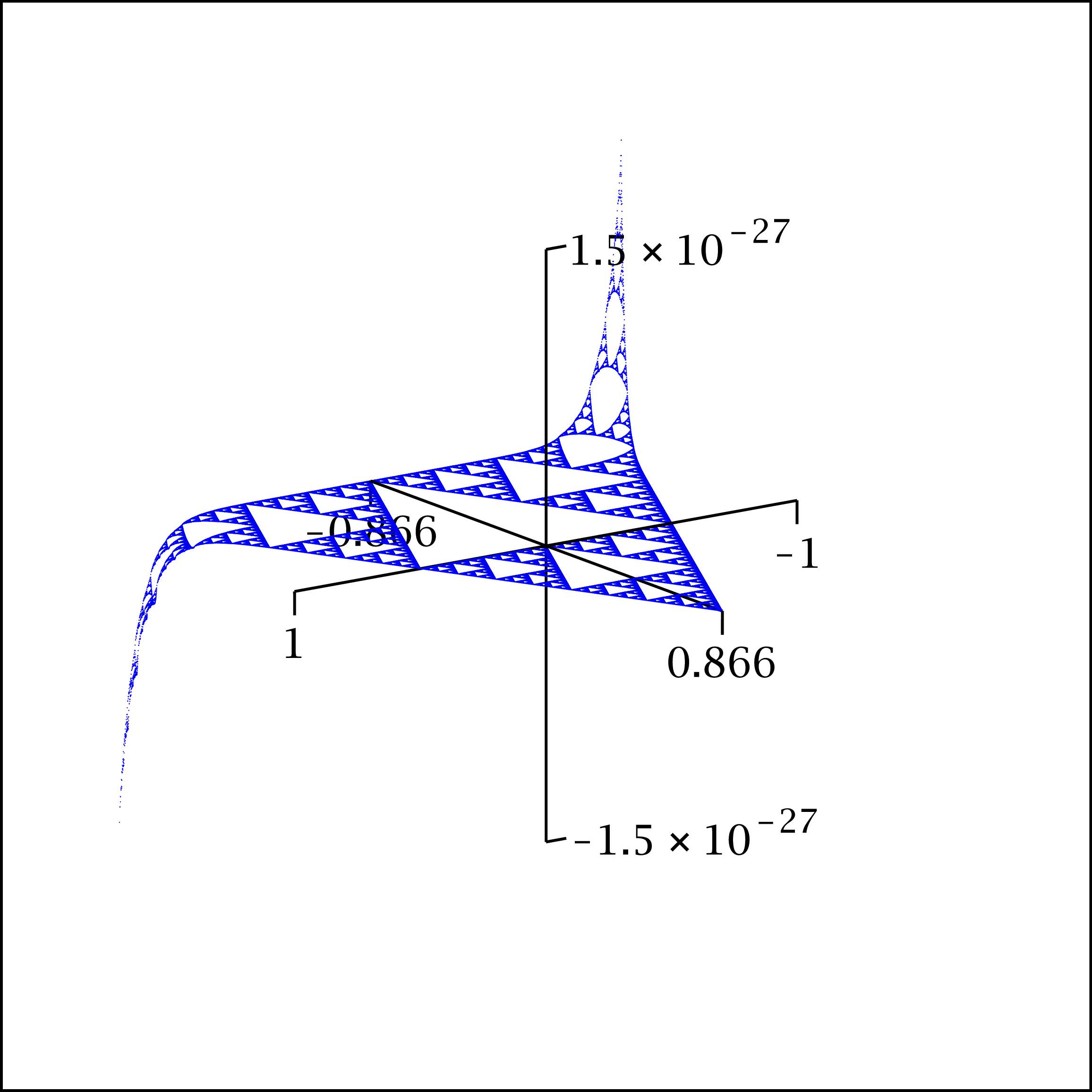}   
        \includegraphics[width=30mm]{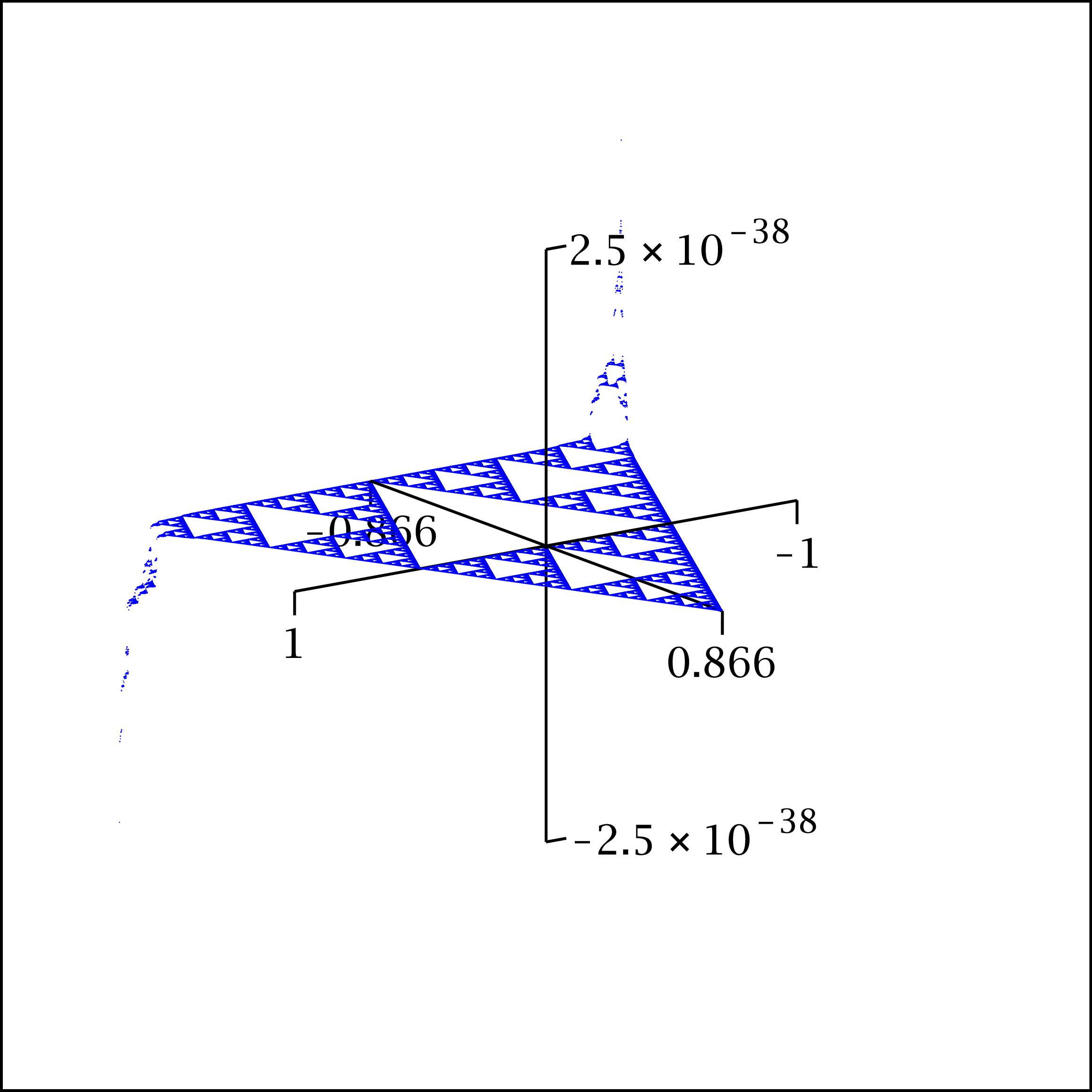}
    \caption{Graphs of  $P_{10,1}^{(r)}$, $P_{10,2}^{(r)} \ , P_{10,3}^{(r)}$ (top to bottom) with $r=\frac1{10}$, $r=1$, $r=10$ (left to right).}
\end{figure}

\section{Conjectures and Further Questions}

The aim of investigating the decay rates of the boundary values of monomials is to understand for which sequences of coefficients $\{c_{j,k}\}$ the sum 
\begin{equation} \sum_{k=1}^3 \sum_{ j=0}^\infty c_{j,k} P_{j,k} ^{(r)} \label{taylor series} \end{equation} 
defines a convergent Taylor series. The data presented in section 4 suggest that for most values of $r$, the coefficients $c_{j,k}$ must satisfy a condition restricting their growth rate to be slower than some exponential. If the only requirement is that \eqref{taylor series} converges, then it seems that the growth of $c_{j,3}$ can be allowed to be faster than that of $c_{j,1}$ and $c_{j,2}$, but if the additional requirement that the series \eqref{taylor series} can be rearranged to a convergent Taylor series about $q_1$ or $q_2$ is enforced, then $c_{j,k}$ must satisfy a growth condition independent of $k$, as in \cite{nsty}. These statements are implied by the following conjectures.

\begin{conjecture} For $r \neq 1, \frac{ \sqrt{17}-3}{4}$, 
\begin{equation} \lim_{j \to \infty } \frac{ \alpha_{j+1} (r) }{ \alpha_j(r)} = \frac{-1} { 2 \lambda_3(r) - \lambda_2(r) }. \end{equation} 
\end{conjecture}

\begin{conjecture} For $r \neq \frac{ \sqrt{17}-3}{4} $, \begin{equation} \lim_{j \to \infty } \frac{ \beta_{j+1} (r) }{ \beta_j(r)} = \frac{-1} { 2 \lambda_3(r) - \lambda_2(r) }. \end{equation} 
\end{conjecture} 
The data seem to show this convergence, and given what  is known about the convergence of $\frac{ \beta_{j+1}(1) } { \beta_j(1)}$, this seems natural to expect. However, for $j$ as large as 49, there are still visible differences between the graphs of $ \frac{\alpha_{j+1}(r)} { \alpha_j(r) }$ and $\frac{-1} { 2 \lambda_3(r) - \lambda_2(r) }$ and between  $\frac{ \beta_{j+1} (r) }{ \beta_j(r)}$ and $\frac{-1} { 2 \lambda_3(r) - \lambda_2(r) }$. This suggests the possibility that the relationship between the ratios of consecutive values of $\alpha_j(r)$ and between consecutive values of $\beta_j(r)$ with the Neumann eigenvalues of $\Delta_r$ may be more complicated and that the relationship in the above conjectures is only approximate. 

\begin{conjecture} For all but at most finitely many values of $r$, 
\begin{equation} \lim_{j \to \infty } \frac{ \gamma_{j+1} (r) }{ \gamma_j(r) } \ \text{exists}.  \label{gamma limit} \end{equation}  Moreover, there exists a continuous function, $g(r)$ with $g(r) <0$ for all $r \in (0, \infty) $ such that when \eqref{gamma limit} exists, 
\begin{equation} \lim_{j \to \infty } \frac{ \gamma_{j+1} (r) }{ \gamma_j(r) } =g(r).\end{equation} 
\end{conjecture} 

Perhaps there is a relationship between $\frac{ \gamma_{j+1}(r)}{ \gamma_j(r)} $ and the spectrum of $\Delta_r$. If such a relationship exists, it seems that it would involve eigenvalues higher up in the spectrum of $\Delta_r$, since $\frac{ \gamma_{j+1}(r)}{ \gamma_j(r)}$ apparently converges to smaller values than the ratios of the $\alpha_j(r)$ and $\beta_j(r)$ do. 

\begin{conjecture} For $r \neq 1, \frac{ \sqrt{17}-3}{4} $, 
\begin{equation} 
\| P_{j,1} ^{(r)} \| _ \infty =O ( |2\lambda_3 (r) - \lambda_2(r)|^{-j} ). \end{equation} 
For $r \neq \frac{\sqrt{17}-3}{4}$
\begin{equation} 
\| P_{j,2} ^{(r)} \| _ \infty = O (  |2\lambda_3 (r) - \lambda_2(r)|^{-j} ). \end{equation} 
For all but at most finitely many values of $r$, 
\begin{equation} 
\| P_{j,3}^{(r)} \| _ \infty = O (| g(r)|^{j} ) . \end{equation} 
\end{conjecture}
There is some numerical evidence suggesting that $| P_{j,1}^{(r)}| $ and $|P_{j,2}^{(r)} | $ attain the values 
$\| P_{j,1}^{(r)}\|_\infty  $ and $\|P_{j,2}^{(r)} \|_ \infty  $, respectively, at either $q_1 $ and $q_2$ or $F_{12}(q_2)$ depending on the values of $j$ and $r$, and that the limits as $j \to \infty$ of the ratios $\frac{ P_{j,1}^{(r) } ( F_{12} q_2) }{ \alpha_j(r) } $ and $\frac{ P_{j,2}^{(r) } ( F_{12} q_2 ) }{ \beta_j(r) } $ exist except for exceptional values of $r$, which would imply conjecture 5.4. In particular, there is explicit evidence that conjecture 2.10 of \cite{nsty} does not hold when $r$ is not equal to 1. It is not clear what the behavior at the exceptional points (other than $r=1$, of course) is like. For small values of $j$, $P_{j,3}^{(r)}$ attains its maximum at $q_1$, but it seems reasonable to expect that once the $\gamma_j(r)$ settle into their complicated pattern, beyond $j=30$, the behavior of $P_{j,3}^{(r)} $ becomes more complicated as well. Particularly, since $\gamma_j(r)$ has roots for $j$ large enough, this implies that the maximum value of $P_{j,3}^{(r)}$ cannot be attained at the boundary, and the skew-symmetry excludes the possibility of the maximum being attained at $F_{12}q_2$. \par

\section{Appendix: Neumann Spectrum of $\Delta_r$ }

The focus of \cite{fkls} is the spectrum of the family of self-similar symmetric Laplacians and the fact that these Laplacians exhibit spectral decimation. This spectral decimation means the spectrum of the Laplacian can be computed via a process that extends an eigenfunction of $\Delta_r^{(m)}$, which is defined on $V_m$, to a function defined on $V_{m+1}$ which is an eigenfunction of $\Delta_r^{(m+1)}$ (\cite{fs}). 
The eigenvalues corresponding to eigenfunctions which satisfy Dirichlet boundary conditions are computed in \cite{fkls}. Here, with the hope of generalizing the fact from \cite{nsty} that $\frac{ \beta_j(1) }{ \beta_{j+1} (1) } \to - \lambda_2$ as $j \to \infty$, where $\lambda_2$ is the second nonzero Neumann eigenvalue of the standard Laplacian $\Delta_1$, the eigenvalues of $\Delta_r$ corresponding to eigenfunctions which satisfy Neumann boundary conditions ($\partial_{n,r} u( q_i) =0$) are of interest. \par
If $u$ is assumed to satisfy $\Delta_r u  = \lambda u$, with $u(q_0) =x_0, u(q_1) = x_1, u(q_2) =x_2$ given, then equations (7.5), (7.6), and (7.7) of \cite{fkls} give an algorithm for extending $u$ as an eigenfunction of $\Delta_r$. Thus the problem of determining the Neumann spectrum of $\Delta_r$ is reduced to the problem of determining which values of $x_0, x_1, x_2$ yield extensions to Neumann eigenfunctions. \par 
As is demonstrated in \cite{str}, the condition that $u$ is a Neumann eigenfunction can be restated as the condition that $u$ extends evenly across the boundary of $SG$ to an eigenfunction. By considering $u$ as a function on $SG$ with an extra cell of level 1 attached at each boundary point and defining $u$ on each of these extra cells to be the reflection across the boundary of the restriction of $u$ to the adjacent cell, the eigenvalue equation can be applied not only at junction points, but at the boundary of $SG$ as well to obtain a Neumann eigenfunction. \par
Labeling the values of $u$ at each point in $V_1$, and labeling the values of $u$ at the level 1 points of the extra cells as in figure 6.1, and writing the equation 
\begin{equation} 
L(r) \Delta_r^{(1)} u(x) = \lambda_1 u(x) 
\end{equation} 
in terms of these values for each $x \in V_1$ gives a system of 15 equations in 16 variables, but when the left hand side of this system of equations is represented as a matrix, it is clear that the permissible values of $\lambda_1$ are the eigenvalues of this matrix, and the boundary values of Neumann eigenfunctions come from the corresponding eigenvectors of this matrix. \par
 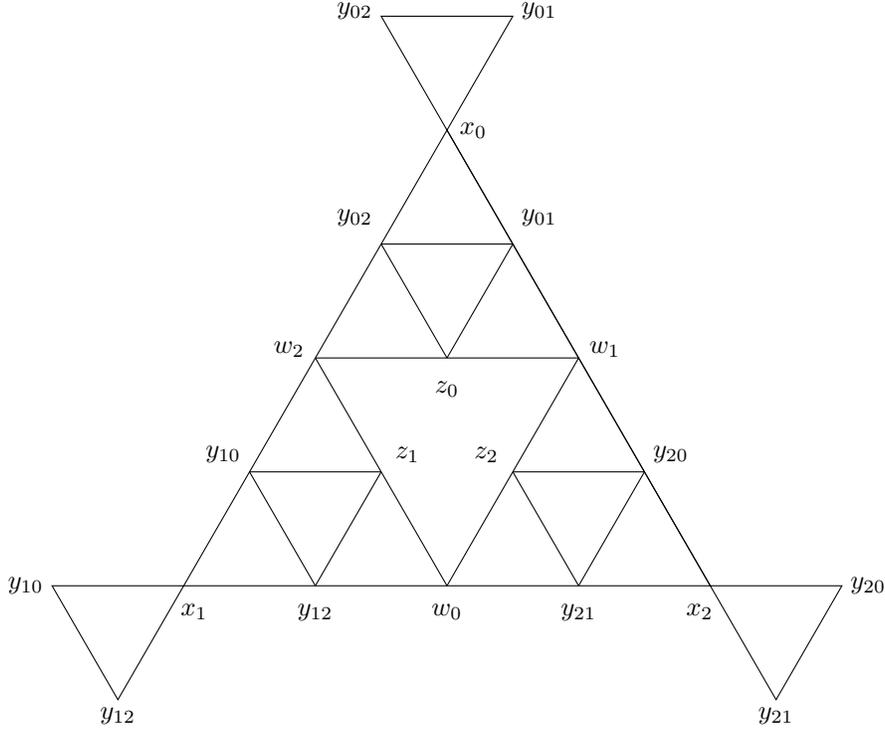
\begin{figure} 
 \begin{center}
 \begin{tikzpicture}[scale=7]
\draw[ -] (-.5,0) -- (.5,0);
\draw[-](-.5,0) -- (-.75,0)--(-.625,-.21651)--(-.5,0);
\draw[-](.5,0) -- (.75,0)--(.625,-.21651)--(.5,0);
\draw[-](0,.86603)--(.125, 1.08254)--(-.125,1.08254)--(0,.86603);
\draw[ -](-.5,0) --(0,.86603);
\draw[ -](.5,0) --(0,.86603);
\draw[ -](.5,0) --(0,.86603);
\draw[ -](0,0) --(.25, .43302);
\draw[ -](0,0) --(-.25, .43302);
\draw[ -](.25, .43302) --(-.25, .43302);
\draw[ -](0, .43302) --(-.125, .64952);
\draw[ -](0, .43302) --(.125, .64952);
\draw[ -](.125, .64952) --(-.125, .64952);
\draw[ -](-.25, 0)-- (-.125,.21651);
\draw[ -](-.25, 0)-- (-.375,.21651);
\draw[ -](-.375,.21651)-- (-.125,.21651);
\draw[ -](.25, 0)-- (.125,.21651);
\draw[ -](.25, 0)-- (.375,.21651);
\draw[ -](.375,.21651)-- (.125,.21651);
\filldraw (.48,.-.05) circle (0pt) node {$x_2$};
\filldraw (-.48,.-.05) circle (0pt) node { $x_1$};
\filldraw (0.05,.86603) circle (0pt) node { $x_0$};
\filldraw (.175,.69952) circle (0pt) node {$y_{01}$};
\filldraw (-.175,.69952) circle (0pt) node {$y_{02}$};
\filldraw (.175,1.09254) circle (0pt) node {$y_{01}$};
\filldraw (-.175,1.09254) circle (0pt) node {$y_{02}$};
\filldraw (.3,.45) circle (0pt) node {$w_1$};
\filldraw (-.3,.45) circle (0pt) node {$w_2$};
\filldraw (0,.375) circle (0pt) node {$z_0$};
\filldraw (.425,.25) circle (0pt) node {$y_{20}$};
\filldraw (-.425,.25) circle (0pt) node {$y_{10}$};
\filldraw (.8,0) circle (0pt) node {$y_{20}$};
\filldraw (-.8,.0) circle (0pt) node {$y_{10}$};
\filldraw (.075,.25) circle (0pt) node {$z_2$};
\filldraw (-.075,.25) circle (0pt) node {$z_1$};
\filldraw (.25,.-.05) circle (0pt) node {$y_{21}$};
\filldraw (-.25,.-.05) circle (0pt) node { $y_{12}$};
\filldraw (.625,-.24651) circle (0pt) node { $y_{21}$};
\filldraw (-.625,-.24651) circle (0pt) node {$y_{12}$};
\filldraw (0,.-.05) circle (0pt) node {$w_0$};
\end{tikzpicture}
\caption{The extension of $u$ to $V_1$ and its even extension across the boundary.}
\end{center}
\end{figure}
These eigenvalues and their respective multiplicities, as well as graphs of the eigenvalues as functions of $r$ are given in figure 6.2. 

\begin{figure}
\begin{center}
\begin{subfigure}{0.4 \textwidth}
\begin{center}
 \begin{tabular}{|c|c|}
\hline
 Eigenvalue & Multiplicity \\ \hline
 $0$ & $1$ \\ \hline
 $\frac{3(2r+1)}{r+1}$ & $1$ \\ \hline   $\frac{3(2r+3)}{r+1}$ & $1$ \\ \hline   $\frac{15r+15-3\sqrt[]{9r^2+18r+15}}{4(r+1)}$ & $2$ \\ \hline   $\frac{15r+15+3\sqrt[]{9r^2+18r+15}}{4(r+1)}$ & $2$  \\ \hline    $\frac{9}{2}$ & $2$ \\ \hline $9$ & $6$ \\ 
 \hline
\end{tabular}
\end{center}
\end{subfigure}
\begin{subfigure}{0.55 \textwidth}
\begin{center}
\includegraphics[width=50mm]{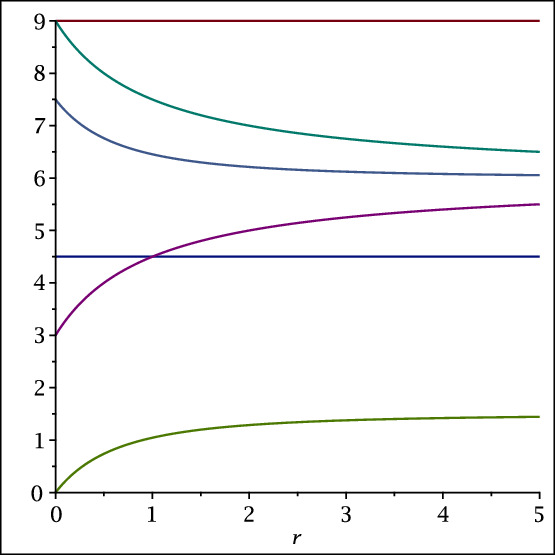} 
\end{center}
\end{subfigure}
\end{center}
\caption{The eigenvalues of the level 1 graph Laplacian as functions of $r$.}
\end{figure}

It is interesting to note that the second smallest nonzero Neumann eigenvalue of $\Delta_r$ changes when $r=1$, so it is reasonable to expect that for $r \neq 1$, both the second and third smallest nonzero Neumann eigenvalues of the Laplacian are related to the decay rate of $\alpha_j(r) $ and $\beta_j(r)$. This means that the eigenvalues of $\Delta_1$ occur with different multiplicities than those of $\Delta_r$ for $r\neq 1$, demonstrating further that the standard Laplacian is exceptional. \par 
From the eigenvalues of the level 1 graph Laplacian, the eigenvalues of the continuous Laplacian are computed as 
\begin{equation} \lambda= \lim_{m \to \infty } L(r)^{-m} \lambda_m
\end{equation} 
where 
\begin{equation} \label{spectraldec} \lambda_m(\lambda_{m+1} , r ) =- \frac{ 1}{27r ((r+1) \lambda_{m+1} -3r-6)} \bigg ( 2 \lambda_{m+1} (( r+1) \lambda_{m+1} - 6r - 3 ) \bigg ( (r+1)^2 \lambda_{m+1}^3\end{equation} $$ - 15 (r+1)^2 \lambda_{m+1}^2 + \bigg ( \frac{ 243}{4}  r^2 + \frac{ 279}{2} r + \frac{ 279}{4} \bigg ) \lambda_{m+1} - \frac{ 243}{4} r^2 - \frac{351}{2} r -\frac{ 405}{4} \bigg ) \bigg ) .$$
Solving for $\lambda_{m+1} $ in terms of $\lambda_m$ involves inverting a quintic polynomial and must be done numerically. Depending on $r$ and $\lambda_m$, there may be several values of $\lambda_{m+1}$ which satisfy \eqref{spectraldec}, but to obtain the first few eigenvalues in the spectrum of $\Delta_r$, one must select the smallest value of $\lambda_{m+1}$ at each step in this process. The fact that $\lambda_{m+1}=0$ satisfies $\lambda_m( \lambda_{m+1} ,r) =0$ combined with the fact that  
\begin{equation} 
\frac{\partial \lambda_{m}( \lambda_{m+1},r) }{\partial \lambda_{m+1} } \bigg \vert_{\lambda_{m+1} =0} = \frac{ (2r+1)(9r^2+26r+15)}{2r(r+2)} >0
\end{equation}
implies that the first 6 nonzero eigenvalues of $\Delta_r$ appear in the same order as the eigenvalues of $\Delta_r^{(1)} $ shown in figure 6.2. A graph of the first 6 nonzero Neumann eigenvalues of $\Delta_r$ is shown in figure 6.3. 

\begin{figure}[t]
\begin{center}
\includegraphics[width=50mm]{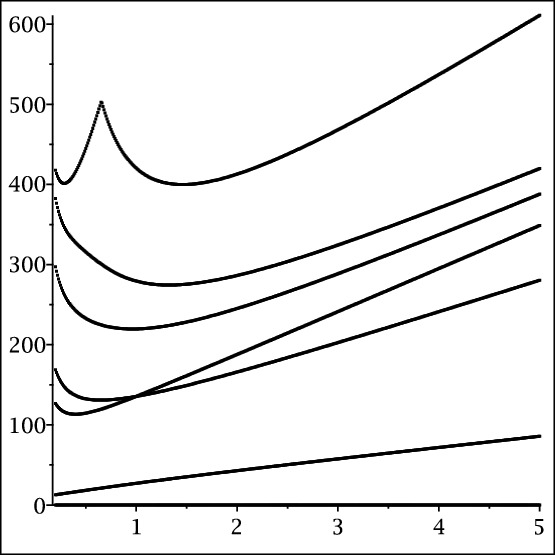} 

\end{center}
\caption{The first 6 nonzero Neumann eigenvalues of $\Delta_r$ as functions of $r$.}
\end{figure}

Other than the apparent relationship between the Neumann eigenvalues of the Laplacians and the decay rates of $\alpha_j(r)$ and $\beta_j(r)$, there are other surprising connections between the behavior of the monomials and the eigenvalues of $\Delta_r$. The function $\lambda_{m} ( \lambda_{m+1} , r)$ defined by \eqref{spectraldec}, when viewed as a function of only the variable $\lambda_{m+1}$ for fixed $r$, has a discontinuity when $\lambda_{m+1} = \frac{ 3(r+2) }{r+1}$. For most values of $r$, this discontinuity is an asymptote on the graph of $\lambda_{m} ( \lambda_{m+1} , r) $, but when $r=1$ and $r = \frac{ \sqrt{17}-3}{4}=0.2807764 $, this discontinuity turns out to be removable. This is notable because in section 4, it was noted that the functions $\alpha_j(r)$ have roots that appear to converge to $1$ and $0.2807764$ as $j$ increases, and the functions $\frac{ \beta_{j+1} (r)}{\beta_j(r) }$ have a local maximum that appears to converge to $0.2807764$ as $j$ increases. Moreover, it appears that the ratio 
$\frac{ \alpha_{j+1}(r) } {\alpha_j(r)} $ converges to a function which is continuous except at $r= \frac{ \sqrt{17}-3}{4}$ and $r=1$, and $\frac{ \beta_{j+1} (r)}{\beta_j(r) }$ appears to converge to a function which is continuous except at $r=\frac{ \sqrt{17}-3} 4$. So, in addition to the exceptional behavior of the standard Laplacian, $\frac{\sqrt{17} -3 }{4}$ seems to be another exceptional value of $r$.

\vspace{5mm}

\noindent Christian Loring, Department of Mathematics, Penn State University, McAllister Building, \\
University Park, State College, PA 16802, USA \\
email: \href{mailto: cnr5164@psu.edu} {\texttt{cnr5164@psu.edu}}\\ \\
W. Jacob Ogden, School of Mathematics, University of Minnesota, Vincent Hall, \\
Minneapolis, MN 55455, USA \\
email: \href{mailto: ogden048@umn.edu}{\texttt{ogden048@umn.edu} }\\\\
Ely Sandine, Department of Mathematics, Cornell University, Malott Hall, \\
Ithaca, NY, 14853, USA\\
email: \href{mailto: ebs95@cornell.edu}{\texttt{ebs95@cornell.edu}} \\\\
Robert S. Strichartz, Department of Mathematics, Cornell University, Malott Hall, \\
Ithaca, NY, 14853, USA\\
email: \href{mailto: str@math.cornell.edu} {\texttt{str@math.cornell.edu} }

\end{document}